\documentclass[11pt,a4paper]{article}
\usepackage{amsmath,amssymb,amsthm,amsfonts,latexsym,graphicx,subfigure,hyperref}
\usepackage[all,import]{xy}
\usepackage[numbers,sort&compress]{natbib}
\usepackage{indentfirst}
\usepackage{fancyhdr,enumerate}
\usepackage{bbm,color}
\usepackage{tikz,wasysym}
\usetikzlibrary{shapes.geometric, arrows}
\usepackage{accents,cases}
\usepackage{multirow}
\usepackage[lined, ruled, linesnumbered, longend]{algorithm2e}

\usepackage{array,float}
\newcommand{\PreserveBackslash}[1]{\let\temp=\\#1\let\\=\temp}
\newcolumntype{C}[1]{>{\PreserveBackslash\centering}p{#1}}
\newcolumntype{R}[1]{>{\PreserveBackslash\raggedleft}p{#1}}
\newcolumntype{L}[1]{>{\PreserveBackslash\raggedright}p{#1}}

\DeclareMathOperator*{\vol}{\ensuremath{vol}}

\makeatletter
\def\wbar{\accentset{{\cc@style\underline{\mskip8mu}}}}

\makeatother

\renewcommand{\vec}[1]{\mbox{\boldmath \small $#1$}}

\newcommand{\K}{\ensuremath{\mathcal{K}}}

\newcommand{\power}{\ensuremath{\mathcal{P}}}
\newcommand{\R}{\ensuremath{\mathbb{R}}}
\newcommand{\A}{\ensuremath{\mathcal{A}}}

\theoremstyle{plain}
\newtheorem{theorem}{Theorem}[section]

\newtheorem{defn}{Definition}[section]

\newtheorem{lemma}{Lemma}[section]
\newtheorem{remark}{Remark}[section]

\newtheorem{cor}{Corollary}[section]
\newtheorem{pro}{Proposition}[section]
\newtheorem{example}{Example}[section]
\newtheorem{Question}{Question}
\newtheorem{open}{Open Question}

\newtheorem{Claim}{Claim}[section]

\begin{document}
\allowdisplaybreaks

\title{
Homological eigenvalues of graph $p$-Laplacians
}
  \author{Dong Zhang\footnotemark[1]
  }
\footnotetext[1]{
LMAM and School of Mathematical Sciences, 
        Peking University,  
      100871 Beijing, China
\\
Email address: 
{\tt dongzhang@math.pku.edu.cn} 
(Dong Zhang).
}
\date{}\maketitle

\begin{abstract}
Inspired by persistent homology in topological data analysis,  we introduce the homological  eigenvalues of the graph $p$-Laplacian $\Delta_p$, which  allows us    to  analyse  and classify    non-variational eigenvalues. We  show the stability of homological  eigenvalues, and we prove that  for any homological  eigenvalue $\lambda(\Delta_p)$, the function  $p\mapsto p(2\lambda(\Delta_p))^{\frac1p}$ is locally increasing,  while the function  $p\mapsto 2^{-p}\lambda(\Delta_p)$ is  locally  decreasing. 
As a special class of homological eigenvalues,  the min-max eigenvalues  $\lambda_1(\Delta_p)$, $\cdots$, $\lambda_k(\Delta_p)$, $\cdots$,  are locally Lipschitz continuous  with respect to $p\in[1,+\infty)$. We also establish the   monotonicity of $p(2\lambda_k(\Delta_p))^{\frac1p}$ and $2^{-p}\lambda_k(\Delta_p)$ with respect to $p\in[1,+\infty)$.  

These  results  systematically establish a  refined  analysis 
of  $\Delta_p$-eigenvalues for varying $p$,  which  lead to several applications, including: (1)   settle an open problem by Amghibech  on  the monotonicity of some function  involving eigenvalues of $p$-Laplacian with respect to $p$; 
(2)  resolve a question asking whether the third eigenvalue of graph $p$-Laplacian is of  min-max form;  (3) 
refine  
the higher order Cheeger inequalities for  graph $p$-Laplacians by Tudisco and Hein, and extend 
the  multi-way Cheeger inequality by Lee, Oveis Gharan and Trevisan  to the   $p$-Laplacian case.

Furthermore, for the 1-Laplacian case, we characterize the  homological  eigenvalues  and min-max  eigenvalues from the perspective of  topological combinatorics, where our  idea is similar to the authors' work  on   discrete Morse theory. 

\vspace{0.2cm}

\end{abstract}
\tableofcontents
\section{Introduction}

As a discrete version of $p$-Laplacian, the graph $p$-Laplacian has been successfully used in various 
applications, including data and image processing problems 
and spectral clustering.  
Furthermore,  eigenvalue problems involving the graph $p$-Laplacian are also  important in  machine learning,  especially in the fields of   semi-supervised learning and unsupervised learning. 
Plenty of recent works  indicate  
that the graph $p$-Laplacian may enhance the performance of classical algorithms based on the standard  graph Laplacian \cite{HeinBuhler2010,HeinBuhler2009}. This has contributed 
to several  progresses   on both the theoretical and the numerical aspects of  $p$-Laplacians on graphs and networks   \cite{Chang,CSZ17,TudiscoHein18,UJT21,ZS06,ZHS06}. 
A remarkable development is that the second eigenvalue has a mountain-pass characterization and thus it is a min-max eigenvalue, and more importantly, the second eigenvalue satisfies the Cheeger inequality \cite{HeinBuhler2010,HeinBuhler2009,TudiscoHein18}.


Another
important result says  that the second and the largest eigenvalues satisfy  certain  monotonicity. Precisely, $p(2\lambda_p)^{\frac1p}$ and $p(2\gamma_p)^{\frac1p}$ are increasing with respect to $p$, where $\lambda_p$ and $\gamma_p$ represent  the second and the largest eigenvalues of $p$-Laplacians, respectively. But, it is unknown whether the other eigenvalues satisfy this  property. So, Amghibech asked in his original paper \cite{Amghibech}: 

\begin{Question}
\label{ques:1}Is there a version of the monotonicity for other eigenvalues for general graphs?
\end{Question}

To the best of the author's knowledge, this question has not been answered. In fact, this problem is quite difficult because  the eigenvalues of $p$-Laplacian are full of mysteries --- we even don't know the number of the  eigenvalues of   $p$-Laplacian. 

In the spectra of $p$-Laplacians, the variational eigenvalues have attracted particular  attention  because they possess     Rayleigh-type  quotient  reformulation,   good  nodal domain properties,  and  multi-way  Cheeger inequalities \cite{Chang,CSZ17,TudiscoHein18,UJT21,DFT21}.   
Here are some  important  questions on  the min-max (variational)  eigenvalues  of  
$p$-Laplacian: 
\begin{Question}\label{ques:2}
Can we find an eigenvalue of $p$-Laplacian that is not in the
list of min-max (variational)  eigenvalues?
\end{Question}

\begin{Question}
\label{ques:3}Is there an eigenvalue of $p$-Laplacian  larger than the second eigenvalue but smaller than the third min-max eigenvalue?
\end{Question}
 \begin{Question}\label{ques:4}
Is there a graph such that for both $p=1$ and some $p>1$, its $p$-Laplacian has a non-variational eigenvalue? 
 \end{Question}
In the setting of    $p$-Laplacian on Euclidean domains, Question \ref{ques:2} is a major long-standing open problem on the higher order eigenvalues of 
$p$-Laplacian \cite{DM21,DLK13,FMZ19}.  
The only known   case is that  the domain is of one-dimension, i.e., an interval (see e.g., \cite{Drabek92}).


While, in the setting of discrete $p$-Laplacian on  connected graphs with $p\not\in\{1,2\}$,  the only known cases\footnote{Deidda, Putti and Tudisco \cite{DFT21} also  construct a graph of order $4$ which possesses non-variational $p$-Laplacian  eigenvalues for $p\not\in\{1,2\}$.  Their new example shows the smallest simple graph whose unnormalized $p$-Laplacian has a  non-variational eigenvalue. 
If we work on generalized graphs that have repeated edges with  different  real incidence coefficients, then we can give a graph with 2 vertices and 2 edges, whose normalized  $p$-Laplacian has a  non-variational eigenvalue (see Remark \ref{rem:general-graph}). 
}  
are tree graphs and complete graphs 
(see \cite{Amghibech,DFT21}).  Precisely, Amghibech observed that there are more
eigenvalues than the number of vertices 
 for $p$-Laplacian on  a complete graph \cite{Amghibech}. That is a positive answer to Question \ref{ques:2} on  complete graphs with $p\not\in\{1,2\}$.  However, it is surprising that  the variational spectrum of   1-Laplacian on complete graphs is the entire spectrum (see  
 Section \ref{sec:complete-graph}).  
Moreover, interestingly,  
based on the  advanced  nodal domain theory \cite{DFT21}, a recent breakthrough by   Deidda, Putti and Tudisco states that all the eigenvalues of $p$-Laplacian on a tree are variational eigenvalues \cite{DFT21}, that is, Question \ref{ques:2} has a negative answer on forests. 

Nevertheless,  we do not yet know the answer to Question \ref{ques:2} on   connected graphs other than trees and complete graphs. We also 
don't know if Question \ref{ques:2} has a positive answer for both $p=1$ and some $p>1$ on certain  graphs. In addition, whether the third eigenvalue is in the sequence of min-max eigenvalues is still waiting to be explored.  From these perspectives,  Questions \ref{ques:3} and \ref{ques:4} are very  natural, and they indeed first appeared in the field of  $p$-Laplacians on Euclidean domains \cite{DLK13,Drabek07,PereraAgarwalO'Regan}.   

The difficulty to study the above  questions as well as  other relevant aspects on  $p$-Laplacian  eigenvalue problem  is twofold: 
on one hand, 
we cannot compute all the eigenvalues of  $p$-Laplacian for general graphs whenever  $p\not\in\{1,2\}$; 
on the other hand, the feature and structure of the spectra of $p$-Laplacians are less  known \cite{Amghibech,HeinBuhler2009,HeinBuhler2010,TudiscoHein18,Drabek07}. In other words,  the structure  of the eigenspaces of  $p$-Laplacian is unclear.




In this paper, we 
answer  Questions \ref{ques:1}, \ref{ques:2}, \ref{ques:3} and \ref{ques:4}  
by introducing the homological eigenvalues of graph $p$-Laplacian. More specifically, we overcome the difficulty by offering the following new contributions:   
(1) establishing the upper  semi-continuity of the  spectra of $p$-Laplacians, and the locally Lipschitz  continuity of the variational eigenvalues of $p$-Laplacians with respect to $p$; (2) introducing  homological eigenvalues for $p$-Laplacians and  showing certain local  monotonicity on those eigenvalues with respect to $p$; (3)  bringing the ideas from  persistent  homology theory and discrete Morse theory to graph  $1$-Laplacian, and characterizing its   homological  eigenvalues  and variational  eigenvalues from the perspective of  topological combinatorics. 

To state our result precisely, we first present the eigenvalue problem of graph  $p$-Laplacian. 
In this paper, we are working on  a finite simple graph $G=(V,E)$ with the vertex set $V=\{1,\cdots,n\}$ and the edge set $E$. For $p>1$,  the $p$-Laplacian $\Delta_p:\R^n\to \R^n$ is defined by
$$(\Delta_p\vec x)_i=\sum_{j\in V:\{j,i\}\in E}|x_i-x_j|^{p-2}(x_i-x_j),\;\;\forall i\in V,\;\forall \vec x=(x_1,\cdots,x_n)\in\R^n.$$
Following the definition in \cite{TudiscoHein18},  the  (normalized)  
eigenvalue problem for $\Delta_p$ 
is to find $\lambda\in\R$ and $\vec x\ne\vec0$ such that\footnote{We can also use the normalized $p$-Laplacian  $\hat{\Delta}_p=\mathrm{diag}(1/\deg(1),\cdots,1/\deg(n))\Delta_p$ instead of $\Delta_p$, and write the eigenvalue problem as $\hat{\Delta}_p \vec x=\lambda (|x_1|^{p-2}x_1,\cdots,|x_n|^{p-2}x_n)$.} $$(\Delta_p\vec x)_i=\lambda\deg(i)|x_i|^{p-2}x_i,\;\;\forall i\in V,$$
where $\deg(i)$ indicates the number of edges that are incident to $i$. For the case of $p=1$, we refer to Definition \ref{def:weighted-graph1-Lap}.

The Lusternik-Schnirelman theory
allows 
to define  a sequence of  
eigenvalues   of the $p$-Laplacian: 
\begin{equation}\label{eq:min-max-Yang}
\lambda_k(\Delta_p):= \inf_{\gamma(S)\ge k}\sup\limits_{\vec x\in S}F_p(\vec x),\;\;k=1,\cdots,n,    
\end{equation}
where$$ F_p(\vec x):=\frac{\sum_{\{j,i\}\in E}|x_i-x_j|^{p}}{\sum_{i\in V}\deg(i)|x_i|^p}\text{ for }\vec x\ne\vec 0,$$
and $\gamma(S)$ represents the 
Yang index  
of a centrally symmetric compact subset $S$ in  $\mathbb{R}^n\setminus\{\vec0\}$ (see Definition \ref{def:Yang}). 
If we use the Krasnoselskii genus $\gamma^-$ instead of the Yang index $\gamma$ in \eqref{eq:min-max-Yang}, then we actually  define the so-called variational eigenvalues $\lambda_1^-(\Delta_p),\cdots,\lambda_n^-(\Delta_p)$.
Similarly, if we replace  the Yang index $\gamma$ in \eqref{eq:min-max-Yang} by the  index-like quantity $\gamma^+$ probably first studied by  Conner and Floyd  \cite{CF60}, then we define a sequence of  
 min-max 
eigenvalues $\lambda_1^+(\Delta_p),\cdots,\lambda_n^+(\Delta_p)$, which was first introduced for Dirichlet $p$-Laplace eigenproblem by Dr\'abek and Robinson \cite{DR99}. 

Since  $\lambda_k(\Delta_p)$, $\lambda_k^-(\Delta_p)$ and $\lambda_k^+(\Delta_p)$ are all in min-max form, we call them min-max eigenvalues. 
These three sequences of min-max eigenvalues were originally defined for the continuous p-Laplacian (see Section 0.7 in \cite{PereraAgarwalO'Regan}). 
In this paper, we would mainly  focus on $\{\lambda_k(\Delta_p)\}_{k=1}^n$ and  $\{\lambda_k^-(\Delta_p)\}_{k=1}^n$, which are often referred to as the 
`the min-max eigenvalues'
and `the variational eigenvalues', respectively. 
These eigenvalues satisfy the relations \cite{Amghibech,TudiscoHein18,JZ21-PM,HeinBuhler2010,CSZ17,JZ21}:
$$\{\text{min-max }\Delta_p\text{-eigenvalues}\}\subset\{\text{critical values of }F_p\}=\{\Delta_p\text{-eigenvalues}\}\subset[0,2^{p-1}]$$
for $p>1$, and 
$$\{0,h,1\}\subset\{\text{min-max }\Delta_1\text{-eigenvalues}\}\subset\{\text{critical values of }F_1\}{\bf\subset }\, \{\Delta_1\text{-eigenvalues}\}\subset[0,1],$$
where $h$ indicates the usual Cheeger constant of the graph.  
 It is  worth to note that the critical points of $F_1$ are to be understood in the sense of Clarke subdifferential calculus. 

We 
refine the above  inclusion relations by employing the homological critical values  \cite{BS14,SEH07,Govc16}: 
\begin{align*}
0\in\{\lambda_1(\Delta_p),\cdots,\lambda_n(\Delta_p)\}&\subset \{\textbf{homological critical values}\text{ of }F_p\}
\\&\subset\{\text{critical values of }F_p\}\subset\{\text{eigenvalues of }\Delta_p\}    \subset[0,2^{p-1}]
\end{align*}
where  the  penultimate inclusion  `$\{\text{critical values of }F_p\}\subset\{\text{eigenvalues of }\Delta_p\}  $' is indeed an equality   if $p>1$, while for $p=1$, all the sets above are mutually different in general. 
For convenience, we  call such  homological critical values of $F_p$ the \textbf{homological  eigenvalues} of $\Delta_p$, and we refer to Section \ref{sec:homological-eigen} for the  
definitions. 

Our most significant new contribution is 
the following theorem on 
 monotonicity: 
\begin{theorem}\label{thm:most}
For any $p\ge1$ and any isolated homological  eigenvalue (resp., any eigenvalue  produced by   homotopical linking\footnote{See  Definition \ref{def:homotopy-eigen-p} for the definition, and Section \ref{sec:homological-eigen} for some basic  properties.}) $\lambda(\Delta_p)$ of $\Delta_p$, and any  $\epsilon>0$, there exists $\delta>0$ such that for any $q\in(p-\delta,p+\delta)$, there 
exists $\Delta_q$-eigenvalue  $\lambda(\Delta_q)\in (\lambda(\Delta_p)-\epsilon,\lambda(\Delta_p)+\epsilon)$, such that  $\lambda(\Delta_q)$ is a  homological eigenvalue (resp., an  eigenvalue  produced by   homotopical linking) with the property: 
$$\begin{cases}
p(2\lambda(\Delta_p))^{\frac1p}\le q(2\lambda(\Delta_q))^{\frac1q}\text{ and }2^{-p}\lambda(\Delta_p)\ge 2^{-q}\lambda(\Delta_q),&\text{ if }q\in(p,p+\delta),\\
p(2\lambda(\Delta_p))^{\frac1p}\ge q(2\lambda(\Delta_q))^{\frac1q}\text{ and }2^{-p}\lambda(\Delta_p)\le 2^{-q}\lambda(\Delta_q),&\text{ if }q\in(p-\delta,p).\end{cases}$$
For any $k$, the min-max eigenvalue $\lambda_k(\Delta_p)$ is locally Lipschitz continuous with respect to $p$, and moreover, we have the following monotonicity of certain 
functions  involving min-max  
eigenvalues: 
\begin{itemize}
    \item the function  $p\mapsto p(2\lambda_k(\Delta_p))^{\frac1p}$ 
is increasing  on $[1,+\infty)$;
\item 
the function 
$p\mapsto2^{-p}\lambda_k(\Delta_p)$ is decreasing on $[1,+\infty)$.
\end{itemize}
All the claims above hold if we use $\lambda_k^-(\Delta_p)$ or  $\lambda_k^+(\Delta_p)$ instead of $\lambda_k(\Delta_p)$.
\end{theorem}

\begin{remark}
The monotonicity of the function 
$p\mapsto p(2\lambda_k(\Delta_p))^{\frac1p}$ also holds under the domain setting, but the monotonicity of the function 
$p\mapsto2^{-p}\lambda_k(\Delta_p)$ only holds for the graph setting. It is worth noting   that if $\lambda_k(\Delta_p)\ne0$ for some $p\ge 1$, then $p\mapsto p(2\lambda_k(\Delta_p))^{\frac1p}$ is strictly increasing  on $[1,+\infty)$, and $p\mapsto2^{-p}\lambda_k(\Delta_p)$ is strictly decreasing on $[1,+\infty)$. It is 
noteworthy that the local  monotonicity and stability for homological eigenvalues  
in Theorem \ref{thm:most} do not hold for non-homological eigenvalues (see Section  \ref{sec:homological-eigen} and Example \ref{example:path6}). The eigenvalue  produced by   homotopical linking is introduced in Section \ref{sec:homological-eigen}. 
\end{remark}

Theorem \ref{thm:most} establishes   asymptotic behaviors   
for homological   eigenvalues of $\Delta_p$ with respect to $p$. 
It 
not only answers a  question by Amghibech  \cite{Amghibech}, but also provides an elegant  solution to  Question \ref{ques:2} 
 about 
higher order eigenvalues:  we can find an eigenvalue that is not in the list of min-max eigenvalues, and in fact, the third one satisfies the requirement. 
Roughly speaking, we construct a graph and prove that  
there is a homological  nonvariational eigenvalue of $\Delta_p$ for some $p>1$. 

\begin{theorem}\label{thm:main}
Let $G=(V,E)$ be the simple graph on  $V=\{1,2,3,4,5,6\}$ shown as:
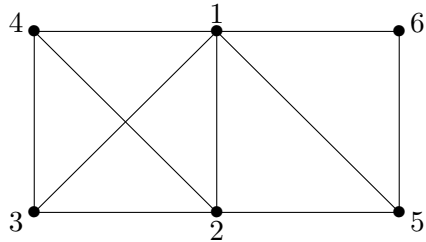
\begin{center}
\begin{tikzpicture}[scale=1.2]
\draw (0,0)--(0,2)--(2,2)--(2,0)--(0,2);
\draw (2,2)--(0,0)--(2,0);
\draw (4,0)--(4,2)--(2,2)--(4,0)--(2,0);
\node (1) at  (2,2) {$\bullet$};
\node (1) at  (2,2.2) {$1$};
\node (2) at  (2,0) {$\bullet$};
\node (2) at  (2,-0.2) {$2$};
\node (3) at  (0,0) {$\bullet$};
\node (3) at  (-0.2,-0.1) {$3$};
\node (4) at  (0,2) {$\bullet$};
\node (4) at  (-0.2,2.1) {$4$};
\node (6) at  (4,2) {$\bullet$};
\node (6) at  (4.2,2.1) {$6$};
\node (5) at  (4,0) {$\bullet$};
\node (5) at  (4.2,-0.1) {$5$};
\end{tikzpicture}
\end{center}

Then, for any $0<\epsilon<\frac{1}{10}$, there exists $0<\delta<1$ such that for any $p\in [1,1+\delta)$, there is a $p$-Laplacian eigenvalue $\lambda(\Delta_p)\in (\frac59-\epsilon,\frac59+\epsilon)$ which is not  in the lists of min-max eigenvalues $\{\lambda_k(\Delta_p)\}_{k=1}^n$,  $\{\lambda_k^-(\Delta_p)\}_{k=1}^n$ and $\{\lambda_k^+(\Delta_p)\}_{k=1}^n$.  
\end{theorem}

In order to prove   Theorem \ref{thm:main},  we just need to carefully select homological eigenvalues from nonvariational eigenvalues of $\Delta_1$, and then apply Theorem \ref{thm:most} to complete the verification.  It remains to construct a homological eigenvalue of $\Delta_1$ which is not a  variational eigenvalue. 
But it is very difficult to check whether a $\Delta_1$-eigenvalue is a homological eigenvalue.  
Fortunately, we establish the following characterization of homological eigenvalues of $1$-Laplacian.

For any subset $A\subset V$, we use $\vec1_A$ to denote the characteristic vector of $A$. Let $K_n$ be the   simplicial complex on the  vertex set  $\{-1,0,1\}^n\setminus\{\vec 0\}$ with $n!\cdot2^n$ maximal simplexes, where each maximal simplex is a $(n-1)$-dimensional simplex of the form $\mathrm{conv}\{\vec1_{A_i}-\vec1_{B_i}:i=1,\cdots,n\}$,  whenever $A_1\subset \cdots\subset A_n$ and $B_1\subset \cdots\subset B_n$  and $A_n\cap B_n=\varnothing$ and $A_i\cup B_i\ne A_{i+1}\cup B_{i+1}$, $i=1,\ldots,n-1$. 
Then, $K_n$ is a pure simplicial complex and its geometric  realization $|K_n|$ is actually the boundary of the  hypercube
$\{\vec x\in\R^n:\|\vec x\|_\infty\le1\}$ which is a PL 
manifold. Thus, we  regard $K_n$ as a triangulation of the unit $l^\infty$-sphere, and we also identify   $|K_n|$ with $\{\vec x\in\R^n:\|\vec x\|_\infty=1\}$. In addition, we can also see $K_n$ as the order complex on $\mathcal{P}_2(V)=\{(A,B):A\cap B=\varnothing,A\cup B\ne\varnothing,A,B\subset V\}$ (see Section \ref{sec:1-lap-simplicial-complex} for details).

\begin{theorem}\label{thm:tri-link} 
$\lambda\in\R$ is a  homological eigenvalue of $\Delta_1$ if and only if 
the subcomplex of $K_n$  induced by the vertices located in the open sublevel set $ \{\vec x\in|K_n|:F_1(\vec x)< \lambda\}$ and  the subcomplex of $K_n$  induced by the vertices located in the closed  sublevel set $ \{\vec x\in|K_n|:F_1(\vec x)\le \lambda\}$ have different homology groups.  Particularly, if 
there exists 
$A\ne\varnothing$  such that:
\begin{itemize}
\item[(1)]  $F_1(\vec 1_A)=\lambda$,  $F_1(\vec v)\ne\lambda$ for any vertex $\vec v$ located in $ \mathrm{link}(\vec 1_A)$, and 
\item[(2)]  the subcomplex of $K_n$  induced by the vertices located in the open sublevel set $ \{\vec x\in \mathrm{link}(\vec 1_A):F_1(\vec x)< \lambda\}$ is  homologically non-trivial.
\end{itemize}
 then $\lambda$ is a homological  eigenvalue of $\Delta_1$.
\end{theorem}

This result is the second ingredient of the proof of Theorem \ref{thm:main}, by which we  only need to check the vertices of $K_n$. 

More significantly, 
based on  
Theorem \ref{thm:most}, we obtain  Cheeger-type  inequalities which not only refine and strengthen  the higher-order Cheeger   inequality for graph  $p$-Laplacian by Tudisco and Hein  \cite{TudiscoHein18}, but also establish  the first nonlinear  multi-way Cheeger inequality for  $p$-Laplacian which generalizes the related  works   
by Miclo \cite{Miclo08}, Lee, Oveis Gharan and Trevisan \cite{LGT12}. Recall the multi-way Cheeger constant \cite{Miclo08,LGT12}: 
\begin{equation}\label{eq:k-way-Cheeger}
h_k:=\min_{\text{non-empty disjoint } A_1, \cdots, A_k} \max_{1\le i\le k}\frac{|\partial A_i|}{\vol(A_i)},\;k=1,2,\cdots,n,   
\end{equation}
where $|\partial A|$ is the number of edges  connecting  $A$ and $V\setminus A$, and  $\vol(A):=\sum_{i\in A}\deg(i)$.

We introduce the  modified combinatorial $k$-way  Cheeger constant (see Section \ref{sec:Cheeger-in} for details)
$$\hat{h}_k=\min\limits_{\A\in \mathcal{S}_k}\max\limits_{(A,B)\in\A}\frac{|\partial A|+|\partial B|}{\vol(A\cup B)}$$
where  $\mathcal{S}_k:=\{\A\subset \power_2(V):\text{the Yang index 
of the  subcomplex of }K_n\text{ induced by  }\A\text{ is at least }k\}$.  
We similarly define $\hat{h}_k^-$ and $\hat{h}_k^+$ by using 
$\mathcal{S}_k^-:=\{\A\subset \power_2(V):\text{the Krasnoselskii genus
of the subcomplex}$  
$\text{of }K_n\text{ induced by  }\A\text{ is at least }k\}$ and $\mathcal{S}_k^+:=\{\A\subset \power_2(V):\text{the }\gamma^+\text{-index
of the  subcomplex of }K_n$ $\text{induced by  }\A\text{ is at least }k\}$ instead of $\mathcal{S}_k$ in the definition of $\hat{h}_k$.  
 We would devote some words  to clarify how to go from a
family $\mathcal{A}\subset\mathcal{P}_2(V)$  to its corresponding subcomplex, which will be  explained at length in   Section \ref{sec:1-lap-simplicial-complex}. 
For convenience, we would use a partial order $\prec$ on $\power_2(V)$ defined as $(A,B)\prec (A',B')$ if $A\subset A'$ and $B\subset B'$. This induces the partial order $\prec$ on  a
subfamily $\mathcal{A}\subset\mathcal{P}_2(V)$  in the same way, from which we obtain a partial order set $(\A,\prec)$.  
The    subcomplex of $K_n$ induced by  $\A$  is defined as the order complex on  $\A$ with respect to the partial order for set-pairs, i.e.,  the faces of the corresponding subcomplex are the  chains (totally ordered subsets) of the partial order set  $(\A,\prec)$.

In general, $\hat{h}_k\le h_k$, and for $k=2$,  $\hat{h}_2$ always agrees with the usual Cheeger constant $h_2$. Then, we have:
\begin{theorem}\label{thm:p-lap-C}
For any $p\ge 1$, and $k=1,\cdots,n$,
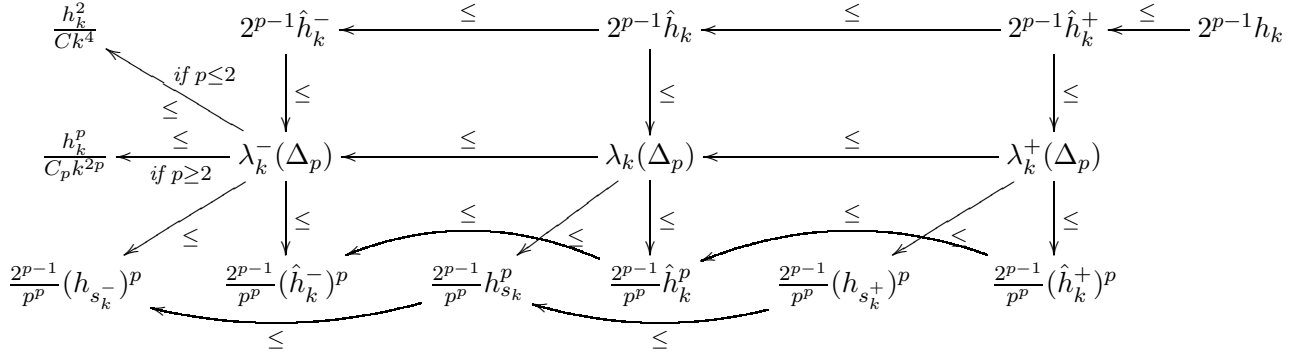
\[
\xymatrix{  \frac{h^2_k}{Ck^4} & 2^{p-1}\hat{h}_k^-\ar[d]^{\le} &  & 2^{p-1}\hat{h}_k\ar[d]^{\le}\ar[ll]_{\le} & & 2^{p-1}\hat{h}_k^+\ar[d]^{\le}\ar[ll]_{\le}  &2^{p-1}h_k\ar[l]_{\le} \\
\frac{h^p_k}{C_pk^{2p}}
& \lambda_k^-(\Delta_p)\ar[l]^{\text{if }p\ge2}_{\le}\ar[lu]_{\text{if }p\le2}^{\le}\ar[ld]^{\le}\ar[d]^{\le} &  & \lambda_k(\Delta_p)\ar[ld]^{\le}\ar[d]^{\le}\ar[ll]_{\le} &  & \lambda_k^+(\Delta_p)\ar[ld]^{\le}\ar[d]^{\le} \ar[ll]_{\le}&\\
 \frac{2^{p-1}}{p^p}(h_{s_k^-})^p & \frac{2^{p-1}}{p^p}(\hat{h}_k^-)^p & \frac{2^{p-1}}{p^p}h_{s_k}^p\ar@/^1.2pc/[ll]^{\le}   & \frac{2^{p-1}}{p^p}\hat{h}_k^p \ar@/_1.7pc/[ll]_{\le} &\frac{2^{p-1}}{p^p}(h_{s_k^+})^p\ar@/^1.2pc/[ll]^{\le}  & \frac{2^{p-1}}{p^p}(\hat{h}_k^+)^p\ar@/_1.7pc/[ll]_{\le} &}\]
 where $C>0$ and $C_p>0$ are universal constants, 
and  the quantity at the end of each arrow is greater than or equal to the quantity at the corresponding arrowhead, 
 i.e., $a\longleftarrow b$ means $a\le b$. In the above diagram,  $$s_k:=\max\limits_{\vec x\in \bigcup\limits_{q\in[1,p]} \mathsf{S}_{k}(\Delta_q)} \mathfrak{S}(\vec x)
,\;\;s_k^-:=\max\limits_{\vec x\in \bigcup\limits_{q\in[1,p]} \mathsf{S}_{k}^-(\Delta_q)} \mathfrak{S}(\vec x)\text{ and }s_k^+:=\max\limits_{\vec x\in \bigcup\limits_{q\in[1,p]} \mathsf{S}_{k}^+(\Delta_q)} \mathfrak{S}(\vec x),$$ where $\mathfrak{S}(\vec  x)$ is the number of strong nodal domains of $\vec x$,  

$\mathsf{S}_{k}(\Delta_q):=\left\{\vec x\ne\vec0\left|\,\vec x\text{ is an eigenvector corresponding to some eigenvalue }\lambda\text{ with }\lambda\le\lambda_k(\Delta_q)\right.\right\}$,  

$\mathsf{S}_{k}^\pm(\Delta_q):=\left\{\vec x\ne\vec0\left|\,\vec x\text{ is an eigenvector corresponding to some eigenvalue }\lambda\text{ with }\lambda\le\lambda_k^\pm(\Delta_q)\right.\right\}$.
\end{theorem}

\begin{remark}
By the above theorem, for any $k$,  the quantities on the left-hand-side and the right-hand-side of the higher order Cheeger-type inequality
\begin{equation}\label{eq:strong-Cheeg2}
\swarrow\frac{1}{2^{p-1}}\lambda_k^\pm(\Delta_p)\le \hat{h}_k^\pm\le \frac p2 (2\lambda_k^\pm(\Delta_p))^{\frac1p}\nearrow
\end{equation}
are decreasing and increasing with respect to $p\in[1,+\infty)$, respectively. 
\end{remark}

Taking $p=2$ and $k=2$ in Theorem  \ref{thm:p-lap-C}, we get the usual Cheeger inequality on graphs, while taking $p=1$ and $k=2$, we recover the identity  $\hat{h}_2^-=\lambda_2^-(\Delta_1)$ proved by Hein-B\"{u}hler \cite{HeinBuhler2010} and Chang \cite{Chang} independently. More importantly, taking  $p=1$ and $k\ge 1$, we indeed establish  the equality  $\hat{h}_k^\pm=\lambda_k^\pm(\Delta_1)$  which means that the $k$-th min-max eigenvalue of graph 1-Laplacian is actually the $k$-th  modified Cheeger constant (a combinatorial quantity). In particular, note that our
findings in combination with the recent work \cite{DFT21} show that $\hat{h}_k^-=\hat{h}_k=\hat{h}_k^+=h_k$ holds on a tree. In addition, Theorem \ref{thm:p-lap-C} 
and  \eqref{eq:strong-Cheeg2}  refine   Tudisco-Hein's higher order Cheeger inequality for $p$-Laplacian  \cite{TudiscoHein18}, and establish  the first $p$-Laplacian version of  Lee-Oveis Gharan-Trevisan's multi-way Cheeger inequality \cite{LGT12}. 
Theorem \ref{thm:p-lap-C} implies $\hat{h}_k\ge \frac{h_k^2}{Ck^4}$. We further conjecture that there exists a universal constant $C>0$ such that for any $k$,   $\hat{h}_k\ge \frac{h_k}{Ck^2}$. If such a conjecture has a positive answer, then by Theorem \ref{thm:most},  we can obtain a  
strict 
 refinement of the famous multi-way Cheeger inequality proved by Lee, Oveis Gharan and Trevisan \cite{LGT12}. 

 \begin{cor}\label{cor:=}
For any $k$, $\lambda^-_k(\Delta_2)=\lambda_k(\Delta_2)=\lambda^+_k(\Delta_2)$. 
For any  $p\ge 1$, 
$\lambda^-_1(\Delta_p)=\lambda_1(\Delta_p)=\lambda^+_1(\Delta_p)=0$, $\lambda^-_n(\Delta_p)=\lambda_n(\Delta_p)=\lambda^+_n(\Delta_p)=1$ and $\lambda^-_2(\Delta_p)=\lambda_2(\Delta_p)=\lambda^+_2(\Delta_p)$. 
Moreover, on forests, we further have 
\[\lambda_k^-(\Delta_1)=\lambda_k(\Delta_1)=\lambda_k^+(\Delta_1)=\hat{h}_k^-=\hat{h}_k=\hat{h}_k^+=h_k,\;\forall k=1,\cdots,n.\]
\end{cor}

Furthermore, Theorem \ref{thm:most} allows to bound some min-max  eigenvalues in a novel way: 
\begin{theorem}\label{thm:combina-estimate}
For any connected graph, there is a min-max eigenvalue of $\Delta_p$  in  
$$\begin{cases}
(2^{p-3},2^{p-1}(\frac{\sqrt{3}}{p})^p),&\text{ if }p< 2,\\
[\frac12,\frac32],&\text{ if }p=2,\\
(\frac{2^{p-1}}{p^p},3\times 2^{p-3}),&\text{ if }p> 2.
\end{cases}$$

For any  graph $G$, there are at least $\alpha_*(G)$ min-max eigenvalues of $\Delta_p$ larger than 
$\frac{2^{p-1}}{p^p}$ when $p>1$, and there are at least  $\alpha_*(G)$ min-max eigenvalues of $\Delta_1$ equal to $1$,  
where $\alpha_*(G)$ is the pseudo-independence  number introduced in \cite{Zhang18} (see also  Corollary \ref{cor:independence-number} in  Section \ref{sec:min-max-1-Lap}  for details). 
\end{theorem}

\begin{remark} The first statement in Theorem \ref{thm:combina-estimate} is nontrivial when $p>\sqrt{3}$, and  the case of $p=2$ is the main result established in the author's recent work \cite{JMZ21}.\end{remark}

\begin{remark}
The inequality 
$\lambda_{n-\alpha_*(G)+1}(\Delta_p)>\frac{2^{p-1}}{p^p}$ for $p>1$ in Theorem \ref{thm:combina-estimate} can be compared with the inertia bound for $p$-Laplacian \cite{JZ21-PM,Cvetkovic71}, i.e.,  $\lambda_{n-\alpha(G)+1}(\Delta_p)\ge1$, where $\alpha(G)$ is the standard  independence number of $G$. In some cases, it is better than the inertia bound. For example, considering a triangle graph whose Laplacian eigenvalues are $0,\frac32,\frac32$, by the inertia  bound for $p=2$, there is an  eigenvalue larger than or equal to $1$ which is not optimal, while by Theorem  \ref{thm:combina-estimate}, there are two eigenvalues larger than 
$\frac12$ 
which is sharp. 
\end{remark}


Other  new  results that we would like to  highlight in this paper are:
\begin{itemize}
\item We show that the  graph 1-Laplacian is zonotope-valued, which reveals that  
the graph 1-Laplacian is closely related to combinatorial geometry. See Section \ref{sec:usc-p-lap} for details. 
\item We derive that for any graph of order larger than or equal to $4$, there exists one number $c\in (\lambda_2,\lambda_n)$ which is not an eigenvalue of $\Delta_p$ for any $p$ sufficiently close to 1. And for particular graphs like trees  and complete graphs  (see Section \ref{sec:complete-graph}), for any $p\ge 1$, there always exists a number in $(\lambda_2,\lambda_n)$  that is not
in the list of the $p$-Laplacian eigenvalues.  This partially answers 
the question: is there a number between the second and the largest $\Delta_p$-eigenvalues that is not an eigenvalue? For the  
$p$-Laplacian on a domain, the question asks whether there exists one number $c>\lambda_2$ that is not an eigenvalue \cite{DLK13}. Our method  is expected to solve this  open problem in the domain setting.
\item We obtain a characterization for the homological eigenvalues of $\Delta_1$ (see Theorem \ref{thm:tri-link}), which is  inspired by the  discrete Morse theory on simplicial complexes \cite{Forman98} and the PL Morse theory on   triangulated  manifolds \cite{Brehm-Kuhnel,EH10}. We also note that there are infinitely many graphs whose 1-Laplacian eigenvalues are more 
than their orders. For details, see Section \ref{section:1Lapl}.
\item We offer new perspectives that $\Delta_1$ is a combinatorial operator because it encodes so many combinatorial properties involving graphs, while $\Delta_p$ induces a nonlinear evolution from the linear operator $\Delta_2$ to the combinatorial operator $\Delta_1$, and this 
nonlinear evolution  
essentially  implies  
Cheeger-type  inequalities (see Remarks \ref{rem:Cheeger-p-1} and \ref{rem:p-evolution} for details). 
\end{itemize}







The general theory on the eigenvalues of  graph $p$-Laplacians  that we are exploring can be compactly represented in the following diagram:

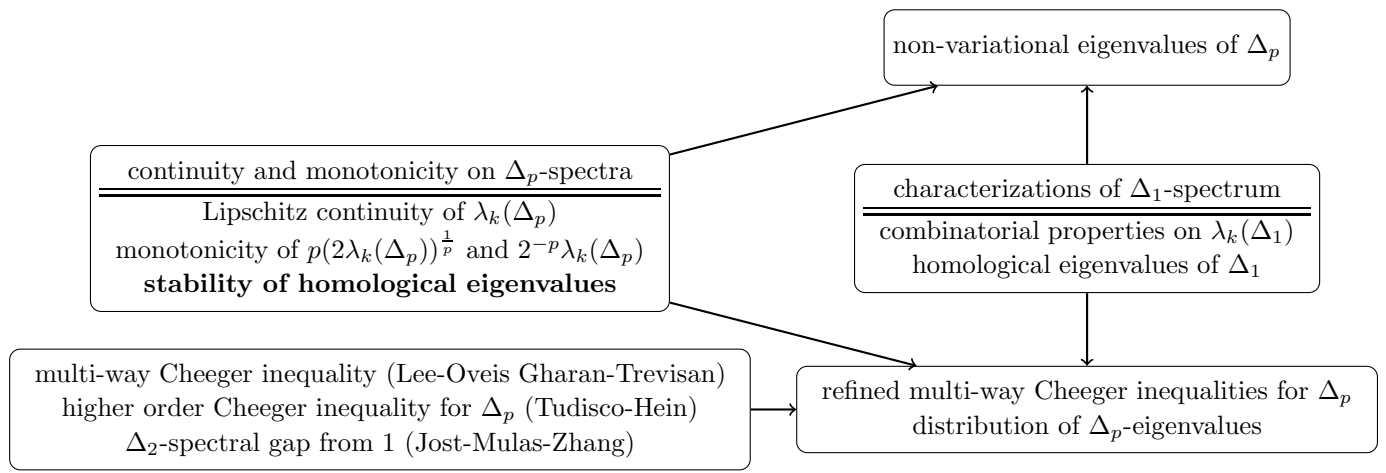
\tikzstyle{startstop} = [rectangle, rounded corners, minimum width=1cm, minimum height=1cm,text centered, draw=black, fill=red!0]
\tikzstyle{io1} = [rectangle, trapezium left angle=80, trapezium right angle=100, minimum width=1cm, minimum height=1cm, text centered, draw=black, fill=blue!0]
\tikzstyle{io2} = [trapezium,  rounded corners, trapezium left angle=100, trapezium right angle=100, minimum width=1cm, minimum height=1cm, text centered, draw=black, fill=yellow!0]
\tikzstyle{process} = [rectangle, minimum width=1cm, minimum height=1cm, text centered, draw=black, fill=orange!0]
\tikzstyle{decision} = [circle, minimum width=1cm, minimum height=1cm, text centered, draw=black, fill=green!0]
\tikzstyle{decision2} = [ellipse, rounded corners=10mm, minimum width=2cm, minimum height=2cm, text centered, draw=black, fill=green!0]
\tikzstyle{arrow} = [thick,->,>=stealth]

\begin{center}{\small
\begin{tikzpicture}[node distance=6.9cm]
\node (1) [startstop] {   \begin{tabular}{c}
    continuity and  monotonicity on  $\Delta_p$-spectra \\ \hline\hline
 Lipschitz continuity of  $\lambda_k(\Delta_p)$ \\
  monotonicity of $p(2\lambda_k(\Delta_p))^{\frac1p}$ and $2^{-p}\lambda_k(\Delta_p)$ \\   \textbf{stability of homological eigenvalues} \\  
    \end{tabular} };

\node (3) [startstop, right of=1, xshift=2.4cm, yshift=0cm] {  
\begin{tabular}{c}
   characterizations of $\Delta_1$-spectrum  \\ \hline\hline combinatorial properties on  $\lambda_k(\Delta_1)$ \\
   {homological eigenvalues of $\Delta_1$}  \\  
    \end{tabular}
};
\node (2) [startstop][startstop, above of=3, xshift=0cm, yshift=-4.5cm] {  non-variational eigenvalues of $\Delta_p$
};
\node (45) [startstop][startstop, below of=3, xshift=0cm, yshift=4.5cm] {
\begin{tabular}{c}
refined multi-way Cheeger inequalities for $\Delta_p$ 
\\
distribution of $\Delta_p$-eigenvalues
\end{tabular}};
\node (6) [startstop][startstop, below of=1, xshift=0cm, yshift=4.5cm] {  \begin{tabular}{c}
  multi-way Cheeger inequality (Lee-Oveis Gharan-Trevisan)
  \\ 
higher order Cheeger inequality for $\Delta_p$ (Tudisco-Hein) 
\\
$\Delta_2$-spectral gap from 1   (Jost-Mulas-Zhang) 
\\ 
    \end{tabular}
};
\draw  [thick,->](1) to node[anchor=north] {  }(2);
\draw  [thick,->](3) to node[anchor=north] {  }(2);
\draw  [thick,->](6) to node[anchor=north] {  }(45);
\draw  [thick,->](1) to node[anchor=north] {  }(45);
\draw  [thick,->](3) to node[anchor=north] {  }(45);
\end{tikzpicture}
}
\end{center}

The organization of this paper is as follows. 
We present in Section \ref{sec:main-proof} auxiliary lemmas and  more relevant results  on continuity and monotonicity of $p$-Laplacian,  and  we  establish sharp estimates for variational eigenvalues in Section \ref{sec:eigen-estimate}, and we study  nonvariational eigenvalues for typical  graphs in Section \ref{sec:nonvariational}.  We refer to  Section \ref{sec:monotonicity} and Proposition \ref{pro:min-max-eigen} for the proof of Theorem \ref{thm:most}. We prove  Theorem \ref{thm:main} in  Section \ref{sec:verification}, and Theorem \ref{thm:tri-link} in Section \ref{sec:1-lap-simplicial-complex}. The proofs of Theorem  \ref{thm:p-lap-C}
 and Theorem  \ref{thm:combina-estimate} are established in  Section \ref{sec:Cheeger-in} and Section \ref{sec:distribution-gap}, respectively. 
 In order 
to make the paper accessible to experts in graph $p$-Laplacian theory as well as those in homology theory we include in the paper somewhat more than the usual amount of 
background material. 



\section{
The spectrum of  $\Delta_p$}\label{sec:main-proof}

First, we recall the definition of graph 1-Laplacian, which has been systematically  studied in  \cite{JZ21,HeinBuhler2010,Chang,CSZ17}. 
\begin{defn}[1-Laplacian for  graphs]\label{def:weighted-graph1-Lap}
Given  a simple, unweighted, undirected, finite graph 
$G=(V,E)$ with $V=\{1,\cdots,n\}$,  the 1-Laplacian $\Delta_1$ is a set-valued map  on $\R^n$  
 defined by
$$(\Delta_1\vec x)_i=\left\{\left.\sum_{j\in V:\{j,i\}\in E}z_{ij}\right|z_{ij}\in \mathrm{Sgn}(x_i-x_j),z_{ij}=-z_{ji}\right\},\;\; i\in V,$$
in which $$\mathrm{Sgn}(t):=\begin{cases}
 \{1\} & \text{if } t>0,\\
 [-1,1] & \text{if }t=0,\\
 \{-1\} & \text{if }t<0.
 \end{cases}$$
\end{defn}
The 1-Laplacian eigenvalue problem 
is to find $\lambda\in\R$ and $\vec x\ne\vec0$ such that
$$(\Delta_1\vec x)_i\cap \lambda \deg(i) \mathrm{Sgn}(x_i)\ne\varnothing,\;\;\forall i \in V,$$
i.e., 
there exist  $z_{ij}\in\mathrm{Sgn}(x_i-x_j)$ with $z_{ij}=-z_{ji}$, $\forall \{i,j\}\in E$,  such that 
\begin{align}
\sum_{j\in V:\{j,i\}\in E}z_{ij}\in \lambda \deg(i) \mathrm{Sgn}(x_i),\;\;\forall i\in V. \label{eq:brief-1-Lap-1}
\end{align}
It is known that the critical points and  critical values\footnote{The critical points and  critical values of the locally Lipschitz function $F_1$ are defined by means of the Clarke subdifferential.}  of the 
Rayleigh quotient
  $$F_1(\vec x):=\frac{\sum_{\{i,j\}\in E} |x_i-x_j|}{\sum_{i\in V}\deg(i)|x_i|}$$
  are eigenvectors and eigenvalues of the graph 1-Laplacian,  respectively.

 Before recalling the variational eigenvalues for 1-Laplacians, we give the following
preliminary definition. 

\begin{defn}\label{def:genus}
For a centrally symmetric  compact set $S$ in $\mathbb{R}^n\setminus\{\vec0\}$, its \emph{Krasnoselskii  genus} is 
\begin{equation*}
\gamma^-(S) :=
\begin{cases}
\min\limits\{k\in\mathbb{Z}^+: \exists\; \text{odd continuous}\; \varphi: S\to \mathbb{S}^{k-1}\} & \text{if}\; S\ne\emptyset,\\
0 & \text{if}\; S=\emptyset.
\end{cases}
\end{equation*}
We let $\gamma^-(S)=0$ if $S$ is not centrally symmetric with respect to the origin  $\vec0$.
\end{defn}
The constants 
  $$ \lambda_k^-(\Delta_1):= \inf_{\gamma^-(S)\ge k}\sup\limits_{\vec x\in S}F_1(\vec x),\;\;k=1,\cdots,n,$$
  define a sequence of critical values of the Rayleigh quotient $F_1$, which are called the {\sl variational eigenvalues} of $\Delta_1$. 
  
The recent work \cite{DFT21} shows that  on a tree graph, the  variational eigenvalues coincide with the multi-way Cheeger constants exactly,  i.e., $\lambda_k^-(\Delta_1)=h_k$, $k=1,\cdots,n$.

\begin{defn}\label{def:genus+}
Given a centrally symmetric compact set $S$ in $\mathbb{R}^n\setminus\{\vec0\}$, define  
\begin{equation*}
\gamma^+(S) :=
\begin{cases}
\max\limits\{k\in\mathbb{Z}^+: \exists\; \text{odd continuous}\; \varphi:\mathbb{S}^{k-1} \to S\} & \text{if}\; S\ne\emptyset,\\
0 & \text{if}\; S=\emptyset.
\end{cases}
\end{equation*}
Let $\gamma^+(S)=0$ when $S$ is not centrally symmetric with respect to the origin  $\vec0$.
\end{defn}

We shall give an brief introduction for the Yang index \cite{Yang}. First, let $-$ be the antipodal map from $\R^n$ to $\R^n$. 
For any centrally symmetric set $S\subset \R^n\setminus\{\vec0\}$ with respect to the origin  $\vec0$, $-$ is a  continuous involution without fixed point, and $-S=S$. 
Let $C_*(S)$ be the singular chain complex with $\mathbb{Z}_2$-coefficients, and denote by $-_\#$ the chain map of $C_*(S)$ induced
by the antipodal map $-$. We say that a $q$-chain $c$ is symmetric if $-_\#(c)=c$. 
The symmetric $q$-chains form a subgroup $C_q(S,-)$ of $C_q(S)$,  and
the boundary operator $\partial_q$ maps  $C_q(S,-)$ to  $C_{q-1}(S,-)$. Then, these
subgroups form a subcomplex $C_*(S,-)$, and we can define the corresponding cycles $Z_q(S,-)$, boundaries $B_q(S,-)$, and homology groups $H_q(S,-)$, respectively. 
Let $\nu:Z_q(S,-)\to\mathbb{Z}_2$ be homomorphisms inductively defined by
\[\nu(z)=\begin{cases}
\mathrm{In}(c),&\text{ if }q=0,
\\
\nu(\partial_q z),&\text{ if }q\ge1
\end{cases}
\]
if
$z=-_\#(c)+c$, where the index of a 0-chain $c=\sum n_i\sigma_i$ is  defined by $\mathrm{In}(c):=\sum n_i$. It is known that $\nu$ is well-defined and $\nu B_q(S,-)=0$, and thus it induces the index homomorphism $\nu_*:H_q(X,-)\to\mathbb{Z}_2$ by $\nu_*([z])=\nu(z)$ (see \cite{Yang}).

\begin{defn}\label{def:Yang}
The Yang index of a centrally symmetric compact set $S$ in $\mathbb{R}^n\setminus\{\vec0\}$ is defined as   
\begin{equation*}
\gamma(S) :=
\begin{cases}
\min\limits\{k\in\mathbb{Z}^+: \nu_*H_k(S,-)=0\} & \text{if}\; S\ne\emptyset,\\
0 & \text{if}\; S=\emptyset.
\end{cases}
\end{equation*}
And we take $\gamma(S)=0$ when $S$ is not centrally symmetric with respect to the origin  $\vec0$.
\end{defn}

By using $\gamma^+$ and  $\gamma$ instead of  $\gamma^-$ in the definition of the variational eigenvalue $\lambda_k^-(\Delta_1)$, we can define $\lambda_k^+(\Delta_1)$ and $\lambda_k(\Delta_1)$, respectively. 
It is actually known that for any symmetric set $S$, $\gamma^+(S)\le\gamma(S)\le \gamma^-(S)$, and thus for any $p\ge1$, 
\begin{equation}\label{eq:inequality-Yang+}
  \lambda_k^-(\Delta_p) \le\lambda_k(\Delta_p)  \le\lambda_k^+(\Delta_p),\;k=1,\cdots,n. 
\end{equation}
  
For the sake of completeness, we give a detailed proof of $\gamma^+(S)\le\gamma(S)\le \gamma^-(S)$ below.
\begin{proof}
For a given symmetric set $S$, if there exist odd continuous maps $\varphi_+:\mathbb{S}^{k^+-1}\to S$ and $\varphi_-:S\to \mathbb{S}^{k^--1}$, then Proposition 2.4 in \cite{PereraAgarwalO'Regan} implies $k^+=\gamma^+(\mathbb{S}^{k^+-1})\le \gamma(S)\le \gamma^+(\mathbb{S}^{k^--1}) =k^-$. According to Example 3.4 in \cite{PereraAgarwalO'Regan}, one has $\gamma^+(\mathbb{S}^{k^+-1})=k^+$ and $\gamma^-(\mathbb{S}^{k^--1})=k^-$. Thus, we obtain $k^+\le \gamma(S)\le k^-$. By the definition of $\gamma^+$ and $\gamma^-$, there holds $\gamma^+(S)\le \gamma(S)\le \gamma^-(S)$. 
\end{proof}

\subsection{
Upper semi-continuity of the spectra of $p$-Laplacians when $p$ varies}
 \label{sec:usc-p-lap}
 In this subsection, we mainly prove 
the upper semi-continuity for  the set of  eigenvalues of $\Delta_p$, namely:
\begin{lemma}\label{lemma:p-lap-usc}
Given a graph, denote by $\mathrm{spec}(\Delta_p)$ the set of all the eigenvalues of  $\Delta_p$. Then the set-valued map $p\mapsto \mathrm{spec}(\Delta_p)$ is upper semi-continuous on $[1,+\infty)$, i.e., for any $p\ge 1$ and  for any $\epsilon>0$, there exists $\delta>0$ such that for any $p'\in (p-\delta,p+\delta)$  with $p'\ge 1$, 
$$\mathrm{spec}(\Delta_{p'})\subset \bigcup\limits_{\lambda\in \mathrm{spec}(\Delta_p)}(\lambda-\epsilon,\lambda+\epsilon).$$
\end{lemma}

For two points (or vectors) $\vec a$ and $\vec b$ in $\R^n$, we use $[\vec a,\vec b]$ to denote the segment with the  endpoints $\vec a$ and $\vec b$. A  \textbf{zonotope} is the Minkowski summation of finitely many segments. For convenience, we also regard a point (resp., a segment) as a 
zonotope of dimension $0$ (resp., dimension  $1$). 
\begin{pro}\label{pro:1-Laplacian-center-p-Lap}
The graph 1-Laplacian maps each vector to a zonotope in the following way:
$$\Delta_1\vec x=\lim\limits_{p\to 1^+}\Delta_p\vec x+ \sum_{\{i,j\}\in E:\,x_i=x_j}[\vec e_i-\vec e_j,\vec e_j-\vec e_i]
$$
 where  
  the addition `$+$' is  in the sense of Minkowski summation, $(\vec e_i)_{i=1}^n$ is the standard orthogonal base of $\R^n$. 
\end{pro}

\begin{proof}Note that
\begin{align*}
\Delta_1\vec x&=\left\{\left.\sum_{i=1}^n\sum_{j\in V:\{j,i\}\in E}z_{ij}\vec e_i\right|z_{ij}\in \mathrm{Sgn}(x_i-x_j),z_{ij}=-z_{ji}\right\}
\\&=\left\{\left.\sum_{\{i,j\}\in E}z_{ij}(\vec e_i-\vec e_j)\right|z_{ij}\in \mathrm{Sgn}(x_i-x_j)\right\}
\\&=\left\{\left.\sum_{\{i,j\}\in E:x_i>x_j}(\vec e_i-\vec e_j)+\sum_{\{i,j\}\in E:x_i=x_j}z(\vec e_i-\vec e_j)\right|z\in [-1,1]\right\}
\\&=\sum_{\{i,j\}\in E:x_i>x_j}(\vec e_i-\vec e_j)+\sum_{\{i,j\}\in E:x_i=x_j}[\vec e_i-\vec e_j,\vec e_j-\vec e_i],
\end{align*}
where the addition `$+$' in the last  equality is  in the sense of Minkowski summation. 

We  note that $\lim\limits_{p\to 1^+}(\Delta_p\vec x)_i=\sum\limits_{j\in V:\{j,i\}\in E}\mathrm{sign}(x_i-x_j)$, and thus
 $$\lim\limits_{p\to 1^+}\Delta_p\vec x=\sum_{i=1}^n\sum_{j\in V:\{j,i\}\in E}\mathrm{sign}(x_i-x_j)\vec e_i=\sum_{\{i,j\}\in E:x_i>x_j}(\vec e_i-\vec e_j)
 ,$$
where 
$$\mathrm{sign}(t):=\begin{cases}
 1 & \text{if } t>0,\\
 0 & \text{if }t=0,\\
 -1 & \text{if }t<0,
 \end{cases}$$
indicates the standard sign function. The proof is completed. 
\end{proof}

It is known that the $1$-Laplacian can be regarded as the limit of the  $p$-Laplacians in some sense. 
Interestingly, we have the following exact description of the limit of $\Delta_p$ as $p$ tends to $1$. 
\begin{pro}\label{pro:p-to-1-converge}
For any $\vec x\in\R^n$, 
$$\Delta_1\vec x=\lim\limits_{\substack{\hat{\vec x}\to\vec x\\ p\to1^+}}\Delta_p\hat{\vec x}=\lim\limits_{\substack{\delta\to 0^+\\ p\to1^+}}\Delta_p(\mathbb{B}_\delta(\vec x)),$$
and it is interesting that 
$$\lim\limits_{p\to1^+}\lim\limits_{\hat{\vec x}\to\vec x}\Delta_p\hat{\vec x}=\lim\limits_{p\to1^+}\lim\limits_{\delta\to 0^+}\Delta_p(\mathbb{B}_\delta(\vec x))=\lim\limits_{p\to1^+}\Delta_p\vec x=\text{the center point of }\Delta_1\vec x,$$
$$\lim\limits_{\hat{\vec x}\to\vec x}\lim\limits_{p\to1^+}\Delta_p\hat{\vec x}=\lim\limits_{\delta\to 0^+}\lim\limits_{p\to1^+}\Delta_p(\mathbb{B}_\delta(\vec x))= \bigcup_{\hat{\vec x}\text{ near }\vec x} \lim\limits_{p\to1^+}\Delta_p\hat{\vec x}=\{\text{the centers of the faces of 
}\Delta_1\vec x\},$$
where  $\mathbb{B}_\delta(\vec x)$ is the  $\delta$-neighborhood
of $\vec x$, i.e., the $\delta$-ball centered at $\vec x$. 
The last expression `the centers of the faces of $\Delta_1\vec x$' depends on the fact that a convex  polytope is a zonotope if and only if every face is  centrally symmetric (see \cite{Coxeter62,Shephard67}). 


In addition, the limit points of the eigenvalues of $\Delta_p$ are  
eigenvalues of $\Delta_1$, as $p$ tends to $1$.
\end{pro}

\begin{proof}
We study the limit  
$\lim\limits_{p\to1^+}\Delta_p(\mathbb{B}_\delta(\vec x))$ when $\delta$ is sufficiently small. In fact, for any $\vec x$, taking $0<\delta<\frac12\min(\{|x_i-x_j|:\{i,j\}\in E\}\setminus\{0\})
$, then for any $\vec y\in \mathbb{B}_\delta(\vec x)$,  for $\{i,j\}\in E$,  $x_i>x_j$ implies $y_i>y_j$, and thus 
$\Delta_1\vec y$ is a face of $\Delta_1\vec x$.  
In consequence, by Proposition \ref{pro:1-Laplacian-center-p-Lap},  $\lim\limits_{p\to1^+}\Delta_p\vec y=\text{the center of }\Delta_1\vec y$, and therefore,  $\lim\limits_{p\to1^+}\Delta_p(\mathbb{B}_\delta(\vec x))=\text{the centers of the faces of 
}\Delta_1\vec x$.

By the definition of the $p$-Laplacian, for $p>1$,  $\Delta_p:\R^n\to\R^n
$ is a continuous map, and thus 
$\lim\limits_{\hat{\vec x}\to\vec x}\Delta_p\hat{\vec x}=\lim\limits_{\delta\to 0^+}\Delta_p(\mathbb{B}_\delta(\vec x))=\Delta_p\vec x$. Combining with  Proposition \ref{pro:1-Laplacian-center-p-Lap}, $\lim\limits_{p\to1^+}\lim\limits_{\hat{\vec x}\to\vec x}\Delta_p\hat{\vec x}=\lim\limits_{p\to1^+}\Delta_p\vec x=\text{the center point of }\Delta_1\vec x$. 

It is easy to check that the set of  limit points of the function $(t,p)\mapsto|t|^{p-2}t$ as $p\to1^+$ and $t\to 0$ is the closed interval $[-1,1]$, which we shall write as $\lim\limits_{t\to 0, p\to1^+}|t|^{p-2}t=[-1,1]$. 
Taking $t_{ij}=\hat{x}_i-\hat{x}_j=-t_{ij}$ and $E_x=\{\{i,j\}\in E:x_i=x_j\}$, we have 
$\Delta_p\hat{\vec x}=\sum_{i\in V}\sum_{j\in V:\{j,i\}\in E}|t_{ij}|^{p-2}t_{ij}\vec e_i=\sum_{\{i,j\}\in E}|t_{ij}|^{p-2}t_{ij}(\vec e_i-\vec e_j)$ and 
\begin{align*}
 \lim\limits_{\substack{\hat{\vec x}\to\vec x\\ p\to1^+}}\Delta_p\hat{\vec x}&=\lim\limits_{\substack{t_{ij}\to x_i-x_j\\ p\to1^+}}\sum_{\{i,j\}\in E\setminus E_x}|t_{ij}|^{p-2}t_{ij}(\vec e_i-\vec e_j)+\lim\limits_{\substack{t_{ij}\to 0\\ p\to1^+}}\sum_{\{i,j\}\in E_x}|t_{ij}|^{p-2}t_{ij}(\vec e_i-\vec e_j)
 \\&=\sum_{\{i,j\}\in E:x_i>x_j}(\vec e_i-\vec e_j)+\sum_{\{i,j\}\in E_x}[-1,1](\vec e_i-\vec e_j)
 \\&=\sum_{\{i,j\}\in E:x_i>x_j}(\vec e_i-\vec e_j)+\sum_{\{i,j\}\in E:x_i=x_j}[\vec e_i-\vec e_j,\vec e_j-\vec e_i]=\Delta_1\vec x
\end{align*}
where the last equality has been shown in the proof of Proposition \ref{pro:1-Laplacian-center-p-Lap}.

Let $(\lambda(\Delta_p),\vec x^p)$ be an eigenpair of $\Delta_p$ such that $\lambda(\Delta_p)\to\lambda$,  $\vec x^p\to \vec x\ne\vec0$, $p\to 1^+$.  
Then the limit points $\lim_{p\to1^+}\Delta_p\vec x^p$ are included in  $\Delta_1\vec x$. 
Similarly, the limit points  $\lim_{p\to1^+}\lambda(\Delta_p)|(\vec x^p)_i|^{p-2}(\vec x^p)_i$ are included in  $\lambda \mathrm{Sgn}(x_i)$. 

Thus, it follows from the eigenvalue equation $(\Delta_p\vec x^p)_i=\lambda(\Delta_p)\deg(i)|(\vec x^p)_i|^{p-2}(\vec x^p)_i$ and the compactness of $\Delta_1\vec x$ and $\mathrm{Sgn}(x_i)$ that $(\Delta_1\vec x)_i\cap \lambda \deg(i) \mathrm{Sgn}(x_i)\ne\varnothing$, $\forall i \in V$, which indicates  that $(\lambda,\vec x)$ is an eigenpair.

Suppose that $\lim_{p\to1^+}\lambda(\Delta_p)=\lambda$. 
We can always assume that each eigenvector $\vec x^p$ is normalized, i.e., $\|\vec x^p\|_2=1$. Then $(\vec x^p)_{p\to1^+}$ has a convergent subsequence,  with a limit  point 
denoted by $\vec x$. Then, by the above discussions, there is no difficulty to show that $(\lambda,\vec x)$ is an eigenpair of $\Delta_1$. 
 \end{proof}

Denote by  $S_\lambda(\Delta_p)=\{\vec x\ne\vec0:(\lambda,\vec x)\text{ is an eigenpair of }\Delta_p\}$ the eigenspace  corresponding to $\lambda$. For convenience, we set $S_\lambda(\Delta_p)=\varnothing$ if $\lambda$ is not an eigenvalue. And, we usually  work on the set of   the normalized  
eigenvectors, i.e.,  $\hat{S}_\lambda(\Delta_p)=\{\vec x\in\R^n:(\lambda,\vec x)\text{ is an eigenpair of }\Delta_p,\,\|\vec x\|_\infty=1\}$. 

\begin{pro}\label{pro:usc-eigenspace}
The set-valued map 
$(p,\lambda)\mapsto \hat{S}_\lambda(\Delta_p)$ defines an upper semi-continuous map on $[1,+\infty)\times[0,+\infty)$. 
\end{pro}

\begin{proof}
Suppose $p_n\to p$ and $\lambda_n\to \lambda$, $n\to+\infty$. 
By a slight generalization of  Proposition \ref{pro:p-to-1-converge}, if $\lambda$ is not an eigenvalue of $\Delta_p$, then $\lambda_n$ cannot be an eigenvalue of $\Delta_{p_n}$ when $n$ is sufficiently large. In this case, $\hat{S}_\lambda(\Delta_p)=\varnothing=\hat{S}_{\lambda_n}(\Delta_{p_n})$. So $\hat{S}_\lambda(\Delta_p)$ is continuous at $(p,\lambda)$. 

If $\lambda$ is an eigenvalue of $\Delta_p$, we  shall prove that $\hat{S}_\lambda(\Delta_p)$ is upper semi-continuous at $(p,\lambda)$. 
Suppose the contrary, that $(p,\lambda)\mapsto \hat{S}_\lambda(\Delta_p)$ is not upper semi-continuous at some $(p,\lambda)\in [1,+\infty)\times[0,+\infty)$. Then, there exist $\epsilon>0$ and a subsequence $(p_n,\lambda_n)\to (p,\lambda)$ such that $\hat{S}_{\lambda_n}(\Delta_{p_n})\not\subset B_\epsilon\left(\hat{S}_\lambda(\Delta_p)\right)$, where $B_\epsilon\left(\hat{S}_\lambda(\Delta_p)\right)$ denotes the $\epsilon$-neighborhood
of the nonempty   compact  set $\hat{S}_\lambda(\Delta_p)$. Then, by a standard technique, we derive that there is a limit point $\vec x$ of $\hat{S}_{\lambda_n}(\Delta_{p_n})$ which is not in $\hat{S}_\lambda(\Delta_p)$. However,  similar to the proof of Proposition \ref{pro:p-to-1-converge},  $\vec x$  must be an eigenvector of $\lambda(\Delta_p)$, which is a contradiction.
\end{proof}

\begin{remark}
Similarly,  $p\mapsto \bigcup\limits_{\lambda\in \mathrm{spec}(\Delta_p)}\hat{S}_\lambda(\Delta_p)$ is also an upper semi-continuous set-valued map.
\end{remark}

\begin{proof}[Proof of Lemma \ref{lemma:p-lap-usc}]


We only need to show that $p\mapsto \mathrm{sp ec}(\Delta_p)$ is an  upper semi-continuous set-valued map.  Suppose the contrary, that there exists  $\epsilon>0$ and a subsequence $p_n\to p$ and a subsequence $\lambda_{p_n}\in\mathrm{sp ec}(\Delta_{p_n})\setminus  B_\epsilon(\mathrm{sp ec}(\Delta_p))$. Without loss of generality, we may assume that $\lambda_{p_n}$ converges to some number $\lambda$. Then, by $\hat{S}_{\lambda_n}(\Delta_{p_n})\ne\varnothing$ and  Proposition \ref{pro:usc-eigenspace}, there holds $\hat{S}_\lambda(\Delta_p)\ne \varnothing$, implying that $\lambda\in \mathrm{sp ec}(\Delta_p)$. 
This is a contradiction. 
\end{proof}

In addition, we have the upper semi-continuity  of the eigenvalue  multiplicity. Precisely,  denote by  $m(p,\lambda):=\gamma(\hat{S}_\lambda(\Delta_p))$  the $\gamma$-multiplicity of $\lambda$ of $\Delta_p$ (if $\lambda$ is not an eigenvalue of $\Delta_p$, we simply write $m(p,\lambda)=0$). As an analogous, we can define the $\gamma^-$-multiplicity $m^-(p,\lambda)$ and the $\gamma^+$-multiplicity $m^+(p,\lambda)$ for an eigenvalue $\lambda$ of   $\Delta_p$. 
We remark here that the $\gamma^-$-multiplicity was introduced in \cite{JZ21-PM}, and independently in \cite{DFT21}.  
\begin{pro}\label{pro:usc-mult}
 The multiplicity function  $m^-:[1,+\infty)\times [0,+\infty)\to\mathbb{N}$ 
 is upper semi-continuous, i.e., $\limsup\limits_{(p',\lambda')\to (p,\lambda)}m^-(p',\lambda')\le m^-(p,\lambda)$. 
\end{pro}

\begin{proof}
By the continuity of the Krasnoselskii  genus $\gamma^-$, there exists $\epsilon>0$ such that $\gamma^-(B_\epsilon(\hat{S}_\lambda(\Delta_p)))=\gamma^-(\hat{S}_\lambda(\Delta_p))$. 
By Proposition \ref{pro:usc-eigenspace}, for any $(p',\lambda')$ sufficiently close to $(p,\lambda)$, $\hat{S}_{\lambda'}(\Delta_{p'})\subset B_\epsilon(\hat{S}_\lambda(\Delta_p))$. 
Then, by the monotonicity of $\gamma^-$, we have $\gamma^-(\hat{S}_{\lambda'}(\Delta_{p'}))\le \gamma^-(B_\epsilon(\hat{S}_\lambda(\Delta_p)))$. 
Therefore, $m^-(p',\lambda')=\gamma^-(\hat{S}_{\lambda'}(\Delta_{p'}))\le \gamma^-(\hat{S}_\lambda(\Delta_p))=m^-(p,\lambda)$.
\end{proof}

Finally, we show the locally Lipschitz continuity of variational eigenvalues. 
\begin{pro}\label{pro:min-max-eigen}
For any $k\in\{1,\cdots,n\}$,  the functions $p\mapsto \lambda_k(\Delta_p)$,  $p\mapsto \lambda_k^-(\Delta_p)$ and  $p\mapsto \lambda_k^+(\Delta_p)$, 
 are locally Lipschitz continuous with respect to $p\in [1,+\infty)$.
\end{pro}

\begin{proof}[Proof of Proposition \ref{pro:min-max-eigen}]
It is known that the  $k$-th min-max eigenvalue $\lambda_k^-(\Delta_p)$ varies 
  continuously in $p$. This statement  
  is a direct consequence of the main result in \cite{DM14}, and thus we omit it. Actually, the continuity of $\lambda_k(\Delta_p)$, $\lambda_k^-(\Delta_p)$ and $\lambda_k^+(\Delta_p)$ w.r.t. $p$ can 
  be derived from   Theorem \ref{thm:most} straightforwardly. 
  
Moreover, since the function $(p,q)\mapsto \max_{x\ne0}|F_q(\vec x)-F_p(\vec x)|$ is locally Lipschitz, it follows from Lemma \ref{lemma:homology-eigen-increasing} that $p\mapsto \lambda_k(\Delta_p)$,  $p\mapsto \lambda_k^-(\Delta_p)$ and  $p\mapsto \lambda_k^+(\Delta_p)$ 
 are also locally Lipschitz. 
\end{proof}

\begin{remark}
The continuity of the $k$-th min-max eigenvalue $\lambda_k(\Delta_p)$ is  essentially 
known. We refer to  \cite{Parini11,DM14} for the case of   
the  $p$-Laplacian on Euclidean  domains under Dirichlet  boundary conditions. 
\end{remark}

\subsection{Homological eigenvalues 
of $\Delta_p$ 
and eigenvalues produced by homotopical linking}
\label{sec:homological-eigen}
The homological critical value is an important concept used in the theory of persistent
homology. We shall adopt the definition proposed  by  Govc \cite{Govc16}, 
which is inspired by the definition suggested by Bubenik and Scott \cite{BS14}: 
\begin{defn}\label{defn:homological-critical}
A  real number $c$ is a {\sl homological regular value} of the function $F$ if there exists $\epsilon> 0$ such that for each pair of real numbers $t_1<t_2$ on the interval $(c-\epsilon,c+\epsilon)$, the inclusion $\{F\le t_1\}\hookrightarrow \{F\le t_2\}$ induces isomorphisms on all homology groups \cite{BS14}. A
real number that is not a homological regular value of $F$ is called a {\sl homological
critical value} of $F$.
\end{defn}

It is clear that 
the following 
\textbf{critical value lemma} holds: 

\begin{lemma}\label{lemma:homological-cri}
Suppose the  function $F$ has no homological critical values on the closed interval $[x,y]$, then the inclusion  $F^{-1}(-\infty,x]\hookrightarrow F^{-1}(-\infty,y]$
induces isomorphisms on all homology groups. 
\end{lemma}

\begin{remark}
A {\sl symmetric homological critical value} \cite{SEH07} of $F$ is a real number $c$ for which there exists an integer $k$ such that for all sufficiently small $\epsilon> 0$, the map $H_k(\{F\le c-\epsilon\})\hookrightarrow H_k(\{F\le c+\epsilon\})$ induced by inclusion is not an isomorphism \cite{BS14}. Here $H_k$ denotes the $k$-th singular homology (possibly with coefficients in a field). 
As mentioned in \cite{Govc16,BS14}, this widely cited definition
 given by Cohen-Steiner,  Edelsbrunner and  Harer \cite{SEH07} doesn't imply the critical value lemma in generic scenarios. 

In fact, if we replace homological critical values by symmetric homological critical values in Lemma \ref{lemma:homological-cri}, the conclusion 
doesn't hold even though $F$ is continuous.  Govc \cite{Govc16} presented a  modified version of the critical value lemma: Suppose $F$ is continuous and has no symmetric homological critical values on the interval $[x,y)$, then the inclusion  $F^{-1}(-\infty,x)\hookrightarrow F^{-1}(-\infty,y)$
induces isomorphisms on all homology groups. 
\end{remark}

We shall simply recall the deformation lemma which will be essentially used many times in the present paper. In critical point theory, the deformation lemma roughly says that if an interval $[a,b]\subset\R$ contains no  critical value of a typical continuous function $F$, then there is a continuous  deformation from the sublevel sets $\{F\le b\}$ to $\{F\le a\}$, which induces a homotopy equivalence between these sublevel sets 
\cite{Clarke,PereraAgarwalO'Regan}.

\begin{pro}\label{pro:minmax-homological}
Let $F$ be a continuous  even function on $\{\vec x\in\R^n:\|\vec x\|_\infty=1\}$. Then, the min-max critical values
\[\lambda_k:=\inf_{\gamma(S)\ge k}\sup\limits_{\vec x\in S}F(\vec x),\;k=1,2,\cdots,\]
of $F$ are homological  critical values of $F$. 
\end{pro}
\begin{proof}
Suppose  $\lambda_k$ is the $k$-th min-max critical value of $F$, and assume  $\lambda_k<\lambda_{k+1}$. We shall prove that $\lambda_k$ is a homological critical value. On one hand, there exists $A\subset \{F\le \lambda_k\}$ with $\gamma(A)\ge k$, which yields  $\gamma\{F\le \lambda_k\}\ge k$. On the other hand, for any $A'$ with $\gamma(A')\ge k+1$, $A'\not\subset \{F\le \lambda_k\}$, which implies that  $\gamma\{F\le \lambda_k\}< k+1$. Therefore, $\gamma\{F\le \lambda_k\}= k$. 

Similarly, it is easy to check that  $\gamma\{F\le \lambda_k-\epsilon\}<k$ for any sufficiently small $\epsilon>0$. 
Since 
different Yang indices imply different homology groups (see \cite{Yang} or Section 0.7 in \cite{PereraAgarwalO'Regan}), the homology of the topology of the  sublevel set changes from $\lambda_k-\epsilon$ to $\lambda_k$. Therefore, $\lambda_k$ is a homological  critical value of $F$.
\end{proof}

Returning to graph $p$-Laplacians,  we introduce the concept of  homological  eigenvalues. 
\begin{defn}
We say $\lambda$ is a homological  eigenvalue of  $\Delta_p$, if $\lambda$ is a homological  critical value of the $p$-Rayleigh quotient
$$F_p(\vec x):=\frac{\sum_{\{j,i\}\in E}|x_i-x_j|^{p}}{\sum_{i\in V}\deg(i)|x_i|^p}$$
where $1\le p<+\infty$. We say $\lambda$ is an {\sl isolated  homological  eigenvalue} of  $\Delta_p$, if $\lambda$ is a homological  eigenvalue, and $\lambda$ is not a limit point of the eigenvalues of   $\Delta_p$.
\end{defn}

For simplicity, we  use $\{F_p\le \lambda\}$ to indicate the lower level set $\{\vec x\in\R^n:F_p(\vec x)\le \lambda,\|\vec x\|_\infty=1\}$.

\begin{theorem}\label{thm:homotopy-eigen}
For any homological  eigenvalue $\lambda$ of $\Delta_1$, and for any $\epsilon>0$, there exists $\delta>0$ such that for any $p\in (1,1+\delta)$,  $\Delta_{p}$ has an     eigenvalue in $(\lambda-\epsilon,\lambda+\epsilon)$.
\end{theorem}

Theorem \ref{thm:homotopy-eigen} overcomes the difficulty 
that in general, a  non-variational  $\Delta_1$-eigenvalue may not be a limit point of the $\Delta_p$-eigenvalues.  
We shall prove it at the end of this section.



The concept of homotopical linking is useful for obtaining critical points. 

\begin{defn}[Definition 0.17 in \cite{PereraAgarwalO'Regan}]
For subsets $Q$, $Q'$ and $S$ of a given topological space $X$ with $Q\subset Q'$,   
we say that $ Q$ homotopically links $S$  with respect to $Q'$, if $ Q\cap S=\varnothing$ and  for any continuous map  $\gamma:Q'\to X$ with $\gamma|_Q=\mathrm{id}$  (i.e., $\gamma(\vec x)= \vec x$ for any $\vec x\in Q$),  $\gamma(Q')\cap S\ne \varnothing$. 
\end{defn}
If  $ Q$ homotopically links $S$  with respect to $Q'$, and $f:X\to\R$ is continuous with  $\min\limits_{\vec x\in S}f(\vec x)>\max\limits_{\vec x\in   Q}f(\vec x)$, then by linking theorem (Theorem 0.21 and Proposition 3.21 in \cite{PereraAgarwalO'Regan}), $$
\inf\limits_{\text{continuous }\gamma:Q'\to X\text{ with }\gamma|_Q=\mathrm{id}}\;\max\limits_{\vec x\in\gamma(Q')}f(\vec x)$$  is a critical value of $f$, which is said to be  a  {\sl critical value of $f$  produced by  homotopical  linking}. 

The next lemma shows the stability and local monotonicity for both isolated homological critical values and critical values produced by homotopical linking.

\begin{lemma}\label{lemma:homology-eigen-increasing}
Given a function  $f$ on a topological space $X$,  for any isolated homological critical value $\lambda$ of $f$, and for any sufficiently small 
$\epsilon>0$, for any $\epsilon$-perturbation $f_\epsilon$ of $f$, i.e.,   $\|f_\epsilon-f\|_\infty<\epsilon$, there is a homological critical value of $f_\epsilon$ in $[\lambda-3\epsilon,\lambda+3\epsilon]$. This is the stability of isolated homological critical
values.

If $f_\epsilon$ is further assumed to an increasing $\epsilon$-perturbation, i.e.,  $f(x)\le f_\epsilon(x)\le  f(x)+\epsilon$, $\forall x$, then 
there is a homological critical value of $f_\epsilon$ in $[\lambda,\lambda+\epsilon]$. This shows the local monotonicity  of isolated homological critical
values.

All the above statements still hold if we use critical values produced by homotopical linking instead of isolated homological critical values, and if we further assume that both $f$ and $f_\epsilon$ are continuous. 
\end{lemma}
\begin{proof}
We concentrate on homological critical values. Without loss of generality, we assume that there is no homological critical value of $f$ in $[\lambda-3\epsilon,\lambda)\cup(\lambda,\lambda+3\epsilon]$.  Then, we can find $\tilde{\epsilon}_1>0$ such that any value in $(\lambda-3\epsilon-\tilde{\epsilon}_1,\lambda)\cup(\lambda,\lambda+3\epsilon+\tilde{\epsilon}_1)$ is homological regular.  

For the first statement about  the stability of isolated homological critical
values,  we suppose the contrary, that there is no homological critical value of $f_\epsilon$ in $[\lambda-3\epsilon,\lambda+3\epsilon]$.  Then, there exists  $\epsilon_1\in(0,\tilde{\epsilon}_1]$ such that any value in $(\lambda-3\epsilon-\epsilon_1,\lambda+3\epsilon+\epsilon_1)$ is   homological regular for  $f_\epsilon$. We consider the inclusion relation
$$\{f\le c\}\mathop{\hookrightarrow}\limits^{i_1} \{f_\epsilon\le c+\epsilon\}\mathop{\hookrightarrow}\limits^{i_2}\{f\le c+2\epsilon\}\mathop{\hookrightarrow}\limits^{i_3}\{f_\epsilon \le c+3\epsilon\}.
$$
Taking $c=\lambda-3\epsilon-\epsilon_1'$ in the above inclusion relation, where $0<\epsilon_1'<\epsilon_1$, 
by the critical value lemma, 
both the inclusion $\{f_\epsilon\le c+\epsilon\} \mathop{\hookrightarrow}\limits^{i_3\circ i_2} \{f_\epsilon \le c+3\epsilon\}$ and the inclusion $\{f\le c\}\mathop{\hookrightarrow}\limits^{i_2\circ i_1} \{f\le c+2\epsilon\}$ induce isomorphisms on all homology groups, and then 
by Lemma 3.1 in \cite{Govc16}, the inclusions  $i_1$, $i_2$ and $i_3$ also induce isomorphisms on all homology groups.  

In consequence, we get the {\sl homological equivalence}   $\{f\le \lambda-3\epsilon-\epsilon_1'\}\sim \{f_\epsilon \le \lambda-\epsilon_1'\}$, i.e., the inclusion  $\{f_\epsilon\le \lambda-3\epsilon-\epsilon_1'\} \hookrightarrow \{f \le \lambda-\epsilon_1'\}$ induces isomorphisms on all homology groups.  

Similarly, taking $c=\lambda+\epsilon_1'$   
in the following inclusion chain
$$\{f_\epsilon\le c\}\subset \{f\le c+\epsilon\}\subset\{f_\epsilon\le c+2\epsilon\}\subset\{f \le c+3\epsilon\},
$$
we have the homological equivalence $ \{f_\epsilon \le \lambda+\epsilon_1'\}\sim \{f\le \lambda+3\epsilon+\epsilon_1'\}$. 

Note that by the inclusion chain  
$$\{f\le \lambda-3\epsilon-\epsilon_1'\}\mathop{\hookrightarrow}\limits^{i^1} \{f_\epsilon \le \lambda-\epsilon_1'\}\mathop{\hookrightarrow}\limits^{i^2} \{f_\epsilon \le \lambda+\epsilon_1'\}\mathop{\hookrightarrow}\limits^{i^3} \{f\le \lambda+3\epsilon+\epsilon_1'\}$$
we have the  
diagram 
$$\xymatrix{H_*(\{f\le \lambda-3\epsilon-\epsilon_1'\})\ar[rr]^{i^3_*\circ i^2_*\circ  i^1_*}\ar[d]^{i^1_*}&& H_*(\{f\le \lambda+3\epsilon+\epsilon_1'\})\\ H_*(\{f_\epsilon \le \lambda-\epsilon_1'\}) \ar[rr]^{i^2_*}&& H_*(\{f_\epsilon \le \lambda+\epsilon_1'\})\ar[u]^{i^3_*}}.$$
Since $f_\epsilon$ has no homological critical value in $(\lambda-3\epsilon-\epsilon_1,\lambda+3\epsilon+\epsilon_1)$, the inclusion $i^2$ induces the isomorphisms $i^2_*$ on all homology groups. 
Then, we have obtained that all  the inclusions $i^1,i^2,i^3$ induce  isomorphisms $i^1_*,i^2_*,i^3_*$, which leads to an isomorphism \[(i^3\circ i^2\circ  i^1)_*=i^3_*\circ i^2_*\circ  i^1_*:H_*(\{f\le \lambda-3\epsilon-\epsilon_1'\})\cong H_*(\{f\le \lambda+3\epsilon+\epsilon_1'\}).\]  
That is, we obtain  the homological equivalence $\{f \le \lambda-3\epsilon-\epsilon_1'\}\sim \{f \le \lambda+3\epsilon+\epsilon_1'\}$,  which is a contradiction with the assumption that $\lambda$ is a homological  critical value of $f$. 

\vspace{0.1cm}

For the second statement on the local monotonicity  of isolated homological critical
values, note that
\begin{equation}\label{eq:f-fe-inclu}
\{f_\epsilon\le c\}\subset \{f\le c\}\subset\{f_\epsilon\le c+\epsilon\}\subset\{f \le c+\epsilon\}
.    
\end{equation}
A similar argument yields that there is a homological critical value of $f_\epsilon$ in $[\lambda-\epsilon,\lambda+\epsilon]$. Suppose the contrary, that there is no homological critical value of $f_\epsilon$ in $[\lambda,\lambda+\epsilon]$. 
Then, without loss of generality, we may assume that there is no homological critical value of $f_\epsilon$ in $[\lambda-\epsilon_1,\lambda+\epsilon+\epsilon_1]$. 
Then, for any $0<\epsilon'<\epsilon_1$, it is not difficult to verify that the inclusion relation 
$$\{f\le \lambda-\epsilon-\epsilon'\}\subset \{f_\epsilon\le \lambda-\epsilon'\}\subset\{f\le \lambda-\epsilon'\}\subset\{f_\epsilon \le \lambda+\epsilon-\epsilon'\}
$$
implies the homological equivalence $\{f_\epsilon\le \lambda-\epsilon'\}\sim\{f\le \lambda-\epsilon'\}$. However, since there is no homological critical value of $f_\epsilon$ in $(\lambda-\epsilon_1,\lambda+\epsilon]$, we have the homological equivalence  $\{f_\epsilon\le \lambda-\epsilon'\}\sim \{f_\epsilon\le \lambda+\epsilon'\}$. Taking $c=\lambda+\epsilon'$ in  \eqref{eq:f-fe-inclu}, we  derive the homological equivalence  $\{f_\epsilon\le \lambda+\epsilon'\}\sim \{f\le \lambda+\epsilon'\}$ in a similar manner as shown in the proof of  the first statement.  Therefore, we deduce the homological equivalence  $\{f\le \lambda-\epsilon'\}\sim \{f\le \lambda+\epsilon'\}$, which  contradicts  to the assumption that $\lambda$ is a homological  critical value of $f$.   


\vspace{0.1cm}

Finally, we focus on the critical values produced by homotopical linking.  Suppose that $ Q$ homotopically links $S$  with respect to $Q'$, where $Q\subset Q'\subset X$, $S\subset X$, and  $\min\limits_{\vec x\in S}f(\vec x)>\max\limits_{\vec x\in   Q}f(\vec x)+4\epsilon$. Then, for any continuous function  $f_\epsilon:X\to\R$ with $\|f-f_\epsilon\|_\infty<\epsilon$, $\min\limits_{\vec x\in S}f_\epsilon(\vec x)>\max\limits_{\vec x\in   Q}f_\epsilon(\vec x)$ which implies that $$\lambda_{Q,Q',S}(f_\epsilon):=
\inf\limits_{\text{continuous }\gamma:Q'\to X\text{ with }\gamma|_Q=\mathrm{id}}\;\max\limits_{\vec x\in\gamma(Q')}f_\epsilon(\vec x)$$  is a critical value of $f_\epsilon$. It is clear that $|\lambda_{Q,Q',S}(f_\epsilon)-\lambda_{Q,Q',S}(f)|<\epsilon$. Similarly, if $f(x)\le f_\epsilon(x)\le  f(x)+\epsilon$, $\forall x\in X$, then $\lambda_{Q,Q',S}(f)\le \lambda_{Q,Q',S}(f_\epsilon)\le \lambda_{Q,Q',S}(f)+\epsilon$. The proof is completed. 
\end{proof}

We should note that Theorem \ref{thm:homotopy-eigen} is a  consequence of Lemma \ref{lemma:homology-eigen-increasing} due to the fact that the eigenvalues of the 1-Laplacian are isolated. 
For the sake of completeness, we write down a brief proof below.

\begin{proof}[Proof of Theorem \ref{thm:homotopy-eigen}]
By  Theorem 1  in \cite{CSZ17}, $\Delta_1$ has finitely many eigenvalues. 
We may assume without loss of generality that $\epsilon<\frac12\min\{|\lambda'-\lambda''|:\text{different eigenvalues  }\lambda',\lambda''\text{ of }\Delta_1\}$.  Then, there is no eigenvalue of $\Delta_1$ in the set $(\lambda-\epsilon,\lambda)\cup (\lambda,\lambda+\epsilon)$. 
That is, $\lambda$ is actually an isolated homological eigenvalue of $\Delta_1$, and thus it is an isolated homological critical value of $F_1$. 


Take sufficiently small $\delta>0$ such that  $|F_{p}(\vec x)-F_1(\vec x)|<\epsilon/4$ for any $1<p<1+\delta$, and for any  $\vec x$ with $\vec x\ne \vec0$. 
By Lemma \ref{lemma:homology-eigen-increasing}, there exists a homological critical value of $F_{p}$ in $(\lambda-\epsilon,\lambda+\epsilon)$,  which is actually an eigenvalue of $\Delta_{p}$. 
The proof is completed.
\end{proof}

The next example shows that the local monotonicity doesn't hold for a non-homological critical value. 
\begin{remark}
Let $f_t(x)=(x+t)^3-\frac{t^3}{6}$. Then $f_t>f_0$ when $t>0$ and $f_t<f_0$ when $t<0$, while the unique critical value of $f_t$ is $-\frac{t^3}{6}$. This implies that a larger function may have a smaller critical value. 

In addition, the stability also  fails  for a non-homological critical value. For example, let $g_t(x)=x^3+t^2x$. Then,   $g_0$ has the unique critical value $0$,  but for  $t\ne0$, $g_t$ has no critical value. 

From these examples, we actually show that both the stability and local monotonicity do not hold in the case of non-homological critical values.
\end{remark}

\begin{defn}\label{def:homotopy-eigen-p}
We define the {\sl eigenvalues of $\Delta_p$ produced by  homotopical  linking}
 as the critical values of $F_p$  produced by  homotopical  linking.
\end{defn}

\begin{remark}
Theorem \ref{thm:homotopy-eigen} still hold if we use  eigenvalues produced by  homotopical  linking instead of homological
eigenvalue. The proof is also based on Lemma \ref{lemma:homology-eigen-increasing}.
\end{remark}

\subsection{Homological eigenvalues 
of $1$-Laplacian}
\label{sec:1-lap-simplicial-complex}

Based on the  discussions in Section \ref{sec:min-max-1-Lap}, we are able to estimate the variational eigenvalues of $\Delta_1$. To further  characterize the homological  eigenvalues of $\Delta_1$, we  borrow some ideas from PL Morse theory and discrete  Morse theory, which are shown in the following proof of Theorem \ref{thm:tri-link}.  

We first recall the simplicial complex $K_n$. Let $\mathcal{P}_2(V)=\{(A,B):A\cap B=\varnothing,A\cup B\ne\varnothing,A,B\subset V\}$ be the collection of all the pairs of disjoint subsets of $V$. 
Then, there is a natural partial order $\prec$ on $\mathcal{P}_2(V)$ defined as 
$(A,B)\prec (A',B')$ if $A\subset A'$ and $B\subset B'$. 
This partial order $\prec$ is actually  the inclusion order for set-pairs. 
The simplicial complex $K_n$ is defined as the order complex on  $\mathcal{P}_2(V)$ with respect to the inclusion order for set-pairs, i.e.,  the faces of $K_n$ refer to the inclusion chains  in $\mathcal{P}_2(V)$.  A natural geometric realization of $K_n$ is the following triangluation of the unit $l^\infty$-sphere $\{\vec x\in\R^n:\|\vec x\|_\infty=1\}$: 

For any $A\subset V$  and any permutation  $\sigma:\{1,\cdots,n\}\to \{1,\cdots,n\}$, we define the simplex
$$\triangle_{A,\sigma}=\mathrm{conv}\left(\{\vec1_A-\vec1_{V\setminus A}\}\cup\{\vec1_{A\setminus\{\sigma(1),\cdots,\sigma(i)\}}-\vec1_{(V\setminus A)\setminus\{\sigma(1),\cdots,\sigma(i)\}}:i=1,\ldots,n-1\}\right)$$ which is of dimension $(n-1)$. Then,  $\{\triangle_{A,\sigma}:A\subset V,\text{ permutation }\sigma\text{ on }V\}$ is the set of maximal simplices\footnote{We usually don't distinguish  the  simplicial complex $K_n$ and its geometric realization $|K_n|$.} of $K_n$, and we rewrite the vertices of $K_n$ as $\{\vec1_A-\vec1_B:(A,B)\in \mathcal{P}_2(V)\}$. 
 Similarly, for a given subset  $\A\subset\mathcal{P}_2(V)$, we regard $\A$ as a poset with the partial order $\prec$, i.e., for any  $(A,B),(A',B')\in\A$, $(A,B)\prec (A',B')$ if and only if  $A\subset A'$ and $B\subset B'$. 
The  subcomplex corresponding to  $\A$ is defined to be  the order complex of the  poset $(\A,\prec)$, and we work on its  natural geometric realization 
whose faces are determined by the geometric simplexes 
$\mathrm{conv}\{\vec1_{A_1}-\vec1_{B_1},\cdots,\vec1_{A_k}-\vec1_{B_k}\}$ for any chain $(A_1,B_1)
\prec\cdots\prec (A_k,B_k)$ in $(\A,\prec)$.

Given a vertex $\vec1_A-\vec1_{B}$ of $\A$, we shall use $\mathrm{star}(\vec1_A-\vec1_{B})$ to denote the \emph{closed star} of $\vec1_A-\vec1_{B}$, that is, the \textbf{closure}  of the union of the  simplices in $|K_n|$ that have $\vec1_A-\vec1_{B}$ as a vertex.  

We use $\mathrm{link}(\vec1_A-\vec1_{B})$ to denote the \emph{link} of $\vec1_A-\vec1_{B}$, that is, $\mathrm{star}(\vec1_A-\vec1_{B})\setminus \mathrm{star}^o(\vec1_A-\vec1_{B})$, where $ \mathrm{star}^o(\vec1_A-\vec1_{B})$ is just the union of the relatively open simplices in $|K_n|$ that have $\vec1_A-\vec1_{B}$ as a vertex.


\begin{proof}[Proof of Theorem \ref{thm:tri-link}]

Assume that the subcomplexes induced by the two sublevel sets $\{\vec x:F_1(\vec x)<\lambda\}$ and $\{\vec x:F_1(\vec x)\le \lambda\}$ have different homology groups. Then,  there exists a pair $(A,B)$ of disjoint subsets of $V=\{1,\cdots,n\}$ such that $F_1(\vec 1_A-\vec 1_B)=\lambda$, that is, $\lambda$ is the value of $F_1$ acting on some vertices of $K_n$. We shall prove that  $\lambda$ is a homological eigenvalue of $\Delta_1$.  Suppose the contrary, that $\lambda$ is not a homological eigenvalue of $\Delta_1$, i.e., $\lambda$ is not a homological critical value of $F_1$. 
Then, by the definition of homological  regular values, there exists $\epsilon> 0$ such that for each pair of real numbers $t_1<t_2$ on the interval $(\lambda-\epsilon,\lambda+\epsilon)$, the inclusion $\{F_1\le t_1\}\hookrightarrow \{F_1\le t_2\}$ induces isomorphisms on all homology groups. Particularly, for any  sufficiently small $\epsilon'>0$,  $\{F_1\le \lambda-\epsilon'\}\hookrightarrow \{F_1\le \lambda\}$ induces isomorphisms on all homology groups.  
It is clear that the subcomplex of $K_n$  induced by the vertices in $ \{\vec x\in\R^n:F_1(\vec x)< \lambda\}$ coincides  with the subcomplex of $K_n$  induced by the vertices in $ \{\vec x\in\R^n:F_1(\vec x)\le \lambda-\epsilon'\}$ exactly, when $\epsilon'>0$ is sufficiently small. 

Let $X_1=\{\vec x\in\R^n:\sum_{i=1}^n\deg(i)|x_i|=1\}$ and let $\widetilde{|K_n|}$ be the  
triangulation of $X_1$ whose  $(n-1)$-dimensional simplexes of $\widetilde{|K_n|}$ possess  the form 
\[\mathrm{conv}\left\{\frac{\vec1_{A_1}-\vec1_{B_1}}{\vol(A_1\cup B_1)},\cdots,\frac{\vec1_{A_n}-\vec1_{B_n}}{\vol(A_n\cup B_n)}
\right\}\]
where $(A_1,B_1)\prec \cdots\prec(A_n,B_n),\#(A_i\sqcup B_i)=i,(A_i,B_i)\in\power_2(V)$. 
Note that both $|K_n|$ and $\widetilde{|K_n|}$  are piecewise linear manifolds sharing the same simplicial complex structure, in which $\widetilde{|K_n|}$ is a triangulation of $X_1$ while $|K_n|$ is a triangulation of $X_\infty:=\{\vec x\in\R^n:\|\vec x\|_\infty=1\}$.  

Consider a map $r:X_1\to X_\infty$ defined as $r(\vec x)=\frac{\vec x}{\|\vec x\|_\infty}$. Then $r$ is a homeomorphism, and map each simplex of $\widetilde{|K_n|}$ to   a simplex of $|K_n|$ by sending each vertex $\frac{\vec1_{A}-\vec1_{B}}{\vol(A\cup B)}$ of  $\widetilde{|K_n|}$ to its corresponding vertex $\vec1_{A}-\vec1_{B}$ of $|K_n|$. 
Clearly, $r^{-1}:|K_n|\to\widetilde{|K_n|}$ is defined as  $r^{-1}(\vec x)=\frac{\vec x}{\sum_{i=1}^n \deg(i)|x_i|}$. 
Moreover, $F_1$ is piecewise linear on $X_1$, and it is actually linear on every simplex of $\widetilde{|K_n|}$. 
It follows from the zero-homogeneity of $F_1$ that $r(\{\vec x\in\widetilde{|K_n|}:F_1(\vec x)\le t\})=\{\vec x\in|K_n|:F_1(\vec x)\le t\}$ for any $t\in\R$. 

We should use the following argument:
\begin{Claim}[Proposition 2.25 in \cite{Bauer11}, see also K\"uhnel \cite{Kuhnel90} and Morozov  \cite{Morozov08}]\label{claim:PL-kuhnel}
Given a PL function $f^{PL}$ on a simplicial complex $|\K|$,  the induced subcomplex of $\K$ on $\{v\in \K_0: f^{PL}(v)\le t\}$ is homotopy equivalent  to the sublevel set $\{f^{PL}\le t\}$, where $\K_0$ is the vertex set of $\K$. 
\end{Claim}

Thus, we can apply K\"uhnel's theorem (i.e., Claim \ref{claim:PL-kuhnel}) to derive that the subcomplex of $K_n$  induced by the vertices in $ \{\vec x\in\widetilde{|K_n|}:F_1(\vec x)\le \lambda-\epsilon'\}$ is homotopy equivalent to the sublevel set $ \{\vec x\in\widetilde{|K_n|}:F_1(\vec x)\le \lambda-\epsilon'\}$. Furthermore, we have the commutative 
diagram: 
$$\xymatrix{\{\vec x\in\widetilde{|K_n|}:F_1(\vec x)\le \lambda-\epsilon'\}\ar[rr]^{h}&& \K|_{\{\vec x\in\widetilde{|K_n|}:F_1(\vec x)\le \lambda-\epsilon'\}}\ar[d]^{r}\\
\{\vec x\in|K_n|:F_1(\vec x)\le \lambda-\epsilon'\} \ar[rr]^{h'}\ar[u]^{r^{-1}}&&\K|_{\{\vec x\in|K_n|:F_1(\vec x)\le \lambda-\epsilon'\}}} $$
where $h$ is a continuous map such that $h\circ i\simeq id$ and $i\circ h\simeq id$, and \[i:\K|_{\{\vec x\in\widetilde{|K_n|}:F_1(\vec x)\le \lambda-\epsilon'\}}\hookrightarrow \{\vec x\in\widetilde{|K_n|}:F_1(\vec x)\le \lambda-\epsilon'\}\] is the inclusion map. 
Since $r$ is a homeomorphism, we have $h'\circ i'=(r\circ h\circ r^{-1})\circ(r\circ i\circ r^{-1}) = r\circ h\circ  i\circ r^{-1} \simeq id$, and similarly, there holds  $i'\circ h'\simeq id$, where 
\[i':\K|_{\{\vec x\in|K_n|:F_1(\vec x)\le \lambda-\epsilon'\}}\hookrightarrow\{\vec x\in|K_n|:F_1(\vec x)\le \lambda-\epsilon'\}\] is the inclusion map. 
Therefore, the subcomplex $\K|_{\{\vec x\in|K_n|:F_1(\vec x)\le \lambda-\epsilon'\}}$ is homotopy equivalent to the sublevel set $\{\vec x\in|K_n|:F_1(\vec x)\le \lambda-\epsilon'\}$.

In consequence, $$H_*(\K|_{F_1<\lambda})=H_*(\K|_{F_1\le\lambda-\epsilon'})=H_*(\{F_1\le\lambda-\epsilon'\})=H_*( \{F_1\le\lambda\})=H_*( \K|_{F_1\le\lambda})$$
which is a contradiction. In the above equalities,  we use  $\K|_{F_1<\lambda}$  to denote the subcomplex of $K_n$  induced by  the vertices in the lower level set  $\{F_1<\lambda\}$. 


For the converse part, assume that $\lambda$ is a homological eigenvalue. Then, it follows  from K\"uhnel’s Theorem  and  the definition of homological eigenvalue  that the inclusion $\K|_{F_1<\lambda}\hookrightarrow \K|_{F_1\le \lambda}$ does not induce isomorphisms on 
 homology groups. 

Now, we move on to the local case. 
Assume that 
there exists 
$A\ne\varnothing$ with $F_1(\vec 1_A)=\lambda$ and $F_1(\vec v)\ne\lambda$ for any vertex $\vec v$ in $ \mathrm{link}(\vec 1_A)$. Then locally,
the subcomplex of $K_n$ induced by the vertices in $ \{\vec x\in \mathrm{star}(\vec 1_A):F_1(\vec x)\le \lambda\}$ is a cone, and therefore it is contractible. It is easy to see that in this situation, the subcomplex of $K_n$  induced by the vertices in $ \{\vec x\in \mathrm{star}(\vec 1_A):F_1(\vec x)< \lambda\}$ coincides with the subcomplex of $K_n$  induced by the vertices in $ \{\vec x\in \mathrm{link}(\vec 1_A):F_1(\vec x)< \lambda\}$. 

Thus, if the subcomplex of $K_n$  induced by the vertices in $ \{\vec x\in \mathrm{link}(\vec 1_A):F_1(\vec x)< \lambda\}$ has non-vanishing reduced homology, 
then $$H_*(\K|_{\{F_1<\lambda\}\cap\, \mathrm{star}(\vec 1_A)})\ne H_*(\K|_{\{F_1\le\lambda\}\cap\, \mathrm{star}(\vec 1_A)})$$
and consequently, it is easy to see that $H_*(\K|_{F_1<\lambda})\ne H_*( \K|_{F_1\le\lambda})$, which implies that $\lambda$ is a homological  eigenvalue of $\Delta_1$.



Incidentally,  $\vec1_A$ is a PL critical point and $\lambda$ is the corresponding PL critical value of $F_1$  in the sense of Brehm and K\"uhnel
\cite{Brehm-Kuhnel} 
(or in the sense of Edelsbrunner \cite{EH10,EHNP03}). 
\end{proof}

\begin{remark}
A similar idea has been used to establish a relationship between Forman's discrete Morse
 theory and  Edelsbrunner's PL Morse theory  in a previous work of the author  \cite{JostZhang-Morse}. Moreover, we have studied the combinatorial  structure of the simplicial complex $K_n$  in another  work  \cite{JZ21}. 
\end{remark}

\subsection{Monotonicity of some functions   involving   eigenvalues of $\Delta_p$ with respect to $p$}
\label{sec:monotonicity}
\begin{proof}[Proof of Theorem \ref{thm:most}]
Given $p,q\ge 1$, we define an odd continuous map  $\Phi_{q/p}:\R^n\to\R^n$ by $$\Phi_{q/p}(\vec x)=(|x_1|^{\frac qp}\mathrm{sign}(x_1),\cdots,|x_n|^{\frac qp}\mathrm{sign}(x_n)).$$ 

\begin{Claim}\label{claim:1}
For any $1\le p\le q$, for any $\vec x\in\R^n$,
\begin{equation}\label{eq:p<q2}
F_p(\Phi_{q/p}(\vec x))\ge 2^{p-q}F_q(\vec x).    
\end{equation}
In addition, if $q<p$, then for any $\vec x$ with $F_q(\vec x)>0$, the inequality in \eqref{eq:p<q2} is strict. 
\end{Claim}
\textbf{Proof}. Denote by $\|\vec x\|_p=(\sum_{i=1}^n\deg(i)|x_i|^p)^{\frac1p}$, $\forall p\ge 1$. It is clear that $\|\Phi_{q/p}(\vec x)\|_p^p=\|\vec x\|_q^q$. Denote by $\mathrm{TV}_p(\vec x)=\sum_{\{i,j\}\in E}|x_i-x_j|^p$. Then 
\begin{align*}
\mathrm{TV}_p(\Phi_{q/p}(\vec x))&=\sum_{\{i,j\}\in E}\left||x_i|^{\frac qp}\mathrm{sign}(x_i)-|x_j|^{\frac qp}\mathrm{sign}(x_j)\right|^p
\\&\ge\sum_{\{i,j\}\in E}|x_i-x_j|^p\left(\left(\frac{|x_i|^{\frac qp}+|x_j|^{\frac qp}}{2}\right)^{1-\frac pq}\right)^p
\\&\ge \sum_{\{i,j\}\in E}|x_i-x_j|^p\left(\left(\frac{|x_i|+|x_j|}{2}\right)^{\frac qp(1-\frac pq)}\right)^p
\\&\ge \sum_{\{i,j\}\in E}|x_i-x_j|^p\left(\frac{|x_i-x_j|}{2}\right)^{q-p}
\\&= 2^{p-q}\sum_{\{i,j\}\in E}|x_i-x_j|^q=2^{p-q}\mathrm{TV}_q(\vec x) 
\end{align*}
where we used the inequality (see Lemma \ref{lem:elementary-inequality}) 
$$\left||b|^{\frac qp}\mathrm{sign}(b)-|a|^{\frac qp}\mathrm{sign}(a)\right|
\ge|b-a|\left(\frac{|b|^{\frac qp}+|a|^{\frac qp}}{2}\right)^{1-\frac pq}.$$
The above inequality is strict whenever $a\ne b$ and $q>p$. This implies that $\mathrm{TV}_p(\Phi_{q/p}(\vec x))>2^{p-q}\mathrm{TV}_q(\vec x)$ if $x_i\ne x_j$ for some  $\{i,j\}\in E$. 

Therefore,
$$F_p(\Phi_{q/p}(\vec x))=\frac{\mathrm{TV}_p(\Phi_{q/p}(\vec x))}{\|\Phi_{q/p}(\vec x)\|_p^p}\ge \frac{2^{p-q}\mathrm{TV}_q(\vec x)}{\|\vec x\|_q^q}=2^{p-q}F_q(\vec x),$$
and the inequality is strict whenever $q>p$ and $F_q(\vec x)>0$. 

The proof of the claim is completed.

Similarly, we have another inequality which has been  essentially shown in \cite{Amghibech}:

\begin{Claim}\label{claim:2} For any $1\le p\le q$, for any $\vec x\in\R^n$,
$$F_p(\Phi_{q/p}(\vec x))\le 2^{\frac pq-1}\left(\frac qp\right)^pF_q(\vec x)^{\frac pq}$$
or equivalently, \begin{equation}\label{eq:p<q1/p}
p(2F_p\circ\Phi_{q/p})^{\frac1p}\le q(2F_q)^{\frac1q}.
\end{equation}
Similarly,  the inequality is strict whenever $q>p$ and $F_q(\vec x)>0$. 
\end{Claim}
To prove Theorem \ref{thm:most},  we also need the following results:

\begin{Claim}\label{claim:homology-critical-coin} The homological critical values of $F_p$ coincide with that of $F_p\circ\Phi_{q/p}$.
\end{Claim}
Proof: Since $\Phi_{q/p}$ is a homeomorphism, it is clear that for any $c\in\R$, the topology of $\{F_p\le c\}$ and the topology of $\{F_p\circ\Phi_{q/p}\le c\}$ are the same.

\begin{Claim}\label{claim:homotopy-critical-coin} The  critical values of $F_p$ produced by homotopical linking  coincide with that of $F_p\circ\Phi_{q/p}$.
\end{Claim}
Proof: Let $Q\subset Q'\subset\R^n\setminus\{\vec0\}$ and $S\subset \R^n\setminus\{\vec0\}$ be such that $ Q$ homotopically links $S$  with respect to $Q'$, where $Q,Q',S$ are  
compact subsets. 
Suppose   $\min\limits_{\vec x\in S}F_p(\vec x)>\max\limits_{\vec x\in   Q}F_p(\vec x)$. Then, for any $q\ge 1$,  $\Phi_{q/p}^{-1}(Q)\subset \Phi_{q/p}^{-1}(Q')\subset\R^n\setminus\{\vec0\}$,  $\Phi_{q/p}^{-1}(S)\subset \R^n\setminus\{\vec0\}$, and $\Phi_{q/p}^{-1}(Q)$ homotopically links $\Phi_{q/p}^{-1}(S)$  with respect to $\Phi_{q/p}^{-1}(Q')$. Clearly, $\min\limits_{\vec x\in \Phi_{q/p}^{-1}(S)}F_p\circ\Phi_{q/p}(\vec x)=\min\limits_{\vec x\in S}F_p(\vec x)>\max\limits_{\vec x\in   Q}F_p(\vec x)=\max\limits_{\vec x\in   \Phi_{q/p}^{-1}(Q)}F_p\circ\Phi_{q/p}(\vec x)$. In consequence, 
$$\inf\limits_{\substack{\text{continuous }\tilde{\gamma}:\Phi_{q/p}^{-1}(Q')\to \R^n\setminus\{\vec0\}\\
\text{ with }\tilde{\gamma}|_Q=\mathrm{id}}}\max\limits_{\vec x\in\tilde{\gamma}(\Phi_{q/p}^{-1}(Q'))}F_p\circ\Phi_{q/p}(\vec x)=\inf\limits_{\substack{\text{continuous }\gamma:Q'\to \R^n\setminus\{\vec0\}\\
\text{ with }\gamma|_Q=\mathrm{id}}}\max\limits_{\vec x\in\gamma(Q')}F_p(\vec x)$$
indicates a  critical value produced by homotopical linking for both $F_p\circ\Phi_{q/p}$ and  $F_p$.

\begin{Claim}\label{claim:minmax-critical-coin} The min-max (variational) critical values of $F_p$ coincide with that of $F_p\circ\Phi_{q/p}$.
\end{Claim}
Proof: Since $\Phi_{q/p}^{-1}:\mathrm{Ind}_k(\R^n)\to \mathrm{Ind}_k(\R^n)$ is a bijection, for any $k=1,\cdots,n$, 
$$\inf_{S\in \mathrm{Ind}_k(\R^n)}\sup\limits_{\vec x\in S}F_p\circ\Phi_{q/p}(\vec x) = \inf_{S\in \mathrm{Ind}_k(\R^n)}\sup\limits_{\vec x\in \Phi_{q/p}^{-1}(S)}F_p\circ\Phi_{q/p}(\vec x) =\inf_{S\in \mathrm{Ind}_k(\R^n)}\sup\limits_{\vec x\in S}F_p(\vec x),$$
where $\mathrm{Ind}_k(\R^n):=\{S\subset \R^n\setminus\{\vec0\}:\gamma(S)\ge k\}$. This implies that the $k$-th min-max  critical value of $F_p\circ\Phi_{q/p}$ agrees with the $k$-th min-max  critical value of $F_p$. 
The same property holds when $\gamma$ is replaced by $\gamma^-$ or $\gamma^+$.

\begin{Claim}\label{claim:6} The  critical points of $F_p$, $p(2F_p)^{\frac1p}$ and $2^{-p}F_p$ are exactly the same, and their (min-max, or homological, or homotopical linking) critical values coincide up to certain scaling and power factors. Precisely, $\lambda$ is a critical value of $F_p$ if and only if $p(2\lambda)^{\frac1p}$ is a critical value of $p(2F_p)^{\frac1p}$ if and only if  $2^{-p}\lambda$ is a critical value of $2^{-p}F_p$
\end{Claim}
This claim is  easy to check, and thus we omit the proof. 

\vspace{0.1cm}

Together with all the above claims, the local monotonicity of the isolated homological eigenvalues  is then derived by 
 Lemma \ref{lemma:homology-eigen-increasing}. In fact, for any isolated homological  eigenvalue (or any eigenvalue  produced by   homotopical linking) $\lambda(\Delta_p)$ of the  $p$-Laplacian, by Claims \ref{claim:homology-critical-coin} and \ref{claim:homotopy-critical-coin},  $\lambda(\Delta_p)$ is an isolated homological  critical value (or a critical value  produced by   homotopical linking) of $F_p\circ\Phi_{q/p}$. 

Since $\lim\limits_{q\to p}\Phi_{q/p}=\mathrm{Id}$ for any $p\ge 1$, we have 
 for any sufficiently small  $\epsilon>0$, there exists $\delta>0$ such that for any $q\in(p-\delta,p+\delta)$, 
 $q(2F_q)^{\frac1q}-\epsilon\le p(2F_p\circ\Phi_{q/p})^{\frac1p}$ and $ 2^{-p}F_p\circ\Phi_{q/p}\le2^{-q}F_q+\epsilon$. By Claims  \ref{claim:1} and \ref{claim:2}, for any $q\in(p,p+\delta)$,  $ p(2F_p\circ\Phi_{q/p})^{\frac1p}\le q(2F_q)^{\frac1q}$ and $2^{-q}F_q\le 2^{-p}F_p\circ\Phi_{q/p}$. In consequence, we get 
$$q(2F_q)^{\frac1q}-\epsilon\le p(2F_p\circ\Phi_{q/p})^{\frac1p}\le q(2F_q)^{\frac1q}\;\text{ and }\;2^{-q}F_q\le 2^{-p}F_p\circ\Phi_{q/p}\le2^{-q}F_q+\epsilon.$$ 
 Then, it follows from 
 Lemma \ref{lemma:homology-eigen-increasing} and Claim  \ref{claim:6} that there is a homological  critical value (or a critical value  produced by   homotopical linking) $\lambda(\Delta_q)$ of $F_q$ satisfying $2^{-q}\lambda(\Delta_q)\le 2^{-p}\lambda(\Delta_p)\le2^{-q}\lambda(\Delta_q)+\epsilon$. The case of   $q(2\lambda(\Delta_q))^{\frac1q}-\epsilon\le p(2\lambda(\Delta_p))^{\frac1p}\le q(2\lambda(\Delta_q))^{\frac1q}$ is similar.

For the case of  min-max eigenvalues, for any $1\le p<q$, according to \eqref{eq:p<q2},  \eqref{eq:p<q1/p},  Claims \ref{claim:minmax-critical-coin}  and  \ref{claim:6}, we can similarly  
verify  that  
$p(2\lambda_k(\Delta_p))^{\frac1p}\le q(2\lambda_k(\Delta_q))^{\frac1q}$ and  $2^{-p}\lambda_k(\Delta_p)\ge 2^{-q}\lambda_k(\Delta_q)$, and the same inequalities hold when we consider $\lambda_k^\pm$ instead of $\lambda_k$. 
Moreover, for positive eigenvalues, these inequalities are  strict,  whenever $p\ne q$. We complete the whole  proof.
\end{proof}

\section{Eigenvalue estimates and refined Cheeger inequalities}
\label{sec:eigen-estimate}
\subsection{Variational eigenvalues of  $1$-Laplacian}
\label{sec:min-max-1-Lap}
 
Given a simple graph $G=(V,E)$,  let $P:=\{V_1,\cdots,V_k\}$ be a subpartition\footnote{A subpartition of $V$ is a family of pairwise disjoint subsets of $V$.} of $V$ such that each $V_i$ induces a complete subgraph.  We denote by $\mathcal{SC}(G)$ the collection of all these 
subpartitions.  
For any  $P\in \mathcal{SC}(G)$,  let $$h_*(P):=h_*(V_1,\cdots,V_k)=\min\limits_{A\subset \cup_{i=1}^kV_i:|A\cap V_i|\le 1}\frac{|\partial A|}{\vol(A)}$$
and  let $$c(P):=\sum_{i=1}^k c(V_i),\;\text{ where }c(V_i)=\begin{cases}1,&\text{ if }|V_i|\le 2,\\
2,&\text{ if }|V_i|\ge 3.
\end{cases}$$

\begin{lemma}\label{lemma:minmax-eigen-lower}
For any subpartition $P\in \mathcal{SC}(G)$, 
$$\lambda_{n-c(P)+1}^-(\Delta_1)\ge h_*(P).$$ 
\end{lemma}

\begin{proof}
We only need to work on a subpartition $P=\{V_1,\cdots,V_k\}$ such that each $V_i$ is either a singleton or a triangle. Here and after, we simply say that a subset $V_i$ is a \textbf{triangle} if it induces a three-order complete subgraph in $G$. 
In fact, for any subpartition $P=\{V_1,\cdots,V_k\}\in \mathcal{SC}(G)$, we take  $P'=\{V_1',\cdots,V_k'\}\in \mathcal{SC}(G)$ satisfying $V_i'\subset V_i$ and 
$$|V_i'|=\begin{cases}3,&\text{ if } |V_i|\ge 3\\
1,&\text{ if } |V_i|\in \{1,2\}. \end{cases}$$
 Then, we get a new subpartition
 $P'=\{V_1',\cdots,V_k'\}$ consisting of only singletons and  triangles. By the definition of $c(P)$, it is clear that $c(P')=c(P)$. 
 Also, note that $$\{A\subset \cup_{i=1}^kV_i:|A\cap V_i|\le 1,i=1,\cdots,k\}\supset \{A\subset \cup_{i=1}^kV_i':|A\cap V_i'|\le 1,i=1,\cdots,k\}$$
 which implies that $h_*(P)\le h_*(P')$. Hence, if we have   $\lambda_{n-c(P')+1}(\Delta_1)\ge h_*(P')$, then $\lambda_{n-c(P)+1}(\Delta_1)=\lambda_{n-c(P')+1}(\Delta_1)\ge h_*(P')\ge h_*(P)$. 
 
According to this fact, we may assume without loss of generality that each $V_i$ of the subpartition $P$ is a singleton 
or a triangle. Before proving Lemma \ref{lemma:minmax-eigen-lower}, we do some preparations. 
 

The {\sl hexagon structure}  corresponding to three linear independent  vectors $\vec a,\vec b,\vec c$ is a spatial hexagon defined as $${\Large\varhexagon}(\vec a,\vec b,\vec c):=[\vec a,-\vec b]\cup [-\vec b,\vec c]\cup [\vec c,-\vec a]\cup[ -\vec a,\vec b]\cup [\vec b,-\vec c]\cup [-\vec c,\vec a]$$
where $[\vec a,-\vec b]$ indicates the segment with the  endpoints $\vec a$ and $-\vec b$. 
The  geometric intuition for  such a hexagon  can be seen in Figure~\ref{fig:hexagon}.

\begin{figure}
\LARGE
\begin{center}
\begin{tikzpicture}[scale=0.6]
    {\draw[color=black,->] (0,0,0) -- (4,0,0) node[anchor=north east]{$\vec b$};}
    {\draw[color=black,->] (0,0,0) -- (0,4,0) node[anchor=north west]{$\vec c$};}
    {\draw[color=black,->] (0,0,0) -- (0,0, 4) node[anchor=south]{$\vec a$};}
\filldraw[red,fill opacity=0.03,line width=0,draw =red!35] (0,0,0)--(0,3+1,0)--(-3-1,0,0);
\filldraw[red,fill opacity=0.05,line width=0,draw =red!35] (0,0,0)--(-3-1,0,0)--(0,0,3+1);
\filldraw[red,fill opacity=0.07,line width=0,draw =red!35] (0,0,0)--(0,0,3+1)--(0,-3-1,0);
\filldraw[red,fill opacity=0.03,line width=0,draw =red!35] (0,0,0)--(0,-3-1,0)--(3+1,0,0);
\filldraw[red,fill opacity=0.05,line width=0,draw =red!35] (0,0,0)--(3+1,0,0)--(0,0,-3-1);
\filldraw[red,fill opacity=0.07,line width=0,draw =red!35] (0,0,0)--(0,0,-3-1)--(0,3+1,0);
\draw[red,thick] (4,0,0)--(0,-4,0)--(0,0,4)--(-4,0,0)--(0,4,0)--(0,0,-4)--(4,0,0);
\node (a) at (0,0, -4-0.6) {\tiny $-\vec a$};
\node (b) at (-4-0.5,-0.2,0) {\tiny $-\vec b$};
\node (c) at (-0.6,-4+0.3,0) {\tiny $-\vec c$};
\end{tikzpicture}
\end{center}
\caption{The hexagon $
{\Large\varhexagon}(\vec a,\vec b,\vec c)$  
is the union of the six red segments.
\label{fig:hexagon}}
\end{figure}

Let \begin{equation}\label{eq:S(V)}
S(V_i)=\begin{cases}
\{-\vec1_{\{v\}},\vec1_{\{v\}}\},&\text{ if }V_i=\{v\},\\
{\Large\varhexagon}(\vec 1_{\{u\}},\vec 1_{\{v\}},\vec 1_{\{w\}}),&\text{ if }V_i=\{u,v,w\},
\end{cases}    
\end{equation}
and let
$$S(P)=S(V_1)*S(V_2)*\cdots*S(V_k)$$
where $*$ denotes the simplicial join\footnote{The simplicial join of  subsets $A$ and $B$ in $\R^n$ is defined as  $A*B:=\{t\vec a+(1-t)\vec b:\vec a\in A,\vec b\in B,0\le t\le1\}$.}.  
We shall prove that $S(P)$ is  homeomorphic to  a sphere of dimension $c(P)-1$,  and $\inf\limits_{\vec x\in S(P)}F_1(\vec x)= h_*(P)$. 
 
\begin{Claim}\label{claim:F_1=h_*} $\min\limits_{\vec x\in S(P)}F_1(\vec x)= h_*(P)$. 
\end{Claim}

Proof: For any $\vec x\in\R^n\setminus\{\vec0\}$, there exist  $\vec x^+,\vec x^-\in \R^n_+$ such that $\vec x=\vec x^+-\vec x^-$, where $n=\#V$. It is easy to check that 
\begin{equation}\label{eq:F_1x+x-}
 F_1(\vec x)=\frac{\sum_{i\sim j}|x_i^+-x_j^+|+\sum_{i\sim j}|x_i^--x_j^-|}{\sum_{i\in V}\deg(j)|x_j^+|+\sum_{i\in V}\deg(j)|x_j^-|}\ge\min\{F_1(\vec x^+),F_1(\vec x^-)\}.   
\end{equation} 
Hence, $\min\limits_{\vec x\in S(P)}F_1(\vec x)= \min\limits_{\vec x\in S(P)\cap\R^n_+}F_1(\vec x)$. Note that $S(P)\cap\R^n_+=(S(V_1)\cap\R^n_+)*\cdots*(S(V_k)\cap\R^n_+)$ which is due to  the definition of $S(V_i)$. 
Therefore, by the definition of $S(V_i)$ in \eqref{eq:S(V)}, for any $\vec x\in S(P)\cap\R^n_+$, $\#(\mathrm{supp}(\vec x)\cap V_i)\le 1$, $\forall i$, where $\mathrm{supp}(\vec x):=\{i\in V:x_i\ne0\}$ is the support of $\vec x$. According to an elementary technique (see \cite{HS11,JZ21}), it can be verified that 
$F_1(\vec x)\ge \frac{|\partial A|}{\vol(A)}$ for some nonempty subset $A\subset \mathrm{supp}(\vec x)$. For readers' convenience, we present a brief  proof here. In fact, by the inequality \eqref{eq:F_1x+x-} and the relation $\mathrm{supp}(\vec x^+)\sqcup  \mathrm{supp}(\vec x^-)=\mathrm{supp}(\vec x)$, we can assume that $\vec x\in\R^n_+$. 
Then it can be verified that
\[F_1(\vec x)=\frac{\int_0^{\|\vec x\|_\infty}|\partial A_t|dt}{\int_0^{\|\vec x\|_\infty}\vol(A_t)dt}\ge \frac{|\partial A_{t_0}|}{\vol(A_{t_0})}\]
for some $0<t_0<\|\vec x\|_\infty$, 
where $A_t
:=\{i\in V:x_i>t\}$.  Hence, we can take such subset $A=A_{t_0}$ which is a subset of $\mathrm{supp}(\vec x)$. 
Thus, we have proved $\min_{\vec x\in S(P)}F_1(\vec x)\ge h_*(P)$. 

The opposite inequality is  easy to prove. Indeed, 
given $P:=\{V_1,\cdots,V_k\}$ and $A\subset \cup_{i=1}^kV_i$ such that $|A\cap V_i|\le 1$, $i=1,\cdots,k$, we can take $\vec x=\vec1_{A}\in S(P)$ and it is easy to check that $F_1(\vec 1_{A})= |\partial A|/\vol(A)$.   
The proof is completed. 

\begin{Claim}
There is an odd homeomorphism   maps  $S(P)$ to $ \mathbb{S}^{c(P)-1}$, 
where $\mathbb{S}^{c(P)-1}$  denotes the unit  sphere of dimension $c(P)-1$.

\end{Claim} 
Proof: 
Suppose that 
 $V_1,\cdots,V_l$ are triangles, and  $V_{l+1},\cdots,V_k$ are single point sets. Then, by the definition of $S(V_i)$ in   \eqref{eq:S(V)}, there exists an odd homeomorphism $\psi_i:S(V_i)\to \mathbb{S}^1\hookrightarrow \R^{V_i}$ if  $i\in\{1,\cdots,l\}$ and an odd homeomorphism $\psi_j:S(V_j)\to \mathbb{S}^0\hookrightarrow \R^{V_j}$ if  $j\in\{l+1,\cdots,k\}$, where $\mathbb{S}^1\hookrightarrow \R^{V_i}$ means that we put the image $\psi(S(V_i))=\mathbb{S}^1$ as a one-dimensional unit  sphere centered at the origin in $\R^{V_i}$. Note that we also have  $S(V_i)\subset \{\vec x\in \R^n:\mathrm{supp}(\vec x)\subset V_i\}$, and   $$\underbrace{\mathbb{S}^1*\cdots*\mathbb{S}^1}\limits_{l \text{ times}}*\underbrace{\mathbb{S}^0*\cdots*\mathbb{S}^0}\limits_{k-l \text{ times}}=\mathbb{S}^{2l+k-l-1}
=\mathbb{S}^{c(P)-1}.$$
Then, we can define a map  $\psi:S(P)\to \mathbb{S}^{c(P)-1}$ 
by  $\psi(\sum_{i=1}^k t_i\vec x^i):=\sum_{i=1}^k t_i\psi_i(\vec x^i)$, $\forall \vec x^i\in S(V_i)$, $0\le t_i\le 1$, $\sum_{i=1}^k t_i=1$, $i=1,\cdots,k$. It is easy to check that $\psi$ is  
an odd homeomorphism via  the diagram: 
$$\xymatrix{S(P)\ar[rrrrr]^{\psi}\ar@{=}[d]&&&&&\mathbb{S}^{c(P)-1}\ar@{=}[d]\\ S(V_1)*\cdots*S(V_l)*S(V_{l+1})*\cdots*S(V_k)\ar[rrrrr]^{\psi_1*\cdots*\psi_k}&&&&&\underbrace{\mathbb{S}^1*\cdots*\mathbb{S}^1}\limits_{l \text{ times}}*\underbrace{\mathbb{S}^0*\cdots*\mathbb{S}^0}\limits_{k-l \text{ times}} }$$

The proof is completed. 


\begin{Claim}\label{claim:nonempty-intersec}
For any centrally symmetric compact subset $S$ with $\gamma^-(S)
\ge n-c(P)+1$, we have $S\cap \{t\vec x:\vec x\in S(P),t\ge0\}\ne\varnothing$.
\end{Claim}

Proof: Continuing the preceding proof, it is not difficult to verify that any $\vec x\in\R^n$ has a unique decomposition $\vec x=\vec y+\sum_{i=1}^lt_i\vec1_{V_i}+\sum_{v\in V\setminus \cup_{i=1}^kV_i}t_v\vec1_{\{v\}}$ with  $t_i,t_v\in \R$, and $\vec y/\|\vec y\|_1\in S(P)$ if $\vec y\ne\vec0$. With the help of  such a decomposition, we then define a map  $\widetilde{\psi}:\R^n\to \R^n$ by
 $$\widetilde{\psi}(\vec x)=\|\vec y\|_1\psi(\frac{\vec y}{\|\vec y\|_1})+\sum_{i=1}^lt_i\vec\xi_{V_i}+\sum_{v\in V\setminus \cup_{i=1}^kV_i}t_v\vec1_v$$
where $\psi$ is the odd homeomorphism introduced in the preceding proof, and   $\vec\xi_i$ is a nonzero vector orthogonal to $\psi(S(V_i))$ in $\R^{V_i}$, $i=1,\cdots,l$. 

One can check that $\widetilde{\psi}:\R^n\to \R^n$ is an odd homeomorphism. Also, $\widetilde{\psi}$ is positively one-homogeneous, i.e., $\widetilde{\psi}(t\vec x)=t\widetilde{\psi}(\vec x)$, $\forall \vec x\in\R^n$,   $\forall t\ge0$. In particular, $\widetilde{\psi}(S(P))=\psi(S(P))=\mathbb{S}^{c(P)-1}$, and thus $\{t\vec x:\vec x\in \psi(S(P)),t\ge0\}=\mathrm{span}(\psi(S(P)))
= \R^{c(P)}$. For any centrally symmetric compact subset $S$ with $\mathrm{genus}(S)\ge n-c(P)+1$, 
$\widetilde{\psi}(S)$ is also a centrally symmetric compact subset with $\mathrm{genus}(\widetilde{\psi}(S))=\mathrm{genus}(S)\ge n-c(P)+1$. 
Finally, 
by the intersection property of Krasnoselskii genus, the subset  $\widetilde{\psi}(S)$ intersects with  the linear subspace $\{t\vec x:\vec x\in \psi(S(P)),t\ge0\}$. Accordingly, 
\begin{align*}
S\cap \{t\vec x:\vec x\in S(P),t\ge0\}&=\widetilde{\psi}^{-1}(\widetilde{\psi}(S))\cap \widetilde{\psi}^{-1} \{t\vec x:\vec x\in \widetilde{\psi}(S(P)),t\ge0\} 
\\&= \widetilde{\psi}^{-1}\left(\widetilde{\psi}(S)\cap  \{t\vec x:\vec x\in \psi(S(P)),t\ge0\}\right) 
\ne\varnothing
\end{align*}
which completes the proof.

\vspace{0.1cm} 

We are ready to prove  Lemma \ref{lemma:minmax-eigen-lower}. 
By the zero-homogeneity of $F_1$,  Claims \ref{claim:F_1=h_*}
and  \ref{claim:nonempty-intersec}, we have
\begin{align*}
\lambda_{n-c(P)+1}^-(\Delta_1)&=\inf_{\mathrm{genus}(S)\ge n-c(P)+1}\sup\limits_{\vec x\in S} F_1(\vec x)
\\&\ge
\inf_{\mathrm{genus}(S)\ge n-c(P)+1}\sup\limits_{\vec x\in S\cap \{t\vec x:\vec x\in S(P),t\ge0\}} F_1(\vec x)
\\&\ge\min\limits_{\vec x\in S(P)}F_1(\vec x)=h_*(P).
\end{align*}

We complete the whole proof.
\end{proof}

\begin{defn}
We say that some subsets $V_1,\ldots,V_k$ of $V$ are \textbf{pairwise non-adjacent} if every edge $\{i,j\}\in E$ intersects 
at most one $V_l$,  $l=1,\cdots,k$. 
\end{defn}

\begin{cor}\label{cor:independence-number}
Let $\mathcal{SC}_{na}(G)=\{P\in \mathcal{SC}(G):\text{the elements in }P\text{ are pairwise  non-adjacent}\}$, and denote by  $$\alpha_*(G)=\max\{c(P)|P\in \mathcal{SC}_{na}(G)\}
$$ 
the so-called pseudo-independence  number introduced in \cite{Zhang18}. Then  
$$\lambda_{n-\alpha_*(G)+1}^-(\Delta_1)=\cdots=\lambda_{n}^-(\Delta_1)=1.$$
\end{cor}

\begin{proof}
For any $P=(V_1,\cdots,V_k)\in \mathcal{SC}_{na}(G)$, $V_1,\cdots,V_k$ are pairwise  non-adjacent. Thus, for any $A\subset \sqcup_{i=1}^kV_i$ with $\#(A\cap V_i)\le 1$, $A$ is the disjoint union of some pairwise  non-adjacent vertices, which implies that   $|\partial A|=\vol(A)$. Therefore, we have  $h_*(P)=1$. In particular, taking   $P'\in\mathcal{SC}_{na}(G)$  such that $c(P')=\alpha_*(G)$, we have $h_*(P')=1$.  By Lemma \ref{lemma:minmax-eigen-lower}, $\lambda_{n-c(P')+1}^-(\Delta_1)\ge h_*(P')=1$. On the other hand, it is known that $\lambda_{n-c(P')+1}^-(\Delta_1)\le\cdots\le  \lambda_{n}^-(\Delta_1)=1$. Hence, we have  $\lambda_{n-\alpha_*(G)+1}^-(\Delta_1)=\lambda_{n-c(P')+1}^-(\Delta_1)=\cdots= \lambda_{n}^-(\Delta_1)=1$. 
\end{proof}

A set of vertices of a graph is independent if the vertices are pairwise nonadjacent. The \emph{independence number} of a graph  is the cardinality of its largest independent set.

Corollary \ref{cor:independence-number}  generalizes and enhances a result in \cite{Zhang18}. 
 Moreover, we have a stronger version of the main theorem in \cite{Zhang18}: 
\begin{pro}
Let $\alpha(G)$ be the independence number of $G$, and let $\gamma(G)$ be the multiplicity of the eigenvalue $1$ of $\Delta_1$. Denote by $t(G)$ the largest integer  satisfying $\lambda_{n-t(G)+1}^-(\Delta_1)=1$. Then,  
$$\alpha(G)\le \alpha_*(G)\le t(G)\le \gamma(G)\le \min\{c_*(G) ,2\alpha(G)\}$$
where 
$c_*(G):=\min\{c(P)|P\in \mathcal{SC}_p(G)\}$ and  $\mathcal{SC}_p(G)=\{P\in \mathcal{SC}(G):P\text{ is a partition of }V\}$.
\end{pro}

\begin{proof}
The inequalities  $\alpha(G)\le \alpha_*(G)$, $t(G)\le \gamma(G)$ and $\gamma(G)\le   2\alpha(G)$ have been established in Theorem 1 
in \cite{Zhang18}. The inequality $\gamma(G)\le   c_*(G)$ is equivalent to  Corollary 1 in \cite{Zhang18}. It remains to show $\alpha_*(G)\le t(G)$. 

By  Corollary \ref{cor:independence-number}, $
\lambda_{n-\alpha_*(G)+1}^-(\Delta_1)=1$.  Since $t(G)$ is the largest number such that $\lambda_{n-t(G)+1}^-(\Delta_1)=1$, we have $t(G)\ge\alpha_*(G)$. 
 The proof is completed. 
\end{proof}

\subsection{Refined multi-way Cheeger inequalities}
\label{sec:Cheeger-in}

In this section, we shall prove Theorem  \ref{thm:p-lap-C}. Before proving this theorem, we collect here some useful claims.

\begin{Claim}[Theorem 8 in \cite{CSZ17}]\label{claim:CSZ}
Let $\vec x$ be an eigenvector  corresponding to an eigenvalue $\lambda$ of $\Delta_1$ with   $\lambda\le\lambda_k^-(\Delta_1)$, and assume that $\vec x$ has $m$ (strong)  nodal domains.  Then, 
$$h_{m}\le \lambda_k^-(\Delta_1)\le h_k.$$
The above inequality still holds when we use $\lambda_k(\Delta_1)$ or $\lambda_k^+(\Delta_1)$ instead of $\lambda_k^-(\Delta_1)$ with the same proof as  that of Theorem 8 in \cite{CSZ17}. 
\end{Claim}

\begin{Claim}[
Theorem 5.1 in \cite{TudiscoHein18}]\label{claim:THp>1}
For $p>1$, let $\vec x$ be an eigenvector  corresponding to the eigenvalue $\lambda_k^-(\Delta_p)$, and assume that $\vec x$ has $m$ strong nodal domains.  Then,  $$ \frac{2^{p-1}}{p^p}h_m^p\le\lambda_k^-(\Delta_p)\le 2^{p-1}h_k.$$
Moreover, if there is an eigenpair $(\lambda,\vec x)$ of $\Delta_p$ with $\lambda\le \lambda_k^-(\Delta_p)$ and $\mathfrak{S}(\vec  x)=m$, then we still have $\frac{2^{p-1}}{p^p}h_m^p\le\lambda_k^-(\Delta_p)$, and this inequality 
 holds when we use $\lambda_k(\Delta_1)$ or $\lambda_k^+(\Delta_1)$ instead of $\lambda_k^-(\Delta_1)$. The proofs of these versions of  $\lambda_k(\Delta_1)$ and $\lambda_k^+(\Delta_1)$ are the same  as  that of Theorem 5.1 in \cite{TudiscoHein18}. 
\end{Claim}

\begin{Claim}[Theorem 1.1 in \cite{LGT12}]\label{claim:multi-way-Chee}
For every graph, and each natural number $k$, 
\begin{equation*}
\frac{h^2_k}{Ck^4}\le \lambda_k(\Delta_2)\le 2h_k
\end{equation*}   
where $C$ is a universal constant. 
\end{Claim}

\noindent\textbf{Convention}: We say that  the index of a subcomplex of  $K_n$ is $k$, if the geometric realization of such subcomplex in $|K_n|$ is centrally symmetric and its index  is $k$, where the word `index' can be Yang index,  Krasnoselskii genus or the $\gamma^+$-index.

\begin{Claim}\label{claim:comb-topo-min-max-1lap}
The $k$-th min-max eigenvalue of graph 1-Laplacian has the following  combinatorial characterization:
\begin{equation}\label{eq:lambda_k-combina}
\lambda_k(\Delta_1)=\hat{h}_k=\min\limits_{\A\in \mathcal{S}_k}\max\limits_{(A,B)\in\A}\frac{|\partial A|+|\partial B|}{\vol(A\cup B)}
\end{equation}
where  $\mathcal{S}_k:=\{\A\subset \power_2(V):\text{the Yang index of the  subcomplex of }K_n\text{ induced by  }\A\text{ is at least }k\}$. 
\end{Claim}

\begin{proof}

We shall 
prove a  general statement:  for any  odd piecewise linear function $F:|K_n|\to\R$ (resp.,  $F:\widetilde{|K_n|}\to\R$) that is linear on each simplex of $|K_n|$ (resp.,  $\widetilde{|K_n|}$), we have
\begin{equation}\label{eq:infsup=minmax}
\inf\limits_{\gamma(S)\ge k}\sup\limits_{\vec x\in S}F(\vec x)=\min\limits_{\A\in \mathcal{S}_k}\max\limits_{(A,B)\in\A}F(\vec1_A-\vec1_B)
,\end{equation}
where $\widetilde{|K_n|}$ is another geometric realization of $K_n$ and is also  a  triangulation of  $X_1:=\{\vec x\in\R^n:\sum_{i=1}^n\deg(i)|x_i|=1\}$. We refer  the reader to the proof of Theorem \ref{thm:tri-link} for the construction of $\widetilde{|K_n|}$.

To show the proof of the above statement, 
we can assume without loss of generality that $F$ is injective on the vertices of any  simplex of $|K_n|$, that is, $F(\vec v)\ne F(\vec u)$ whenever $\vec v$ and $\vec u$ are two different vertices of a simplex in $|K_n|$. 
In fact, if \eqref{eq:infsup=minmax} holds for all such `injective' $F$, it holds also neglecting the injectivity. The reason is as follows. In fact, let $\mathcal{F}:=\{\text{odd continuous function }F:|K_n|\to\R\text{ that is linear on each simplex of }|K_n|\}$ and $\mathcal{F}_{in}:=\{F\in\mathcal{F}\left|\;F\text{ is injective on the vertices of }|K_n|\right.\}$. Clearly, $\mathcal{F}_{in}$ is an open dense subset of $\mathcal{F}$ where we use the topology induced by the maximum norm $\|\cdot\|_\infty$. 
Note that the min-max quantities on both side of \eqref{eq:infsup=minmax} are preserved under uniform convergence, that is, if $f_n\in \mathcal{F}$ and $f_n\rightrightarrows f$ as $n\to+\infty$, then 
\[\lim_{n\to\infty}\inf\limits_{\gamma(S)\ge k}\sup\limits_{\vec x\in S}f_n(\vec x)=\inf\limits_{\gamma(S)\ge k}\sup\limits_{\vec x\in S}f(\vec x)\;\text{ and }\;\lim_{n\to\infty}\min\limits_{\A\in \mathcal{S}_k}\max\limits_{(A,B)\in\A}f_n(\vec1_A-\vec1_B)=\min\limits_{\A\in \mathcal{S}_k}\max\limits_{(A,B)\in\A}f(\vec1_A-\vec1_B).\]
Thus, to prove the equality \eqref{eq:infsup=minmax}
 for any $F\in\mathcal{F}$, it suffices to prove \eqref{eq:infsup=minmax}
 for any $F\in\mathcal{F}_{in}$.

Now, given a function $F\in\mathcal{F}_{in}$, 
for any centrally symmetric  compact subset $S\subset|K_n|$ with $\gamma(S)\ge k$, denote by $S_F$ the set of the maximizers of $F|_S$ (i.e., the restriction of $F$ on $S$).   
According to the linearity of $F$ on each simplex of $|K_n|$, and the injectivity of $F$ on the vertex set of any  simplex of $|K_n|$, it is not difficult to show that the critical points of $F$ must be vertices of $|K_n|$. 
So, if $S_F$ doesn't contain any   vertex of $|K_n|$,  there is no critical point of $F$ in $S_F$. 
Thus, by the deformation theorem in nonlinear analysis, we can take a small perturbation $S'$ of $S$ such that thay are odd-homotopy equivalent 
and  
$$\sup\limits_{\vec y\in S'}F(\vec y)<F(\vec s)=\sup\limits_{\vec x\in S}F(\vec x).$$
At this time, we also have $\gamma(S')\ge k$.  

Therefore, we only need to consider such $S$ with the additional property that $S_F$ contains some vertices of $|K_n|$. In other words, we assume that 
$\max_{\vec x\in S} F(\vec x)$ is achieved at some vertex $\vec v$  of $|K_n|$, i.e., 
$F(\vec v)\ge F(\vec x)$, $\forall \vec x\in S$. Let $\lambda=F(\vec v)$.  Consider the sublevel set $\{F\le \lambda\}$. It is clear that $\gamma(\{F\le \lambda\})\ge \gamma(S)\ge k$ and 
$$\max\limits_{\vec x\in \{F\le \lambda\}}F(\vec x)=\lambda=F(\vec v)=\max\limits_{\vec x\in  S}F(\vec x).$$

\vspace{0.06cm}

By Claim \ref{claim:PL-kuhnel}, for any $c\in\R$,  there is a homotopy equivalence between the sublevel set $\{F\le c\}$ and  the induced  subcomplex  $K_n|_{\{F\le c\}}$ of $|K_n|$, where $K_n|_{\{F\le c\}}$ 
denotes the induced closed subcomplex of $|K_n|$ on the vertices lying in the sublevel set $\{F\le c\}$.   
Thus,  $\gamma(\{F\le c\})=\gamma(K_n|_{\{F\le c\}})$. In consequence, we have the following identities
\begin{align*}
\inf\limits_{S\in\mathrm{Ind}_k}\sup\limits_{\vec x\in S} F(\vec x)&=\inf\limits_{c\in\R\text{ s.t.}\,\{F\le c\})\in\mathrm{Ind}_k}\;\sup\limits_{\vec x\in \{F\le c\}} F(\vec x)
\\&=\min\limits_{c\in\R\text{ s.t.}\,K_n|_{\{F\le c\}}\in \mathrm{Ind}_k}\max\limits_{\text{ vertex }\vec v\text{ of }K_n|_{\{F\le c\}}}F(\vec v)
\\&=\min\limits_{\A\in \mathcal{S}_k}\max\limits_{(A,B)\in\A} F(\vec1_A-\vec1_B),
\end{align*}
where $\mathrm{Ind}_k:=\{S\subset |K_n|:\gamma(S)\ge k\}$, and in the last equality,  we identify  the collection of the induced subcomplexes in $\mathrm{Ind}_k$  with  $\mathcal{S}_k$,  because any vertex $\vec v$ of $K_n$  is in the form of  $\vec1_A-\vec1_B$ which can be   identified with the set-pair   $(A,B)\in\mathcal{P}_2(V)$. Also, we have used the fact that   on each induced subcomplex of $|K_n|$,  $F$ reaches the  maximum at some vertices. The proof of \eqref{eq:infsup=minmax}  is then completed.

Now, taking $F=F_1$ in the above equality \eqref{eq:infsup=minmax}, it is easy to  check that
$$F_1(\vec1_A-\vec1_B)=\frac{|\partial A|+|\partial B|}{\vol(A\cup B)},$$ 
which implies 
$$\lambda_k(\Delta_1)= \inf\limits_{S\in \mathrm{Ind}_k}\sup\limits_{\vec x\in S} F_1(\vec x)=\min\limits_{\A\in \mathcal{S}_k}\max\limits_{(A,B)\in\A}\frac{|\partial A|+|\partial B|}{\vol(A\cup B)}=\hat{h}_k.$$
The proof of \eqref{eq:lambda_k-combina} is completed.
\end{proof}

\begin{remark}
The above proof of Claim \ref{claim:comb-topo-min-max-1lap} indeed uses
a similar approach developed in the author's previous work \cite{JostZhang-Morse}. 
Also, we should note that any $S\in\mathrm{Ind}_k$ realizing  $c_k:=\inf_{S\in\mathrm{Ind}_k}\sup_{\vec x\in S}F(\vec x)$ 
contains a critical point of $F$  corresponding to the critical value $c_k=\max_{\vec x\in S}F(\vec x)$. 
Moreover, we can similarly prove that 
\begin{equation}\label{eq:lambda_k-combina+-}
\lambda_k^\pm(\Delta_1)=\hat{h}_k^\pm=\min\limits_{\A\in \mathcal{S}_k^\pm}\max\limits_{(A,B)\in\A}\frac{|\partial A|+|\partial B|}{\vol(A\cup B)}   
\end{equation}
where  $\mathcal{S}_k^\pm:=\{\A\subset \power_2(V):\gamma^\pm(\text{the  subcomplex of }K_n\text{ induced by  }\A)\ge k\}$, as both $\gamma^-$ and $\gamma^+$ are 
homotopy invariants. 
\end{remark}

\begin{proof}[Proof of Theorem  \ref{thm:p-lap-C}] The inequality $\lambda_k^-(\Delta_p) \le\lambda_k(\Delta_p)  \le\lambda_k^+(\Delta_p)$ has been shown in  \eqref{eq:inequality-Yang+}. Hence, $\mathsf{S}_{k}^-(\Delta_q)\subset \mathsf{S}_{k}(\Delta_q)\subset \mathsf{S}_{k}^+(\Delta_q)$ and therefore $s_k^-\le s_k\le s_k^+$. 
Analogously, the relation $\gamma^+\le\gamma\le \gamma^-$ simply implies $\mathcal{S}_k^+\subset\mathcal{S}_k\subset\mathcal{S}_k^- $ and thus $\hat{h}_k^-\le\hat{h}_k\le \hat{h}_k^+$.

  By Theorem \ref{thm:most}, the function 
$p\mapsto2^{-p}\lambda_k(\Delta_p)$ is decreasing on $[1,+\infty)$, which implies $2^{-1}\lambda_k(\Delta_1)\ge 2^{-p}\lambda_k(\Delta_p)$, $\forall p\ge 1$. So, we have $\lambda_k(\Delta_p)\le 2^{p-1}\lambda_k(\Delta_1)$. 
Together with  $\hat{h}_k=\lambda_k(\Delta_1)$ by Claim \ref{claim:comb-topo-min-max-1lap},  and   $\lambda_k(\Delta_1)\le h_k$ by Claim \ref{claim:CSZ}, we obtain the upper bound estimate:   \[\lambda_k(\Delta_p)\le 2^{p-1}\hat{h}_k\le 2^{p-1}h_k.\] 

Let us move on to the lower bound estimate. By  Theorem \ref{thm:most}, the function  $p\mapsto p(2\lambda_k(\Delta_p))^{\frac1p}$ 
is increasing  on $[1,+\infty)$, which yields $2\lambda_k(\Delta_1)\le p(2\lambda_k(\Delta_p))^{\frac1p}$ for any $p\ge 1$. Hence, $\frac{2^{p-1}}{p^p} \lambda_k(\Delta_1)^p\le \lambda_k(\Delta_p)$. Since $\hat{h}_k=\lambda_k(\Delta_1)$, we get the  inequality   $\lambda_k(\Delta_p)\ge \frac{2^{p-1}}{p^p}\hat{h}_k^p$. 

Let $\vec x$ be an eigenvector  corresponding to  some eigenvalue $\lambda$ of $\Delta_1$ with   $\lambda\le\lambda_k(\Delta_1)$, and assume that $\vec x$ has $m$  strong  nodal domains. Then, by Claims \ref{claim:CSZ} and \ref{claim:comb-topo-min-max-1lap}, $\hat{h}_k=\lambda_k(\Delta_1)\ge h_m$. Therefore, we have  \begin{equation}
\label{eq:lower-bound-1-Lap}\lambda_k(\Delta_p)\ge \frac{2^{p-1}}{p^p}\hat{h}_k^p\ge \frac{2^{p-1}}{p^p}h_m^p.
\end{equation}

For $p\ge q>1$, Theorem \ref{thm:most} also implies  $p(2\lambda_k(\Delta_p))^{\frac1p}\ge q(2\lambda_k(\Delta_q))^{\frac1q}$, which can be reformulated as $\lambda_k(\Delta_p)\ge 2^{\frac pq-1}(\frac qp)^p\lambda_k(\Delta_q)^{\frac pq}$ for any $q\le p$. 
Let $\vec x$ be an eigenvector  corresponding to an   eigenvalue $\lambda$ of $\Delta_q$ with   $\lambda\le\lambda_k(\Delta_q)$, and let $m$ be the number of the strong nodal domains of  $\vec x$. Then, by Claim \ref{claim:THp>1},  $\lambda_k(\Delta_q)\ge \frac{2^{q-1}}{q^q}h_m^q$. Thus,  we have \begin{equation}\label{eq:lower-bound-q-Lap}
\lambda_k(\Delta_p)\ge 2^{\frac pq-1}\left(\frac qp\right)^p\lambda_k(\Delta_q)^{\frac pq}\ge 2^{\frac pq-1}\left(\frac qp\right)^p\left(\frac{2^{q-1}}{q^q}h_m^q\right)^{\frac pq}=\frac{2^{p-1}}{p^p}h_m^p.    
\end{equation}

Since $h_1\le h_2\le \cdots\le h_n$, we may take $s_k$ to be the largest $m$ such that   $m$ is the   number of the strong  nodal domains of some eigenvector corresponding to some eigenvalue less than or equal to  $\lambda_k(\Delta_q)$ for some $q\in[1,p]$. Then, we simply have  $\lambda_k(\Delta_p)\ge  \frac{2^{p-1}}{p^p}h_{s_k}^p$ by the inequalities \eqref{eq:lower-bound-1-Lap} and \eqref{eq:lower-bound-q-Lap}. The analogous  inequalities on   $\lambda_k^-(\Delta_p)$ and $\lambda_k^+(\Delta_p)$ can be derived in the same way. 



\vspace{0.2cm}

Note that for the 2-Laplacian, we always have $\lambda_k^-(\Delta_2)=\lambda_k(\Delta_2)=\lambda_k^+(\Delta_2)$, $\forall k$. To complete the whole proof of Theorem  \ref{thm:p-lap-C},  
we should use 
the  multi-way Cheeger inequality proposed in \cite{LGT12} (see  Claim \ref{claim:multi-way-Chee}): 
 for every graph, and each natural number $k$, 
\begin{equation*}
\lambda_k^-(\Delta_2)=\lambda_k(\Delta_2)\ge \frac{h^2_k}{\overline{C}k^4}
\end{equation*}   
where $\overline{C}$ is a universal constant. 

Again, using Theorem \ref{thm:most}, we have the following simple estimates:
\begin{enumerate}[{Case}  1.] 
\item $1\le p\le 2$

In this case, we use the fact that the function 
$p\mapsto2^{-p}\lambda_k^-(\Delta_p)$ is decreasing. This implies  
$2^{-p}\lambda_k^-(\Delta_p)\ge 2^{-2}\lambda_k^-(\Delta_2)$. Thus,  $\lambda_k^-(\Delta_p)\ge 2^{p-2}\lambda_k^-(\Delta_2)\ge 2^{p-2}\frac{h^2_k}{\overline{C}k^4}\ge \frac{h^2_k}{2\overline{C}k^4}$. Taking $C=2\overline{C}$, we have 
$\lambda_k^-(\Delta_p)\ge h^2_k/Ck^4$.
\item $p\ge 2$

In this case, we use the fact that the function  $p\mapsto p(2\lambda_k^-(\Delta_p))^{\frac1p}$ 
is increasing. This implies  $p(2\lambda_k^-(\Delta_p))^{\frac1p}\ge 2(2\lambda_k^-(\Delta_2))^{\frac12}$. Hence, $$\lambda_k^-(\Delta_p)\ge 2^{\frac p2-1}\left(\frac 2p\right)^p\lambda_k^-(\Delta_2)^{\frac p2}\ge 2^{\frac p2-1}\left(\frac 2p\right)^p \left(\frac{h^2_k}{\overline{C}k^4}\right)^{\frac p2}=\frac{h^p_k}{2^{1-\frac32p}p^p\overline{C}^{\frac p2} k^{2p}}.$$
Taking $C_p=2^{1-\frac32p}p^p\overline{C}^{\frac p2}$, we have $\lambda_k^-(\Delta_p)\ge\frac{h^p_k}{C_pk^{2p}}$.
\end{enumerate}
In summary, we obtain
$$
\lambda_k^-(\Delta_p)\ge \begin{cases}\frac{h^2_k}{Ck^4},&\text{ if } 1\le p\le 2,\\
\frac{h^p_k}{C_pk^{2p}},&\text{ if } p\ge 2.
\end{cases}$$
We have completed the whole proof.
\end{proof}


\begin{proof}[Proof of Corollary \ref{cor:=}
] The $\lambda_k^-(\Delta_1)=h_k$ has been shown in \cite{DFT21}, due to the nodal domain estimates on forests. 
The proof of the  equality $\lambda^-_2(\Delta_p)=\lambda^+_2(\Delta_p)$ is standard, due to the mountain pass characterization, and thus we omit it. 

Together with Theorem  \ref{thm:p-lap-C} and Corollary \ref{cor:independence-number}, we complete the proof of these equalities.
\end{proof}

\begin{remark}\label{rem:Cheeger-p-1}
We note that Cheeger-type inequalities  essentially reflect   the connections between the spectra of $\Delta_p$ and $\Delta_1$. For example, the usual Cheeger inequality on graphs is nothing but a  relation 
between the principal  eigenvalues of  $\Delta_2$ and $\Delta_1$.  Lee-Oveis Gharan-Trevisan's multi-way Cheeger  inequality reveals    a certain   relationship between the higher-order  eigenvalues of $\Delta_2$ and $\Delta_1$.  Tudisco-Hein's higher-order Cheeger inequality for the graph $p$-Laplacian   establishes some estimates between the higher-order variational  eigenvalues of $\Delta_p$  and $\Delta_1$. In addition,  the monotonicity problem proposed by Amghibech is intended to 
find a comparison theorem for the eigenvalues of $\Delta_p$  and $\Delta_q$, for any given $p,q>1$. 
\end{remark}

\begin{remark}\label{rem:p-evolution}
The graph 2-Laplacian is a linear operator, while 
the graph 1-Laplacian (resp., graph $\infty$-Laplacian\footnote{The graph $\infty$-Laplacian  $\Delta_\infty$ is defined as $\Delta_\infty\vec x=\partial \max_{\{i,j\}\in E}|x_i-x_j|$, where $\partial$ indicates the Clarke subgradient.   
We do not  study $\infty$-Laplacian 
in this paper, but we would like to present here that $\Delta_\infty$ can also be  seen as a limit of $\Delta_p$, as $p\to+\infty$. Precisely, $\Delta_\infty\vec x=\lim\limits_{\substack{\hat{\vec x}\to\vec x\\ p\to+\infty}}(\sum_{\{i,j\}\in E}|\hat{x}_i-\hat{x}_j|^p)^{\frac1p-1}\Delta_p\hat{\vec x}$.}) can be viewed as a combinatorial operator. It is interesting that we can regard 
the graph $p$-Laplacian as a  non-linear evolution from the linear case  (i.e., 2-Laplacian) to the combinatorial case (i.e.,  1-Laplacian and  $\infty$-Laplacian). 
\end{remark}

\subsection{Distribution of eigenvalues for  $p$-Laplacians}\label{sec:distribution-gap}

In this section, we shall prove Theorem  \ref{thm:combina-estimate}. 
\begin{proof}[Proof of Theorem  \ref{thm:combina-estimate}]
By Corollary \ref{cor:independence-number}, $\lambda_{n-\alpha_*(G)+1}(\Delta_1)=\cdots=\lambda_{n}(\Delta_1)=\lambda_{n-\alpha_*(G)+1}^-(\Delta_1)=\cdots=\lambda_{n}^-(\Delta_1)=1$, which implies that there are at least  $\alpha_*(G)$ min-max eigenvalues of $\Delta_1$ equal to $1$. Here,  $\alpha_*(G)$ is the pseudo-independence  number introduced in \cite{Zhang18} (see Corollary \ref{cor:independence-number} in  Section \ref{sec:min-max-1-Lap} for the definition). 

Then, by Theorem \ref{thm:most} (or Theorem  \ref{thm:p-lap-C}), for $p>1$, 
$$\lambda_n(\Delta_p)\ge \cdots\ge \lambda_{n-\alpha_*(G)+1}(\Delta_p)> \frac{2^{p-1}}{p^p}$$
meaning that there are at least  $\alpha_*(G)$ eigenvalues (counting multiplicity) of $\Delta_p$  larger than
$\frac{2^{p-1}}{p^p}$ when $p>1$. 

To show  the spectral gap for  min-max eigenvalues of $\Delta_p$ on connected graphs, we recall the following surprising  result  on the largest Laplacian spectral gap from 1:

\begin{Claim}[\cite{JMZ21}]\label{thm:JMZ}
For any connected graph  on $n\geq 3$ nodes,
\begin{equation*}
  \min\limits_{1\le k\le n}|\lambda_k(\Delta_2)-1| \leq \frac{1}{2}.
\end{equation*}

\end{Claim}

By Claim  \ref{thm:JMZ}
, for any  connected graph with $n\geq 3$ vertices, there exists $l\in\{1,\cdots,n\}$ such that 
$\frac12\le\lambda_l(\Delta_2)\le\frac32$. By employing Theorem \ref{thm:most}, we simply  have: 
$$
\lambda_l(\Delta_p)>\begin{cases}2^{p-2}\lambda_l(\Delta_2)\ge 2^{p-3},&\text{ if } 1\le p< 2,\\
2^{\frac p2-1}(\frac 2p)^p\lambda_l(\Delta_2)^{\frac p2}\ge 2^{\frac p2-1}(\frac 2p)^p(\frac12)^{\frac p2}=\frac{2^{p-1}}{p^p},&\text{ if } p> 2,
\end{cases}$$
and 
$$
\lambda_l(\Delta_p)<\begin{cases}2^{\frac p2-1}(\frac 2p)^p\lambda_l(\Delta_2)^{\frac p2}\le 2^{\frac p2-1}(\frac 2p)^p(\frac32)^{\frac p2}=3^{\frac p2}\cdot\frac{2^{p-1}}{p^p},&\text{ if } 1\le p< 2,\\
2^{p-2}\lambda_l(\Delta_2)\le 3\times2^{p-3},&\text{ if } p> 2.
\end{cases}$$
That is, $2^{p-3}<\lambda_l(\Delta_p)<3^{\frac p2}\cdot\frac{2^{p-1}}{p^p}$ if $1\le p<2$, and $\frac{2^{p-1}}{p^p}<\lambda_l(\Delta_p)<3\cdot 2^{p-3}$ if $p>2$. This  completes the proof of Theorem  \ref{thm:combina-estimate}. \end{proof}

\section{Non-variational  eigenvalues}
\label{sec:nonvariational}
\subsection{Homological non-variational  eigenvalues 
}
\label{sec:verification}


For readers' convenience, we redraw the picture shown in Theorem \ref{thm:main} again: 
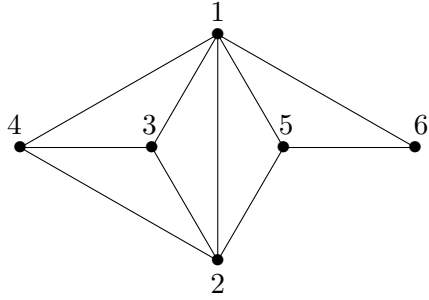
\begin{center}
\begin{tikzpicture}[scale=1.5]
\draw (0,0)--(0,2);
\draw (1.732,1)--(0,2);
\draw (-1.732,1)--(0,2);
\draw (-1.732,1)--(0,0);
\draw (0,2)--(0.577,1)--(0,0)--(-0.577,1)--(0,2);
\draw (0.577,1)--(1.732,1);
\draw (-0.577,1)--(-1.732,1);
\node (1) at  (0,-0.2) {$2$};\node (1) at  (0,0) {$\bullet$};
\node (2) at  (1.782,1.2) {$6$};\node (2) at  (1.732,1) {$\bullet$};
\node (3) at  (0,2.2) {$1$};\node (3) at  (0,2) {$\bullet$};
\node (4) at  (-1.782,1.2) {$4$};\node (4) at  (-1.732,1) {$\bullet$};
\node (123) at  (0.6,1.2) {$5$};\node (123) at  (0.577,1) {{$\bullet$}};
\node (134) at  (-0.6,1.2) {$3$};\node (134) at  (-0.577,1) {{$\bullet$}};
\end{tikzpicture}
\end{center}

Let us go into the details of the computation of  $\lambda_k(\Delta_1)$. 
First, it is known that $\lambda_1(\Delta_1)=0$ and $\lambda_2(\Delta_1)=h_2(G)=\frac25$ because we generally have \[\lambda_2^-(\Delta_1)=\lambda_2(\Delta_1)=\lambda_2^+(\Delta_1)=\hat{h}_2^-=\hat{h}_2=\hat{h}_2^+=h_2\]
due to the Cheeger equality $\lambda_2^-(\Delta_1)=h_2$ in \cite{HeinBuhler2010,Chang} and Theorem  \ref{thm:p-lap-C}. 
Here and after, $h_k(G)$ indicates the $k$-th multi-way Cheeger constant on the graph $G$ (see \eqref{eq:k-way-Cheeger} for the definition).

By Corollary \ref{cor:independence-number}, we have $\lambda_4(\Delta_1)=\lambda_5(\Delta_1)=\lambda_6(\Delta_1)=1$. It remains to compute  
$\lambda_3(\Delta_1)$. 

Let $P=\{V_1,V_2\}$ be a partition of $V$ with $V_1=\{2,3,4\}$ and $V_2=\{1,5,6\}$. Then $c(P)=4$, and $h_*(P)=\frac57$. Thus, by Lemma \ref{lemma:minmax-eigen-lower},  $\lambda_3(\Delta_1)\ge \lambda_3^-(\Delta_1)=\lambda_{6-c(P)+1}^-(\Delta_1)\ge h_*(P)=\frac57$. On the other hand, by the multi-way Cheeger inequality (see Claim \ref{claim:CSZ}),  $\lambda_3(\Delta_1)\le\lambda_3^+(\Delta_1)\le h_3(G)=\frac57$, which implies  $\lambda_3(\Delta_1)=\frac57$.

Therefore, we have determined all the min-max eigenvalues  of $\Delta_1$. 
Note that on the graph $G$ shown above, we have actually proved that $\lambda^-_k(\Delta_1)=\lambda_k(\Delta_1)=\lambda^+_k(\Delta_1)$, $k=1,2,3,4,5,6$. 

Now, we are able to prove our main result in this section. 
\begin{proof}[Proof of Theorem \ref{thm:main}]
Since we have written down the min-max eigenvalues of $\Delta_1$,  $\frac 59$ is not a min-max eigenvalue. The rest of the  
verification  is to  prove that $\frac 59$ is a  homological   eigenvalue of $\Delta_1$.

Note that $F_1(\vec1_{\{2,5,6\}})=\frac59$. 
In order to use Theorem \ref{thm:tri-link}, we  list all the values of $F_1$ acting on the vertices of $  \mathrm{link}(\vec 1_{\{2,5,6\}})$:
\begin{itemize}
\item 
$F_1(\vec 1_{\{2,5,6,3,4\}}-\vec 1_{\{1\}})=\mathtt{\frac12}$, $F_1(\vec 1_{\{2,5,6,3,1\}}-\vec 1_{\{4\}})=F_1(\vec 1_{\{2,5,6,4,1\}}-\vec 1_{\{3\}})=\mathtt{\frac{3}{10}}$,
\item $F_1(\vec 1_{\{2,5,6,1\}}-\vec 1_{\{3,4\}})=\mathtt{\frac25}$, $F_1(\vec 1_{\{2,5,6,3\}}-\vec 1_{\{1,4\}})=F_1(\vec 1_{\{2,5,6,4\}}-\vec 1_{\{1,3\}})=\frac35$, 

\item 
$F_1(\vec 1_{\{2,5,6,1\}}-\vec 1_{\{3\}})=F_1(\vec 1_{\{2,5,6,1\}}-\vec 1_{\{4\}})=\mathtt{\frac{7}{17}}$, $F_1(\vec 1_{\{2,5,6,3\}}-\vec 1_{\{4\}})=F_1(\vec 1_{\{2,5,6,4\}}-\vec 1_{\{3\}})=\frac35$, \item 
 $F_1(\vec 1_{\{2,5,6,3\}}-\vec 1_{\{1\}})=F_1(\vec 1_{\{2,5,6,4\}}-\vec 1_{\{1\}})=\frac{11}{17}$,
 \item $F_1(\vec 1_{\{2,5,6,3\}})=F_1(\vec 1_{\{2,5,6,4\}})=\mathtt{\frac12}$,  $F_1(\vec 1_{\{2,5,6,1\}})=\mathtt{\frac27}$, $F_1(\vec 1_{\{2,5,6,3,4\}})=\mathtt{\frac13}$,
 \item $F_1(\vec 1_{\{2,5,6,3,1\}})=F_1(\vec 1_{\{2,5,6,4,1\}})=\mathtt{\frac{3}{17}}$, $F_1(\vec 1_{\{2,5,6,3,4,1\}})=\mathtt{0}$, 
 \item $F_1(\vec 1_{\{2,5,6\}}-\vec 1_{\{3,1\}})=F_1(\vec 1_{\{2,5,6\}}-\vec 1_{\{4,1\}})=\frac{11}{17}$, $F_1(\vec 1_{\{2,5,6\}}-\vec 1_{\{3,4\}})=\frac{3}{5}$, 
 \item $F_1(\vec 1_{\{2,5,6\}}-\vec 1_{\{3\}})=F_1(\vec 1_{\{2,5,6\}}-\vec 1_{\{4\}})=\frac{2}{3}$, $F_1(\vec 1_{\{2,5,6\}}-\vec 1_{\{1\}})=\frac{5}{7}$, $F_1(\vec 1_{\{2,5,6\}}-\vec 1_{\{1,3,4\}})=\mathtt{\frac{1}{2}}$, 
 \item $F_1(\vec 1_{\{2\}})=F_1(\vec 1_{\{5\}})=F_1(\vec 1_{\{6\}})=F_1(\vec 1_{\{2,6\}})=1$, $F_1(\vec 1_{\{2,5\}})=\frac57$, $F_1(\vec 1_{\{5,6\}})=\frac35$. 
\end{itemize}

Thus, all the vertices in $ \{\vec x\in \mathrm{link}(\vec 1_{\{2,5,6\}}):F_1(\vec x)< F_1(\vec 1_{\{2,5,6\}})\}$ are:
\begin{itemize}
\item[(S1)] $\vec1_{\{2,5,6\}}-\vec1_{\{1,3,4\}}$
\item[(S2)] $\vec 1_{\{2,5,6,3\}}$, $\vec 1_{\{2,5,6,4\}}$, $\vec 1_{\{2,5,6,1\}}$, $\vec 1_{\{2,5,6,3,4\}}$, $\vec 1_{\{2,5,6,3,1\}}$, $\vec 1_{\{2,5,6,4,1\}}$, $\vec 1_{\{2,5,6,4,1,3\}}$, 
$\vec 1_{\{2,5,6,3,4\}}-\vec 1_{\{1\}}$, $\vec 1_{\{2,5,6,3,1\}}-\vec 1_{\{4\}}$, $\vec 1_{\{2,5,6,4,1\}}-\vec 1_{\{3\}}$, $\vec 1_{\{2,5,6,1\}}-\vec 1_{\{3,4\}}$,  $\vec 1_{\{2,5,6,1\}}-\vec 1_{\{3\}}$, $\vec 1_{\{2,5,6,1\}}-\vec 1_{\{4\}}$
\end{itemize}

It is easy to see that these vertices induce a disconnected simplicial complex in $K_n$. In  fact,  
 it is clear that the subcomplex has two connnected components, one is the singleton in (S1), and the other is the subcomplex induced by the vertices listed in (S2). 
 
Therefore,  $\frac 59$ is a  homological   eigenvalue of $\Delta_1$. Note that $\lambda_2(\Delta_1)=\frac25<\frac 59<\lambda_3(\Delta_1)=\frac57$, and  for any $0<\epsilon<\frac{1}{10}$ and  $0<\epsilon'<\frac{1}{20}$,  $(\frac59-\epsilon,\frac59+\epsilon)$ doesn't intersect  with $(\frac25-\epsilon',\frac25+\epsilon')\cup  (\frac57-\epsilon',\frac57+\epsilon')$. 
Then, by Proposition \ref{pro:min-max-eigen} and Theorem \ref{thm:homotopy-eigen},  there exists $0<\delta<1$ such that for any $p\in [1,1+\delta)$, 
$\lambda_2(\Delta_p)\in(\frac25-\epsilon',\frac25+\epsilon')$, $\lambda_3(\Delta_p)\in(\frac57-\epsilon',\frac57+\epsilon')$, and 
there is another   $\Delta_p$-eigenvalue $\lambda(\Delta_p)\in (\frac59-\epsilon,\frac59+\epsilon)$. In consequence,  $\lambda_2(\Delta_p)<\lambda(\Delta_p)<\lambda_3(\Delta_p)$, which implies that such an  eigenvalue $\lambda(\Delta_p)$ is not a (variational) min-max  eigenvalue of  
$\Delta_p$ on the graph $G$, $\forall p\in (1,1+\delta) $.  The proof is completed.
 \end{proof}

\begin{remark}
The graph  presented  in Theorem \ref{thm:main} (or  Section \ref{sec:verification}) was first studied  in a  previous work of the author \cite{CSZ17}. 
But in that paper, we only get partial results  on the 1-Laplacian eigenvalues (for example, we didn't even know the value of $\lambda_3(\Delta_1)$ for the graph in that paper). 
\end{remark}

\begin{remark}
Note that 
$\lambda_3(\Delta_1)$ equals the third Cheeger constant $h_3$. However, from the multi-way Cheeger inequality, we can only get $\lambda_3(\Delta_1)\le h_3$. In fact, it can be verified that every eigenvector of  $\lambda_3(\Delta_1)$ has at most two nodal domains, and thus  both the nodal domain theorem and the multi-way Cheeger inequality in \cite{CSZ17} are not sharp on 
this example. They are not powerful enough to obtain the results that we desired. 
To this end, we establish 
Lemma  \ref{lemma:minmax-eigen-lower} which is  special but strong.  
\end{remark}
The picture in Fig.~\ref{fig} shows a sketch of the proof for 
Theorem \ref{thm:main}. In fact, we can determine all the eigenvalues and all the min-max eigenvalues of $\Delta_1$ by some auxiliary results established in Section \ref{sec:main-proof}, and then we apply 
Theorem \ref{thm:tri-link} to select the eigenvalue $\frac59$ which is homological but not in the min-max form. Note that the 
 min-max eigenvalues of $\Delta_p$  should be far away from $\frac59$  when $p$ is sufficiently close to $1$, due to Proposition \ref{pro:min-max-eigen}. 
Then, Theorem \ref{thm:homotopy-eigen} is applicable to prove the existence of non-minmax homological eigenvalues of $\Delta_p$ near $\frac59$, when $p$ is sufficiently close to $1$.

To some extent, the picture suggests 
us to consider the  `persistent $p$-Laplacian' for varying $p$, which is able to record the birth (appearance) and death (disappearance) of the spectra  
of $\Delta_p$ when  $p$ goes from $1$ to $\infty$, and thus extra combinatorial  information is embedded in it.

\begin{figure}[!htp]
    \centering
\begin{tikzpicture}[scale=8]
\draw[thin] (0,1.544)--(0,0)--(1,0)--(1,1);
\node (26) at  (0,1.544) {$\bullet$};
\node (25) at (0,1.408) {$\bullet$};
\node (24) at (0,1.333) {$\bullet$};
\node (23) at (0,1.1225) {$\bullet$};
\node (22) at (0,0.59175) {$\bullet$};
\node (21) at (0,0) {$\bullet$};
\node (19) at (1,1) { $\bullet$};
\node (18) at (1,7/9) {\color{blue} \tiny $\bullet$};
\node (17) at (1,3/4) {\color{blue} \tiny $\bullet$};
\node (16) at (1,5/7) {$\bullet$};
\node (15) at (1,2/3) {\color{blue} \tiny $\bullet$};
\node (14) at (1,3/5) {\color{blue} \tiny $\bullet$};
\node (13) at (1,5/9) {\color{red} $\bullet$};
\node (12) at (1,2/5) { $\bullet$};
\node (11) at (1,0) { $\bullet$};
\node (0) at (0.9,0.56) {\color{blue} \tiny $\bullet$};
\draw[thick] (0,1.544) to[out=-60,in=180] (1,1);
 \draw[thick] (0,1.408) to[out=-60,in=190] (1,1); \draw[thick] (0,1.333) to[out=-60,in=240] (1,1);
 \draw[thick] (0,1.1225) to[out=-60,in=180] (1,5/7);
  \draw[red] (0.9,0.56) to[out=-60,in=120] (1,5/9);
 \draw[thick] (0,0.59175) to[out=-60,in=180] (1,2/5);
 \draw[thick] (0,0) to[out=0,in=180] (1,0);
  \node (2-6) at (-0.1,1.544) {$\frac{20+\sqrt{10}}{15}$}; 
  \node (2-5) at (-0.09,1.408) {$\frac{6+\sqrt{6}}{6}$};
   \node (2-4) at (-0.05,1.333) {$\frac43$};
    \node (2-3) at (-0.1,1.1225) {$\frac{20-\sqrt{10}}{15}$}; 
  \node (2-2) at (-0.09,0.59175) {$\frac{6-\sqrt{6}}{6}$};
 \node (1-9) at (1.01,1.03) {$1$}; 
  \node (1-8) at (1.02,7/9+0.01) {\small $\frac79$}; 
\node (1-7) at (0.98,3/4) {\small $\frac34$};
\node (1-6) at (1.02,5/7) {$\frac57$};
\node (1-5) at (0.97,2/3) {\small$\frac23$}; 
\node (1-4) at (0.98,3/5+0.01) {\small$\frac35$};
\node (1-3) at (1.02,5/9) {$\frac59$};
\node (1-2) at (1.02,2/5) {$\frac25$};
\node (1-1) at (1.03,0) {$0$};
\node (2-1) at (-0.03,0) {$0$};
\node (2p) at (-0.1,0.06) {$p=2$};
\node (1p) at (1.1,0.06) {$p=1$};
\node (21p) at (0.5,0.06) {$2>p>1$};
\draw[thick] (0.48,0.01) -- (0.5,0)--(0.48,-0.01);
\end{tikzpicture}
\caption{\label{fig}\small This picture illustrates the variance of the eigenvalues of $\Delta_p$ ($2\ge p\ge 1$) for the graph presented in Theorem \ref{thm:main}. The left six numbers are exactly the eigenvalues of $\Delta_2$, while the right nine numbers are exactly the eigenvalues of $\Delta_1$. All the black points are the variational (min-max) eigenvalues, while the red one is the homological eigenvalue which is non-variational. 
The blue points are also 
non-variational eigenvalues of $\Delta_1$, 
but we haven't checked 
whether they are homological eigenvalues. }
\end{figure}

\begin{remark}\label{rem:general-graph}
If we allow that the graph possesses  repeated edges  with different  real incidence coefficients, then we can find  
non-variational eigenvalues for $p$-Laplacians on certain graphs of order 2.  Precisely, let $V=\{1,2\}$ and  $E=\{e_1,e_2\}$,  consider a generalized graph $(V,E,\varphi)$ with  repeated edges and  vertex-edge incidence coefficients  $\varphi:V\times E\to\R$   defined by  $\varphi(1,e_1)=\varphi(2,e_1)=\varphi(1,e_2)=1$ and $\varphi(2,e_2)=-2$. Then, similar to the usual $p$-Laplacian eigenvalue problem on simple graphs, the   generalized    $p$-Laplacian eigenvalue problem is to find $(\lambda,\vec x)$ such that
 $$\begin{cases}
 |x_1+x_2|^{p-2}(x_1+x_2)+|x_1-2x_2|^{p-2}(x_1-2x_2)=2\lambda |x_1|^{p-2}x_1\\
 |x_1+x_2|^{p-2}(x_1+x_2)-2|x_1-2x_2|^{p-2}(x_1-2x_2)=3\lambda |x_2|^{p-2}x_2
 \end{cases}$$
 which includes the critical values and  critical points of 
 the corresponding Rayleigh  quotient $$F_p(\vec x)=\frac{|x_1+x_2|^p+|x_1-2x_2|^p}{2|x_1|^p+3|x_2|^p}.$$
Since $$F_p(1,0)=F_p(0,1)=F_p(-1,0)=F_p(0,-1)=1>F_p(1,-1)=\frac{3^p}{5}>F_p(2,1)=\frac{3^p}{2^{p+1}+3}$$
whenever $1\le p<\log_35$, $F_p$ has a local minimizer in the first quadrant  $\{\vec x\in\R^2:x_1>0,x_2>0\}$, and a local minimizer in  $\{\vec x\in\R^2:x_1>0,x_2<0\}$. Since $(1,-1)$ and $(2,1)$ are local minimizers of $F_1$, and  $F_1(1,-1)=\frac{3}{5}>F_1(2,1)=\frac{3}{7}$, $F_1$ has exactly three different critical values, and the minimum of $F_p$ restricted on $\{\vec x\in\R^2:x_1>0,x_2>0\}$ is smaller than the minimum of $F_p$ restricted on $\{\vec x\in\R^2:x_1>0,x_2<0\}$ if $p$ is sufficiently close to $1$. Therefore, it is easy to see that $F_p$ has at least three different critical values, and thus the $p$-Laplacian   has at least  three distinct  eigenvalues  when $1\le p<\log_35$.
\end{remark}

\subsection{Non-homological eigenvalues}

\begin{example}
Using the same technique in 
Section 
\ref{sec:verification}, we can show that $\vec1_{\{1,2\}}$ is an eigenvector of the  1-Laplacian on the following graph:
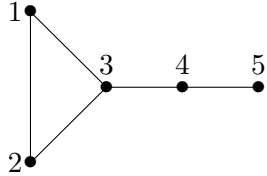
\begin{center}
\begin{tikzpicture}[scale=1]
\draw (0,0)--(-1,1)--(-1,-1)--(0,0)--(1,0)--(2,0);
\node (1) at  (-1.2,1) {$1$};
\node (2) at  (-1.2,-1) {$2$};
\node (3) at  (0,0.3) {$3$};
\node (4) at  (1,0.3) {$4$};
\node (5) at  (2,0.3) {$5$};
\node (1) at  (-1,1) {$\bullet$};
\node (2) at  (-1,-1) {$\bullet$};
\node (3) at  (0,0) {$\bullet$};
\node (4) at  (1,0) {$\bullet$};
\node (5) at  (2,0) {$\bullet$};
\end{tikzpicture}
\end{center}
It can be verified that $\frac12$ is an eigenvalue of the 
 1-Laplacian on the  graph above,  while by Theorem \ref{thm:tri-link}, it is not a homological eigenvalue.
\end{example}

A non-homological  eigenvalue does not possess the stability and local monotonicity, which means that Theorem  \ref{thm:most} cannot be improved. We give an example on path graphs.
\begin{example}\label{example:path6}
By Theorem 3.7 in \cite{DFT21}, for $p>1$,  the $p$-Laplacian on a tree admits only variational eigenvalues,  and by  the discussion in \cite{DFT21} and  Corollary \ref{cor:=}, 
the min-max eigenvalues of the  1-Laplacian on a tree  coincide exactly with the multi-way Cheeger constants. In particular, on the path  graph $P_6$ with six vertices,  we have $h_1=\lambda_1(\Delta_1)=0$, $h_2=\lambda_2(\Delta_1)=1/5$, $h_3=\lambda_3(\Delta_1)=1/2$, and $h_4=h_5=h_6=\lambda_4(\Delta_1)=\lambda_5(\Delta_1)=\lambda_6(\Delta_1)=1$, and by Theorem  \ref{thm:most}, for any $0<\epsilon<\frac{1}{100}$, there exists $\delta>0$ such that for any $1<p<1+\delta$, $\mathrm{spec}(\Delta_p)\subset [0,\epsilon)\cup (\frac15-\epsilon,\frac15+\epsilon)\cup (\frac12-\epsilon,\frac12+\epsilon)\cup (1-\epsilon,1]$. 

By Theorem 2 in \cite{CSZ17}, the set of 1-Laplacian eigenvalues of  the path  graph $P_6$  is 
$\{0,1/5,1/3,1/2,1\}$.  

Therefore, the 1-Laplacian eigenvalue $1/3$ is non-variational. Furthermore, we can show that the eigenvalue $1/3$ must also be  non-homological.  If fact, if $1/3$ is a homological eigenvalue of $\Delta_1$ on $P_6$, then by the stability of homological eigenvalues (Theorem  \ref{thm:most}), for any $p>1$ which is sufficiently close to $1$, $\Delta_p$ on $P_6$ has an eigenvalue that is sufficiently close to $1/3$, which is 
a contradiction to the discussions above. In consequence, $\frac13$ is a non-homological eigenvalue of $\Delta_1$ on $P_6$. Moreover,  from this example,  we see that in contrast to homological eigenvalues of $\Delta_p$, the stability and local monotonicity 
with respect to $p$, do not hold on  non-homological eigenvalues of $\Delta_p$.
\end{example}

\subsection{A note on complete graphs}
Based on the results in \cite{Amghibech} and \cite{CSZ17}, we have:
 \label{sec:complete-graph}
\begin{pro}
For $p\not\in\{1,2\}$, and $n\ge4$, the complete graph of order $n$ has exactly $\lfloor n/2\rfloor (n-\lfloor n/2\rfloor )+1-n$ non-variational eigenvalues of $p$-Laplacian.

While for any $n\ge1$, all the eigenvalues of 1-Laplacian  on the complete graph of order $n$ are variational eigenvalues.
\end{pro}
\begin{proof}
By Theorem 6 in \cite{Amghibech}, for $n\ge3$ and  $p>1$, the nonzero eigenvalues of $\Delta_p$ on the complete graph of order $n$ are: $$\frac{1}{n-1}\left(n-i-j+(i^{\frac{1}{p-1}}+j^{\frac{1}{p-1}})^{p-1}\right),\;i,j\in\mathbb{Z}_+,\,i+j\le n.$$
It can be checked that the number of the 
eigenvalues is  
$\lfloor n/2\rfloor (n-\lfloor n/2\rfloor )+1$, and every eigenvalue has $\gamma^-$-multiplicity 1 (due to the proof of Theorem 6 in \cite{Amghibech}), 
whenever $p\ne 1,2$. 

By Proposition 8 and Theorem 5 in \cite{CSZ17}, for $n\ge3$, the nonzero eigenvalues of $\Delta_1$ on the complete graph of order $n$ are
$$\frac{n-i}{n-1},\;\;i\in\mathbb{Z}_+,\,i\le \lfloor \frac n2\rfloor .$$
Moreover, the $\gamma^-$-multiplicity of the eigenvalue $\frac{n-i}{n-1}$ is $2$, whenever $1\le i<\lfloor \frac n2\rfloor $. The smallest positive eigenvalue (i.e., the Cheeger constant) $\frac{n-\lfloor \frac n2\rfloor }{n-1}$ has $\gamma^-$-multiplicity 
$$\begin{cases}1,&\text{ if }n\text{ is even},\\
2,&\text{ if }n\text{ is odd}.\end{cases}$$
Thus, there is no difficulty to verify that all the eigenvalues are variational eigenvalues, that is, they can be listed as  $\lambda_1^-(\Delta_1),\cdots,\lambda_n^-(\Delta_1)$.
\end{proof}

\subsection{
Further remarks on nonvariational  eigenvalues of the 1-Laplacian 
}
\label{section:1Lapl}
We have shown in 
Section \ref{sec:verification} that for some graphs, 
$$\left\{\text{the min-max }\Delta_p\text{-eigenvalues }\lambda_1(\Delta_p),\cdots,\lambda_n(\Delta_p)\right\}\subsetneqq \left\{\text{homological eigenvalues of }\Delta_p\right\}.
$$
In this section, we show that for $p=1$, the size of  the  difference set 
$$\{\Delta_1\text{-eigenvalues}\}\setminus\{\text{min-max }\Delta_1\text{-eigenvalues}\}$$
can be very large for some graphs. For convenience, given a simple graph $G$, we shall use $\Delta_1(G)$ to denote the 1-Laplacian on the graph $G$. 
\begin{itemize}
    \item \textbf{For sufficiently large integer $n$, there exists a  connected graph  $G_n$ on $n$ vertices with $O(n\ln n)$ eigenvalues ({counting multiplicity}) of $\Delta_1(G_n)$.}

\begin{example}
Let $G_n$ be the  $n$-order cycle  graph, i.e., $V(G_n)=\{1,\cdots,n\}$ and  $E(G_n)=\{\{1,2\},\{2,3\},\cdots,\{n-1,n\},\{n,1\}\}$.  According to Theorem 4 in  \cite{CSZ17}, 
we know that the  eigenvalues of $\Delta_1(G_n)$ are $1,\frac12,\cdots,\frac{1}{\lfloor\frac n2\rfloor},0$, and 
 it is not difficult to verify  that 
the  corresponding  multiplicities are  $\lfloor\frac n2\rfloor$, $\lfloor\frac n4]$, $\cdots$, $\lfloor\frac{n}{2\lfloor n/2\rfloor}\rfloor$, $1$, respectively. Here and after, we use $\lfloor \cdot\rfloor$ to denote the floor function. 
  
  Therefore,  the number of eigenvalues (counting multiplicity) of  $\Delta_1(G_n)$ is  $1+\sum_{i=1}^{\lfloor \frac{n}{2}\rfloor}\lfloor \frac{n}{2i}\rfloor=\frac n2(1+\frac12+\cdots+\frac{1}{\lfloor \frac{n}{2}\rfloor})+O(n)=\frac n2\ln(\frac n2)+O(n)$,  where we used the  logarithmic growth property of  harmonic series.
\end{example}

\item \textbf{For sufficiently large integer $n$, there exists a  connected  graph $G_n$ on $n$ vertices with at least  $\lfloor\frac {3n}{2}\rfloor-2$ {pairwise distinct} eigenvalues of $\Delta_1(G_n)$.}

\begin{example}\label{ex:more-1-Laplacian-eigenvalues}
 For any even number  $n\ge 8$, let $G_n$ be the graph on the vertex set $\{1,\cdots,n\}$, with the edge set 
     $$E(G_n)=\left\{\{i,j\}:i+j\le n+1,\,i,j\in \{1,\cdots,n\}\right\}\cup\left\{\{n,n-i\}:1\le i\le \frac n2-2\right\}.$$
  Then,  we have $\deg(i)=n-i$  for $1\le i\le \frac{n}{2}$, $\deg(\frac n2+1)=\frac n2$, $\deg(n)=\frac n2-1$ and $\deg(j)=n+2-j$, for $  \frac{n}{2}+2\le j\le n-1$.  See Fig.~\ref{fig:more-1-Laplacian-eigenvalues} for the cases  of  $n=8$ and $n=10$, respectively.
    
\begin{figure}[ht]
\centering
\begin{tikzpicture}
\node (1) at (0:3) {$\bullet$}; 
\node (2) at (45:3) {$\bullet$}; 
\node (3) at (90:3) {$\bullet$};
\node (4) at (135:3) {$\bullet$};
\node (5) at (180:3) {$\bullet$};
\node (6) at (225:3) {$\bullet$};
\node (7) at (270:3) {$\bullet$};
\node (8) at (315:3) {$\bullet$};
\draw (0:3)--(45:3);
\draw (0:3)--(90:3);
\draw (0:3)--(135:3);
\draw (0:3)--(180:3);
\draw (0:3)--(225:3);
\draw (0:3)--(270:3);
\draw (0:3)--(315:3);
\draw (45:3)--(90:3);
\draw (45:3)--(135:3);
\draw (45:3)--(180:3);
\draw (45:3)--(225:3);
\draw (45:3)--(270:3);
\draw (90:3)--(135:3);
\draw (90:3)--(180:3);
\draw (90:3)--(225:3);
\draw (135:3)--(180:3);
\draw (315:3)--(225:3);
\draw (315:3)--(270:3);
\node (1) at (0:3.3) {$1$}; 
\node (2) at (45:3.3) {$2$}; 
\node (3) at (90:3.3) {$3$};
\node (4) at (135:3.3) {$4$};
\node (5) at (180:3.3) {$5$};
\node (6) at (225:3.3) {$6$};
\node (7) at (270:3.3) {$7$};
\node (8) at (315:3.3) {$8$};
\end{tikzpicture}
~~~~~~
\begin{tikzpicture}
\node (1) at (0:3) {$\bullet$}; 
\node (2) at (36:3) {$\bullet$};
\node (3) at (72:3) {$\bullet$};
\node (4) at (108:3) {$\bullet$};
\node (5) at (144:3) {$\bullet$};
\node (6) at (180:3) {$\bullet$};
\node (7) at (216:3) {$\bullet$};
\node (8) at (252:3) {$\bullet$};
\node (9) at (288:3) {$\bullet$};
\node (10) at (324:3) {$\bullet$};
\draw (0:3)--(36:3);
\draw (0:3)--(72:3);
\draw (0:3)--(108:3);
\draw (0:3)--(144:3);
\draw (0:3)--(180:3);
\draw (0:3)--(216:3);
\draw (0:3)--(252:3);
\draw (0:3)--(288:3);
\draw (0:3)--(324:3);
\draw (36:3)--(72:3);
\draw (36:3)--(108:3);
\draw (36:3)--(144:3);
\draw (36:3)--(180:3);
\draw (36:3)--(216:3);
\draw (36:3)--(252:3);
\draw (36:3)--(288:3);
\draw (72:3)--(108:3);
\draw (72:3)--(144:3);
\draw (72:3)--(180:3);
\draw (72:3)--(216:3);
\draw (72:3)--(252:3);
\draw (108:3)--(144:3);
\draw (108:3)--(180:3);
\draw (108:3)--(216:3);
\draw (144:3)--(180:3);
\draw (324:3)--(216:3);
\draw (324:3)--(252:3);
\draw (324:3)--(288:3);
\node (1) at (0:3.3) {$1$};
\node (2) at (36:3.3) {$2$};
\node (3) at (72:3.3) {$3$};
\node (4) at (108:3.3) {$4$};
\node (5) at (144:3.3) {$5$};
\node (6) at (180:3.3) {$6$};
\node (7) at (216:3.3) {$7$};
\node (8) at (252:3.3) {$8$};
\node (9) at (288:3.3) {$9$};
\node (10) at (324:3.3) {$10$};
\end{tikzpicture}
\caption{The graphs in Example \ref{ex:more-1-Laplacian-eigenvalues}}
\label{fig:more-1-Laplacian-eigenvalues}
\end{figure}
      
      It is clear that  $\deg(i)\ge 3$, $\forall i\in V$. Thus, applying Proposition \ref{prop:graph-numner-eigen} to $G_n$,  the number of distinct eigenvalues of $\Delta_1(G_n)$ is at least 
      \begin{align*}
          &2+\#\{\deg(i)+\deg(j):\{i,j\}\in E(\Gamma_n)\}
          \\=~&2+\#\left(\{n,\cdots,2n-3\}\cup \{\frac n2+2,\cdots, n-1\}\right)=\frac32n-2.  
      \end{align*}
\end{example}


\end{itemize}

Finally, we present our technical results used in the examples above. 

\begin{pro}\label{prop:graph-numner-eigen}
For a simple graph $G=(V,E)$, the number of distinct eigenvalues of $\Delta_1(G)$ is larger than or equal to 
 $$2+\#\{\deg(i)+\deg(j):\{i,j\}\text{ is an edge s.t. the induced subgraph on  }V\setminus\{i,j\} \text{ has no isolated vertex}\}.$$
In particular, if the minimum degree of 
$G$ is larger than or equal to 
$3$, then the number of distinct eigenvalues of $\Delta_1(G)$ is larger than or equal to 
  $$2+\#\{\deg(i)+\deg(j):\{i,j\}\in E(G)\}.$$
\end{pro}
\begin{proof}
Note that $0$ and $1$ are the minimal and the maximal eigenvalues of $\Delta_1$, respectively. 
By Proposition  \ref{pro:eigenfunction-two-elements} (I), for an edge $\{i,j\}$,  the indicator $ \vec 1_{\{i,j\}}$ is an eigenvector of $\Delta_1$ if and only if the induced subgraph  $G|_{V\setminus\{i,j\}}$ has no isolated vertex;  and in this case, the eigenvector $ \vec 1_{\{i,j\}}$ corresponds to the eigenvalue \[\lambda=1-\frac{2}{\deg(i)+\deg(j)}.\]
This proves the first statement. 

For the second one, it follows from  $\deg(i)\ge 3$, $\forall i\in V$, that $G|_{V\setminus\{i,j\}}$  has no isolated vertex, for any edge $\{i,j\}$ in $G$. Hence, by the first statement, we complete the proof.
\end{proof}

\begin{defn}
For a subset $S\subset V$, the \textbf{$1$-neighborhood} of $S$ is  $\bigcup\limits_{\{i,j\}\in E:\{i,j\}\cap S\ne\emptyset}\{i,j\}$.
\end{defn}

\begin{defn}
A subset $S\subset V$ of order $2$ is  a \textbf{simple nodal set} if $S$ is an edge, and 
the subgraph  induced by $V\setminus S$ has no isolated vertex.  
\end{defn}

\begin{pro}\label{pro:eigenfunction-two-elements} 
\begin{enumerate}[(I)]
    
\item Assume that  $G$ is connected and $S$ is a  connected subset of order $2$. Then, $\vec1_S$ is an eigenvector of $\Delta_1$ if and only if $S$ is a simple  nodal set. 

\item If $S_1,\ldots,S_k$ are  simple  nodal sets  
with pairwise non-adjacent 1-neighborhoods, and $\frac{|\partial S_1|}{\vol(S_1)}=\ldots=\frac{|\partial S_k|}{\vol(S_k)}$, then every nonzero vector $\vec x\in\mathrm{span}(\vec 1_{S_1},\ldots,\vec 1_{S_k})$ is an eigenvector of $\Delta_1$.   

\item If $S$ and $S'$ are simple  nodal sets  
with disjoint 1-neighborhoods,  and $\frac{|\partial S|}{\vol(S)}=\frac{|\partial S'|}{\vol(S')}$, then $t\vec 1_S+t'\vec 1_{S'}$ is an eigenvector of $\Delta_1$, whenever $tt'\le 0$ and $(t,t')\ne(0,0)$.
\end{enumerate}  

\end{pro}

\begin{proof}
It suffices to  check the coordinate equation  \eqref{eq:brief-1-Lap-1} 
for the 1-Laplacian eigenvalue problem.

\begin{enumerate}
    \item[(I)] Without loss of generality, we may assume that $S=\{1,2\}$. Note that $|\partial S|=\vol(S)-2$
    and  $\lambda:=\frac{|\partial S|}{\vol(S)}=1-\frac{2}{\deg(2)+\deg(1)}$. 
Then $z_{1i}=1$ for $i\not\in S$ with $\{1,i\}\in E$, and  $z_{2j}=1$ for  $j\not\in S$ with $\{2,j\}\in E$. 

Assume that $\vec1_S$ is an eigenvector. If $i\in V\setminus S$ is isolated in the subgraph induced by  $V\setminus S$, then $\deg(i)\in\{1,2\}$ and every vertex adjacent to $i$ lies in $\{1,2\}$.  There are only two cases: $i$ connected only to $1$, and $i$ connected to both $1$ and $2$. If $i$ is adjacent to $1$ only, then by the definition,   the eigenequation for the eigenpair $(\lambda,\vec1_{\{1,2\}})$ satisfied in $i$ is $z_{i1}=-1\in\lambda\mathrm{Sgn}(0)$ which implies that $|\lambda|\ge 1$, but since it is known that $0\le \lambda\le 1$, there must hold $\lambda=1$. 
Similarly, if $i$ is adjacent to $1$ and $2$ only, then by the definition,   the eigenequation satisfied in $i$ is $z_{i1}+z_{i2}=-1-1\in2\lambda\,\mathrm{Sgn}(0)$, also implying $\lambda=1$. 
Hence,  
 in any case, we obtain $\lambda=1$, which contradicts to $\lambda=1-\frac{2}{\deg(1)+\deg(2)}<1$. In consequence, the subgraph induced by $V\setminus S$ has no isolated vertex. 

For  $i\in V\setminus S$, denote by $$c_{S,i}=\begin{cases}2,&\text{ if }\{1,i\},\{2,i\}\in E,\\
0,&\text{ if }\{1,i\},\{2,i\}\not\in E,\\
1,&\text{ otherwise}.
\end{cases}$$

We first prove that, if 
 for any $i\in V\setminus S$, \begin{equation}\label{eq:condition-i-S}
\frac{c_{S,i}}{\deg(i)}\le \frac{|\partial S|}{\vol(S)}:=\lambda,    
\end{equation}

then $\vec1_S$ is an eigenvector of $\Delta_1$. 

By letting 
$$z_{12}=\lambda\deg(1)-\deg(1)+1=\frac{\deg(2)-\deg(1)}{\deg(2)+\deg(1)}\in \mathrm{Sgn}(0)=[-1,1],$$ we have 
$$\sum_{i\sim 1}z_{1i}=\deg(1)-1+z_{12} =\lambda\deg(1).$$ Similarly, $\sum_{j\sim 2}z_{2j}=\lambda\deg(2)$. 

 Taking $z_{ij}=0$ for any  $\{i,j\}\in E$ with $i,j\not\in S$, we have 
$$\left|\sum_{j\in V:j\sim i}z_{ij}\right|\le c_{S,i} \le  \lambda\deg(i)$$
$\forall i\in V\setminus S$, where  $\lambda=1-\frac{2}{\vol(S)}$.  Thus,  
$\vec1_S$ is an eigenvector.

 
 Now, 
 we only need to check the remaining case that there exists $i\in V\setminus S$  satisfying $$c_{S,i}
 >\deg(i)\left(1-\frac{2}{\vol(S)}\right)$$
 that is, there is a vertex $i$ that doesn't satisfy the inequality  \eqref{eq:condition-i-S}. 
 \begin{enumerate}
     \item[{Case} 1.]  $c_{S,i}=1$.
     
     Without loss of generality, we assume that $\{1,i\}\in E$  and  $\{2,i\}\not\in E$. 
     
     In this case, $\deg(i)\ge 2$,  $\deg(1)+\deg(2)\ge 1+2=3$, and   $$1>\deg(i)\left(1-\frac{2}{\deg(1)+\deg(2)}\right).$$ These imply  $\deg(i)=2=\deg(1)$ and $\deg(2)=1$. 
      Then, by taking  $\lambda=\frac13$, $z_{1i}=1$, $z_{12}=-\frac13$, $z_{ij}=\frac13$ for $j\not\in \{1,2\}$ with $j\sim i$, and  $z_{jj'}=0$  otherwise, we get a solution of  the 1-Laplacian eigenvalue problem  \eqref{eq:brief-1-Lap-1},    
      which implies that $\vec1_S$ is an eigenvector corresponding to the eigenvalue $\frac13$. 
          \item[{Case} 2.]  $c_{S,i}=2$.

Without loss of generality, we  assume that $\deg(1)\ge \deg(2)$ for simplicity.

     In this case, we have
    $$
     2>\deg(i)\left(1-\frac{2}{\deg(1)+\deg(2)}\right),$$
      $\deg(i)\ge3$ and  $\deg(1)+\deg(2)\ge 2+2=4$. Hence, we get  $\deg(i)=3$, $\deg(2)=2$ and  $\deg(1)\in\{2,3\}$. 
       We denote by $j$ the unique vertex in  $V\setminus S$ that is  adjacent to $i$. 
     \begin{enumerate}
         \item[{Case} 2.1.]  $\deg(i)=3$, $\deg(2)=2=\deg(1)$.
         
         In this case,  we take $\lambda=\frac12$, $z_{12}=0$, $z_{1i}=z_{2i}=1$,  $z_{ij}=\frac12$, 
         and  $z_{j_1j_2}=0$ otherwise.    Then, \eqref{eq:brief-1-Lap-1} is easy to check. This means that $\vec1_S$ is an eigenvector corresponding to the eigenvalue $\frac12$. 
         
         \item[{Case} 2.2.]  $\deg(i)=3=\deg(1)$,  $\deg(2)=2$.
         
         In this case, let $\lambda=\frac35$,  $z_{21}=\frac15=z_{ij}$, $z_{1i}=z_{2i}=z_{1i'}=1$, where $i'$ is the unique vertex in $V\setminus\{1,2,i\}$ that is adjacent to $1$,   and  $z_{j_1j_2}=0$ otherwise. Substituting these parameters in  \eqref{eq:brief-1-Lap-1} implies that $\vec1_S$ is an eigenvector corresponding to the eigenvalue $\frac35$.         
     \end{enumerate}
 \end{enumerate}
 
 \item[(II)-(III)] These two statements can be verified  by solving \eqref{eq:brief-1-Lap-1} with the help of  (I). 
\end{enumerate}
The proof is completed.
\end{proof}

At the end of this section, we establish a result  
which shows  that the size of the difference  set $$\{\text{eigenvectors of }\Delta_1\}\setminus\{\text{critical points of }F_1\}$$ can be very large on some graphs. 
\begin{itemize}
    \item \textbf{The Hausdorff dimension of the eigenspace corresponding to  an eigenvalue $\lambda$ of $\Delta_1$ may be larger than the Hausdorff dimension of the set of  critical points corresponding to  the critical value $\lambda$ of $F_1$}
    
\begin{figure}[ht]
\begin{center}
\begin{tikzpicture}[auto]
\node (1) at (0,0) {$1$};
\node (2) at (1,0) {$2$};
\node (3) at (-1,0.5) {$3$};
\node (4) at (-1,-0.5) {$4$};
\node (5) at (2,0.5) {$5$};
\node (6) at (2,-0.5) {$6$};
\draw (1) to (2);
\draw (1) to (3);
\draw (1) to (4);
\draw (4) to (3);
\draw (6) to (2);
\draw (5) to (2);
\draw (6) to (5);
\draw (1) circle(0.23);
\draw (2) circle(0.23);
\draw (3) circle(0.23);
\draw (4) circle(0.23);
\draw (5) circle(0.23);
\draw (6) circle(0.23);
\end{tikzpicture}\caption{The graph for  Example \ref{example:1-Lap-eigenfunction-not-critical}.}\label{fig:1-Lap-eigenfunction-not-critical}
\end{center}
\end{figure}

\begin{example}\label{example:1-Lap-eigenfunction-not-critical}

 In the graph $G$ shown in Figure \ref{fig:1-Lap-eigenfunction-not-critical}, applying  Proposition \ref{pro:eigenfunction-two-elements} (III) directly to the subsets $\{3,4\}$ and $\{5,6\}$,  $t\vec1_{\{3,4\}}+s\vec1_{\{5,6\}}$ is an  eigenvector w.r.t. the eigenvalue $\frac12$ of  $\Delta_1$  for any $t,s\in\R$ with $ts\le 0$. 
 In fact, the eigenspace of $\frac12$ is  $S_{\frac12}(\Delta_1)=\{t\vec1_{\{3,4\}}+s\vec1_{\{5,6\}}:ts\le 0\}$. 
 
 We will show that    $t\vec1_{\{3,4\}}+s\vec1_{\{5,6\}}$ is a critical point of $F_1$ if and only if $t=0$ or $s=0$. 
 In fact, if $t>0>s$,  by taking $\vec x=t\vec1_{\{3,4\}}+s\vec1_{\{5,6\}}$ and $\vec y=\vec1_{\{1\}}-\vec1_{\{2\}}$ in  Proposition \ref{pro:equivalent-critical}, we can prove that  $\Phi(\vec x,\xi,\vec y)- \frac12 \Psi(\vec x,\xi,\vec y)\le -2$ for any $\xi$, where $\Phi$ and $\Psi$ are  introduced in  Proposition \ref{pro:equivalent-critical}. To see this, note that  
\begin{equation*}
   \Phi(\vec x,\xi,\vec y)=-4+2z_{12}'(\xi)\text{ and }\Psi(\vec x,\xi,\vec y)=3(z_1'(\xi)-z_2''(\xi)),
\end{equation*}
where $z_{12}'(\xi)=\begin{cases}1,&\text{ if }\xi(1)\ge \xi(2),\\
-1,&\text{ if }\xi(1)< \xi(2)\end{cases}$, $z_1'(\xi)=\begin{cases}1,&\text{ if }\xi(1)\ge 0,\\
-1,&\text{ if }\xi(1)< 0\end{cases}$ and $z_1''(\xi)=\begin{cases}1,&\text{ if }\xi(1)> 0,\\
-1,&\text{ if }\xi(1)\le 0.\end{cases}$ 
Now, it is easy to see $\Phi(\vec x,\xi,\vec y)- \frac12 \Psi(\vec x,\xi,\vec y)\le -2$,  $\forall\xi\in \R^n$. By Proposition \ref{pro:equivalent-critical}, we obtain   the desired conclusion. 
Similarly, for $t<0<s$ we take $\vec y=-\vec 1_{\{1\}}+\vec 1_{\{2\}}$;  for $t,s>0$, take $\vec y=\vec 1_{\{1\}}+\vec 1_{\{2\}}$; and for  $t,s<0$, take $\vec y=-\vec 1_{\{1\}}-\vec 1_{\{2\}}$. In summary, for $ts\ne 0$,   $t\vec 1_{\{3,4\}}+s\vec 1_{\{5,6\}}$ is not a critical point of $F_1$.   In contrast to the linear combination of $\vec 1_{\{3,4\}}$ and $\vec 1_{\{5,6\}}$, it is surprising  that both  $\vec 1_{\{3,4\}}$ and $\vec 1_{\{5,6\}}$ are critical points of $F_1$ (also  by Proposition \ref{pro:equivalent-critical}).
\end{example}

\end{itemize}

\begin{pro}\label{pro:equivalent-critical}
    A nonzero vector 
$\vec x$ is a critical point of $F_1$ corresponding  to  the  critical value $\lambda$ if and only if 
for any 
vector  $\vec y\in \R^n$, there exists $\xi\in \R^n$ such that
$$\Phi(\vec x,\xi,\vec y)\ge \lambda \Psi(\vec x,\xi,\vec y),$$
where 
$$\Phi(\vec x,\xi,\vec y)=\sum_{\substack{\{i,j\}\in E\\x_i\ne x_j}} z_{ij} (y_i-y_j)+\sum_{\substack{\{i,j\}\in E\\x_i=x_j\\ \xi_i\ne\xi_j}} z_{ij}(\xi) (y_i-y_j)+\sum_{\substack{\{i,j\}\in E\\x_i=x_j\\ \xi_i=\xi_j}} |y_i-y_j|,$$
$$\Psi(\vec x,\xi,\vec y)=\sum_{\substack{i\in V\\x_i\ne0}}\deg(i)z_iy_i+
\sum_{\substack{i\in V\\x_i=0\\ \xi_i\ne0}}\deg(i)z_i(\xi)y_i+\sum_{\substack{i\in V\\x_i=0\\ \xi_i=0}}\deg(i)|y_i|,$$
$$ z_{ij}:=z_{ij}(\vec x)= \mathrm{sign}(x_i-x_j),\; z_i:=z_i(\vec x)= \mathrm{sign}(x_i),\; z_{ij}(\xi)= \mathrm{sign}(\xi_i-\xi_j)\;\text{ and } z_i(\xi)=\mathrm{sign}(\xi_i).$$
\end{pro}

\begin{proof}
By the definition of Clarke derivative, $\vec x$ is a critical point of $F_1$  
if and only if, for any $\vec y\in \R^n$,
\begin{equation*}
\limsup\limits_{\xi\to0,t\to 0^+}
\frac{1}{t}\left(\frac{TV(\vec x+\xi+t \vec y)}{\|\vec x+\xi+t \vec y\|_1}-\frac{TV(\vec x+\xi)}{\|\vec x+\xi\|_1}\right)\ge0
\end{equation*}
which is equivalent to
\begin{equation}\label{eq:RQ1-critical}
\limsup\limits_{\xi\to0,t\to 0^+}
\frac{ TV(\vec x+\xi+t \vec y)\|\vec x+\xi\|_1-TV(\vec x+\xi)\|\vec x+\xi+t \vec y\|_1}{t}\ge0,
\end{equation}
where $TV(\vec x):=\sum_{\{j,i\}\in E}|x_i-x_j| $ and $\|\vec x\|_1:= \sum_{i\in V}\deg(i)|x_i|$. 
Since both $TV(\cdot)$ and $\|\cdot\|_1$ are piecewise-linear and one-homogeneous, for  sufficiently small $t>0$, we have 
\begin{align*}
TV(\vec x+\xi+t \vec y)=TV(\vec x+\xi)+t\sum_{\substack{\{i,j\}\in E\\x_i+\xi_i\ne x_j+\xi_j}}z_{ij}(\vec x+\xi)(y_i-y_j)
+t\sum_{\substack{\{i,j\}\in E\\x_i+\xi_i= x_j+\xi_j}} |y_i-y_j|,
\end{align*}
and
$$\|\vec x+\xi+t \vec y\|_1=\|\vec x+\xi\|_1+t\sum_{\substack{i\in V\\x_i+\xi_i\ne0}}\deg(i)z_i(\vec x+\xi)y_i+t\sum_{\substack{i\in V\\x_i+\xi_i=0}}\deg(i)|y_i|.$$
Therefore, \eqref{eq:RQ1-critical} can be rewritten as
$$\limsup\limits_{\xi\to0} \Phi(\xi)-\lambda(\xi) \Psi(\xi)\ge 0$$
where $\lambda(\xi):=\frac{TV(\vec x+\xi)}{\|\vec x+\xi\|_1}$,
$$\Phi(\xi):=\sum_{\substack{\{i,j\}\in E\\x_i+\xi_i\ne x_j+\xi_j}}z_{ij}(\vec x+\xi)(y_i-y_j)
+\sum_{\substack{\{i,j\}\in E\\x_i+\xi_i= x_j+\xi_j}} |y_i-y_j|,$$
$$
\Psi(\xi):=\sum_{\substack{i\in V\\x_i+\xi_i\ne0}}\deg(i)z_i(\vec x+\xi)y_i+\sum_{\substack{i\in V\\x_i+\xi_i=0}}\deg(i)|y_i|.$$
Note that, for $\xi$ sufficiently close to  $\vec0$, $\Phi(\xi)=\Phi(\vec x,\xi,\vec y)$ and $\Psi(\xi)=\Psi(\vec x,\xi,\vec y)$. Since  $\Phi(\cdot)$ and $\Psi(\cdot)$ are zero-homogeneous, \eqref{eq:RQ1-critical} is further  equivalent to 
$$\max\limits_{\xi\in \R^n}\Phi(\vec x,\xi,\vec y)-\lambda \Psi(\vec x,\xi,\vec y)\ge 0,$$
with $\lambda:=\frac{TV(\vec x)}{\|\vec x\|_1}$. 
This shows that  $\vec x$ is a critical point of $F_1$ if and only if 
$$\inf\limits_{\vec y\in \R^n}\max\limits_{\xi\in \R^n}\Phi(\vec x,\xi,\vec y)-\lambda \Psi(\vec x,\xi,\vec y)\ge 0.$$
The proof is completed.
\end{proof}

\section{Open problems}

We leave some open questions as 
possible  future works. 
Theorem  \ref{thm:p-lap-C} poses the  following question:
\begin{open}
How the nodal domains of the $p$-Laplacian variational eigenfunctions vary with respect to $p$?
\end{open}

Following Proposition   \ref{pro:usc-mult}, we are interested in the similar property for  $\gamma$ and $\gamma^+$.
\begin{open}
Are the $\gamma$-multiplicity and the  $\gamma^+$-multiplicity upper semi-continuous?   
?\end{open}

Proposition   \ref{pro:minmax-homological}   leads us to the following question for $\lambda^\pm_k(\Delta_p)$. 
\begin{open}
Are the variational eigenvalues $\lambda^-_k(\Delta_p)$ using the Krasnoselskii genus $\gamma^-$, and the min-max eigenvalue $\lambda^+_k(\Delta_p)$ involving $\gamma^+$, homological eigenvalues of $\Delta_p$?
\end{open}

Corollary \ref{cor:=} encourages us to post the following question: 
\begin{open}
Is the equality $\lambda^-_k(\Delta_p)=\lambda_k(\Delta_p)=\lambda^+_k(\Delta_p)$ true?
\end{open}

As a possible way to strengthen  Theorem \ref{thm:p-lap-C}, we ask:
\begin{open}
Is there a universal constant $C>0$ such that for any $p\ge 1$, and for any $k$, \[\lambda_k^-(\Delta_p)\ge \frac{h^p_k}{C k^{2p}}\;?
 \]  
\end{open}

\textbf{Acknowledgement}: The author would  like to thank Professor Kung-Ching Chang and Professor J\"urgen Jost for their encouragement  and support. 
The author would like to thank Professor Francesco Tudisco for many helpful discussions,  and for bringing the very  hard-to-find paper \cite{Amghibech} to my attention. 
 We would also like to thank the anonymous referee for the very valuable  comments. 
 This research was supported by grants from  Fundamental Research Funds for the Central Universities (No. 7101303046).

\renewcommand\thesection{\Alph{section}}
\setcounter{section}{0}
\section{Appendix}

\begin{lemma}\label{lem:elementary-inequality}
For any $t\ge 1$, $b,a\in\R$, 
$$\left||b|^{t}\mathrm{sign}(b)-|a|^{t}\mathrm{sign}(a)\right|
\ge|b-a|(\frac{|b|^{t}+|a|^{t}}{2})^{1-\frac 1t}$$

\end{lemma}
\begin{proof}
To prove this, we only need to show that for any $b>a>0$, 
$$b^t-a^{t}
\ge(b-a)(\frac{b^{t}+a^{t}}{2})^{1-\frac 1t}$$
$$b^t+a^{t}
\ge(b+a)(\frac{b^{t}+a^{t}}{2})^{1-\frac 1t}$$
The second one follows from the mean power inequality $(\frac{b^t+a^t}{2})^{\frac1t}\ge \frac{b+a}{2}$. The first one clearly holds when $t=1$. Now, suppose $t>1$. By the convexity of the function $x\mapsto x^{\frac{t}{t-1}}$ for $x\in(0,+\infty)$, we have 
$$\frac{b-a}{b}\left(\frac{b^t-a^t}{b-a}\right)^{\frac{t}{t-1}}+\frac{a}{b}(a^{t-1})^{\frac{t}{t-1}}\ge (b^{t-1})^{\frac{t}{t-1}}$$
which implies
$$(\frac{b^t-a^t}{b-a})^{\frac{t}{t-1}}\ge \frac{b^{t+1}-a^{t+1}}{b-a}$$
Since $b^{t+1}-a^{t+1}-(b-a)(b^t+a^t)=ab(b^{t-1}-a^{t-1})>0$, we have $$(\frac{b^t-a^t}{b-a})^{\frac{t}{t-1}}\ge \frac{b^{t+1}-a^{t+1}}{b-a}\ge b^t+a^t$$
and thus
$$b^t-a^{t}
\ge(b-a)(b^{t}+a^{t})^{1-\frac 1t}\ge(b-a)(\frac{b^{t}+a^{t}}{2})^{1-\frac 1t},$$
which proves the first inequality.  
\end{proof}

\begin{remark}
For any $t\ge 1$, $b,a\in\R$, 
$$\left||b|^{t}\mathrm{sign}(b)-|a|^{t}\mathrm{sign}(a)\right|
\le t|b-a|(\frac{|b|^{t}+|a|^{t}}{2})^{1-\frac 1t}.$$
This inequality was first  established in \cite{Amghibech}. 
\end{remark}

\end{document}